\documentclass[11pt,reqno]{amsart}

\usepackage[pagewise]{lineno}

\usepackage{latexsym, amssymb, amsmath, bbm}
\usepackage{bm}

\usepackage{amsmath}
\usepackage{cases}

\usepackage{amsfonts}
\usepackage[colorlinks,linkcolor=blue,citecolor=blue]{hyperref}

\usepackage[all]{xy}
\usepackage{pgfplots}
\usepackage{enumerate}

\DeclareSymbolFont{EulerExtension}{U}{euex}{m}{n}
\DeclareMathSymbol{\euintop}{\mathop} {EulerExtension}{"52}
\DeclareMathSymbol{\euointop}{\mathop} {EulerExtension}{"48}

\allowdisplaybreaks[4]

\def \N{\mathbb{N}}
\def \Z{\mathbb{Z}}

\def \k{\mathbbm{k}}

\def \dim{\operatorname{dim}}

\def \N{\mathbb{N}}

\def \Z{\mathbb{Z}}

\def \la{\lambda}
\def \om{\omega}
\def \Om{\Omega}
\def \ta{\theta}
\def \Zeta{\mathrm{Z}}
\def \Chi{\mathrm{X}}
\def \GKdim{\operatorname{GKdim}}

\def \gldim{\operatorname{gldim}}

\def \Ker{\operatorname{Ker}}

\numberwithin{equation}{section}

\newtheorem{theorem}{Theorem}[section]
\newtheorem{lemma}[theorem]{Lemma}
\newtheorem{proposition}[theorem]{Proposition}
\newtheorem{corollary}[theorem]{Corollary}
\newtheorem{definition}[theorem]{Definition}

\newtheorem{remark}[theorem]{Remark}
\newtheorem{question}[theorem]{Question}

\begin{document}

\title[\tiny{Finite duals of affine prime regular Hopf algebras of GK-dimension one}]{The finite duals of affine prime regular Hopf algebras of GK-dimension one}

\author[K. Li]{Kangqiao Li}
\address{School of Mathematics, Hangzhou Normal University, Hangzhou 311121, China}
\email{kqli@hznu.edu.cn}

\author[G. Liu]{Gongxiang Liu}
\address{Department of Mathematics, Nanjing University, Nanjing 210093, China}
\email{gxliu@nju.edu.cn}

\begin{abstract}
This paper is an attempt to construct a special kind of Hopf pairing $\langle-,-\rangle:H^\bullet\otimes H\rightarrow\k$. Specifically, $H^\bullet$ and $H$ should be both affine, noetherian and of the same GK-dimension. In addition, some properties of them would be dual to each other. We test the ideas in two steps for all the affine prime regular Hopf algebras $H$ of GK-dimension one: 1) We compute the finite duals $H^\circ$ of them, which are given by generators and relations; 2) the Hopf pairings desired are determined by choosing certain Hopf subalgebras $H^\bullet$ of $H^\circ$, where $\langle-,-\rangle$ becomes the evaluation.
\end{abstract}


\thanks{$^\dag$The first author was supported by Zhejiang Provincial Natural Science Foundation of China under No.LQ23A010003. The second author was supported by NSFC 12271243.}

\maketitle

\section{Introduction}
Throughout this paper, $\k$ is assumed to be an algebraically closed field of characteristic $0$, and all matrices, vector spaces, algebras and Hopf algebras are assumed to be over $\k$.

Roughly speaking, the concept of Hopf algebras is still valid under taking the linear duals. This elementary point of view suggests to us that structures of the dual are effective when we deal with certain problems on a Hopf algebra. The most basic definition of the dual for a Hopf algebra $H$ would be the \textit{finite dual} $H^\circ$ described by Heyneman and Sweedler \cite{HS69}, which is also a Hopf algebra over the same base field. Of course, there are a number of relationships between $H$ and $H^\circ$, and this motivates the authors to consider the structure of duals for varieties of Hopf algebras.

Of course it is known that when $H$ is a finite-dimensional Hopf algebra, its dual Hopf algebra $H^\ast$ has the same dimension as well as similar or dual properties with $H$. Indeed, the structures of $H$ are determined conversely by those of $H^\ast$. However for an infinite-dimensional Hopf algebra $H$, its finite dual $H^\circ$ might be much too large (or small) to be studied easily. This motivates us to consider some kinds of dual Hopf algebras $H^\bullet$ for certain (infinite-dimensional) Hopf algebras $H$, where the original evaluation is replaced by Hopf pairings $H^\bullet\otimes H\rightarrow\k$. Our primary goal is to formulate such Hopf pairings for noetherian affine Hopf algebras $H$ of finite GK-dimensions, such that:
\begin{itemize}
\item[(HP1)] The pairing is non-degenerate;
\item[(HP2)] \;$H^\bullet$ and $H$ are both affine and noetherian;
\item[(HP3)] \;$\GKdim(H^\bullet)=\GKdim(H)$.
\end{itemize}
Meanwhile, we wish that some properties of $H^\bullet$ and $H$ would be dual to each other (such as (HP4) in Section \ref{subsection:HPs}), which should be noted in different situations. Finally, $H^\bullet$ would be better to exist uniquely for any given $H$ under our assumptions. In this paper, we try to establish the desired Hopf pairings for affine prime regular Hopf algebras $H$ of GK-dimension one. For this purpose, the finite dual $H^\circ$ of each $H$ is determined, and afterwards $H^\bullet$ would be chosen as a suitable Hopf subalgebra of $H^\circ$.

However, it is not easy to determine the structure of $H^\circ$ for a Hopf algebra $H$ in general, especially for infinite-dimensional ones. Of course, the most direct way to get $H^{\circ}$ is by definition. For this, recall that $H^{\circ}$ is the Hopf algebra generated by $f\in H^{\ast}$ which vanish on an ideal $I\subseteq H$ of finite codimension. This means that we need a description of \emph{all} finite codimensional ideals which is impossible in general. To the authors' knowledge, there are two other ways to get $H^{\circ}$ if $H$ is good enough.  One is applying the well-known Cartier-Konstant-Milnor-Moore's Theorem (\cite{MM65}) if $H$ happens to be commutative. The related idea and method were generalized further (see \cite[Chapter 9]{Mon93}, \cite{CM94}). Another one is applying representation theoretical way if the representation category Rep-$H$ of finite dimensional modules happens to be very nice (see \cite{Tak92'}).

In the literature, significant progress has been made in classifying infinite dimensional noetherian Hopf algebras of low Gelfand-Kirillov dimensions (GK-dimension for short). The program of classifying Hopf algebras of GK-dimension one was initiated by Lu, Wu and Zhang \cite{LWZ07}. Then the second author found a new class of examples about prime regular Hopf algebras of GK-dimension one \cite{Liu09}. Brown and Zhang \cite{BZ10} made further efforts in this direction and classified all prime regular Hopf algebras of GK-dimension one under an extra hypothesis. In 2016, Wu, Ding and the second author \cite{WLD16} removed this hypothesis and gave a complete classification of prime regular Hopf algebras of GK-dimension one, in which some non-pointed Hopf algebras were found. Recently, the second author classified prime Hopf algebras of GK-dimension one satisfying two certain conditions \cite{Liu20}.
For Hopf algebras $H$ of GK-dimension two, Goodearl and Zhang \cite{GZ10} gave a classification of the case when $H$ is a domain satisfying $\mathrm{Ext}_H^1(\k,\k)\neq 0$. For those with vanishing Ext-groups, some interesting examples were constructed by Wang, Zhang and Zhuang \cite{WZZ13}.
The progress can also be found in a survey by Brown and Zhang \cite{BZ21}. As for Hopf algebras of GK-dimensions three and four, all connected ones were classified by Zhuang \cite{Zhu13} and Wang, Zhang and Zhuang \cite{WZZ15} respectively.

Our main interest in this paper is to deal with the affine prime Hopf algebras of GK-dimension one, whose detailed structures are recalled in Section \ref{subsection2.1}.
An interesting fact is that all affine prime regular Hopf algebras are commutative-by-finite \cite{BC21}, that is, a finite module over a normal commutative Hopf subalgebra. This suggests that we have a chance to get the finite duals of affine prime regular Hopf algebras of GK-dimension one explicitly. Actually, Ge and the second author use this observation together with some other techniques to determine the finite dual of the infinite dihedral group algebra $\k\mathbb{D}_\infty$ recently \cite{GL21}. In addition, after finishing this paper we realized that A. Jahn and M. Couto already considered this question along this line in their theses \cite{Jah15,Cou19}.

This paper is to develop the idea and techniques in \cite{GL21} further and as the main result we get the finite duals of all the affine prime regular Hopf algebras of GK-dimension one in terms of generators and relations. Now let us recall in our notations that there are five kinds of such Hopf algebras listed in \cite{WLD16}:
$$\k[x],\;\;\k[x^{\pm 1}],\;\;\k\mathbb{D}_\infty,\;\;T_\infty(n,v,\xi),\;\;
  B(n,\om,\gamma),\;\;D(m,d,\xi).$$
Since the finite duals of $\k[x],\k[x,x^{-1}]$ and $\k\mathbb{D}_\infty$ are known, it remains for us to determine the latter three kinds of Hopf algebras. Our main result is:
\begin{theorem}
 \emph{(1)} As a Hopf algebra, $T_\infty(n,v,\xi)^\circ$ is isomorphic to the Hopf algebra $T_{\infty^\circ}(n,v,\xi)$ constructed in Subsection \ref{subsection:T1},

\emph{(2)} As a Hopf algebra, $B(n,\om,\gamma)^\circ$ is isomorphic to the Hopf algebra $B_\circ(n,\om,\gamma)$ constructed in Subsection \ref{subsection:B1},

\emph{(3)} As a Hopf algebra, $D(m,d,\xi)^\circ$ is isomorphic to the Hopf algebra $D_\circ(m,d,\xi)$ constructed in Subsection \ref{subsection:D1}.
\end{theorem}
These three claims are presented respectively as Theorems \ref{thm:T}, \ref{thm:B}, and \ref{thm:D}. In brief, our computation for finite dual $H^\circ$ starts with properties of certain polynomial or Laurent polynomial subalgebra $P$, by which we figure cofinite left ideals of $H$. Then we construct some key elements in $H^\circ$ and determine their relations under Hopf operations. Finally it is proved that these elements generate $H^\circ$ with desired relations.

Using our descriptions of finite duals, we find that for an affine prime regular Hopf algebra $H$ one can always get a Hopf algebra $H^{\bullet}$ such that there is a non-degenerate Hopf pairing (see Majid \cite{Maj90} for the definition) between $H$ and $H^{\bullet}$. The definition can also be found as Definition \ref{def:Hopfpairing} in this paper. Some properties of $H^{\bullet}$ are listed in the following.
\begin{proposition}\label{Prop.1.2}
\emph{(1)} All Hopf algebras $(\k\mathbb{D}_\infty)^\bullet$, $T_\infty(n,v,\xi)^\bullet$, $B(n,\om,\gamma)^\bullet$ and $D(m,d,\xi)^\bullet$ have GK-dimension one and are minimal under inclusion relation;

\emph{(2)} As algebras, $(\k\mathbb{D}_\infty)^\bullet$ is regular while  $T_\infty(n,v,\xi)^\bullet$, $B(n,\om,\gamma)^\bullet$ and $D(m,d,\xi)^\bullet$ are not for $n,m\geq 2$.
\end{proposition}
Proposition \ref{Prop.1.2} is a direct consequence of Propositions \ref{prop.6.4} and \ref{prop:bulletproperty}. Through the pairing we constructed and this proposition,  we get two observations: 1) For each affine prime regular Hopf algebra $H$ of GK-dimension one, we get a quantum group in the sense of Takeuchi \cite{Tak92} naturally (see Remark \ref{rmk:quantumgrp} for details); 2) We give a counterexample to an infinite-dimensional analogous version of the semisimplicity result by Larson and Radford \cite{LR88} which states that for a finite-dimensional Hopf algebra $H$, $H$ is semisimple if and only if $H^{\ast}$ is so.

After we finished the present paper, we found the nice preprint \cite{BCJ} written by K.A. Brown, M. Couto and A. Jahn. They refined the results in \cite{Jah15,Cou19} and determined the structures of the finite dual of affine commutative-by-finite Hopf algebras. As the Hopf algebras of GK-dimension one considered in this paper are affine commutative-by-finite, the conclusions gotten in \cite{BCJ} can apply to our case directly. We remark that a number of formulas obtained in Section \ref{section3}, \ref{section4} and \ref{section5} in our paper were essentially the same as those in \cite[Section 7.5]{BCJ}. As pointed out in \cite{BCJ}, the results in \cite{BCJ} and the present paper are largely complementary.
Indeed, we can use our results to answer a question appeared in \cite[Remarks 7.1(3)]{BCJ} (see Remark \ref{rmk:cocyletrivial} in detail).

The organization of this paper is as follows: We gather in Section \ref{section2} necessary concepts and techniques, including the classification result for $H$, some equivalent definitions of $H^\circ$ for computation, and invertible matrices used for later proofs. Sections \ref{section3}, \ref{section4} and \ref{section5} are devoted to giving the structures for $T_\infty(n,v,\xi)^\circ$, $B(n,\om,\gamma)^\circ$ and $D(m,d,\xi)^\circ$ respectively. Finally in Section \ref{section6}, constructions and related properties of desired non-degenerate Hopf pairings on $\k\mathbb{D}_\infty$, $T_\infty(n,v,\xi)$, $B(n,\om,\gamma)$ and $D(m,d,\xi)$ are studied. We pose a question at the last section too.

\subsection{Notation}

We end this section by introducing some notations used in this paper. Let $\lfloor\frac{\bullet}{\bullet}\rfloor$ stand for the floor function, that is, for any natural numbers $m,n$, $\lfloor\frac{m}{n}\rfloor$ denotes the biggest integer which is not bigger than $\frac{m}{n}$.

Moreover, let $n$ be a positive integer. We usually denote condition $0\leq i\leq n-1$ simply by ``$i\in n$''. This can be referred to some convention in set theory that $n=\{0,1,\cdots,n-1\}$ as a set. Moreover, the cartesian product of $\{0,1,\cdots,m-1\}$ and $\{0,1,\cdots,n-1\}$ are denoted by $m\times n$, which is always ordered lexicographically.

Another notation which will be used is that: For an algebra $H$, we write $I\triangleleft_l H$ if $I$ is a left ideal of $H$, and denote the principal left ideal generated by $h\in H$ by $(h)$.

\section{Preliminaries}\label{section2}

In this section, we recall the classification result on affine prime regular Hopf algebras of GK-dimension one and give necessary tools and results for the determination of their finite duals. About general background knowledge, the reader is referred to
\cite{Mon93,Swe69} for Hopf algebras and \cite{LWZ07,BZ10} for exposition about noetherian Hopf algebras.
\subsection{The Classification Result}\label{subsection2.1}
At first, let us recall the classification of affine prime regular Hopf algebras of GK-dimension one given in \cite[Theorem 8.3]{WLD16}.

\begin{lemma}
Any prime regular Hopf algebra of GK-dimension one must be isomorphic to one of the following:
\begin{itemize}
  \item[(1)] Connected algebraic groups of dimension one: $\k[x]$ and $\k[x^{\pm 1}]$;
  \item[(2)] Infinite dihedral group algebra $\k\mathbb{D}_\infty$;
  \item[(3)] Infinite dimensional Taft algebras $T_\infty(n,v,\xi)$, where $n,v$ are integers satisfying $0\leq v\leq n-1$, and $\xi$ is a primitive $n$th root of $1$;
  \item[(4)] Generalized Liu's algebras $B(n,\om,\gamma)$, where $n$, $\om$ are positive integers and $\gamma$ is a primitive $n$th  root of $1$;
  \item[(5)] The Hopf algebras $D(m,d,\xi)$, where $m,d$ are positive integers satisfying $(1+m)d$ is even and $\xi$ is a primitive $2m$th root of $1$.
\end{itemize}
\end{lemma}

Detailed structures of $T_\infty(n,v,\xi)$, $B(n,\om,\gamma)$ and $D(m,d,\xi)$ are recalled as follows:

\begin{definition}\label{def:T}
Let $n$ be a positive integer, $0\leq v\leq n-1$, and $\xi$ be a primitive $n$th root of $1$. As an algebra, $T_\infty(n,v,\xi)$ is generated by $g$ and $x$ with relations
$$g^n=1,\;\;xg=\xi gx.$$
Then $T_\infty(n,v,\xi)$ becomes a Hopf algebra with comultiplication, counit and antipode given by
\begin{align*}
& \Delta(g)=g\otimes g,\;\;\Delta(x)=1\otimes x+x\otimes g^v,\;\;
  \varepsilon(g)=1,\;\;\varepsilon(x)=0,  \\
& S(g)=g^{n-1},\;\;S(x)=-\xi^{-v}g^{n-v}x.
\end{align*}
\end{definition}

\begin{definition}\label{def:B}
Let $n$ and $\om$ be positive integers, and $\gamma$ be a primitive $n$th root of $1$. As an algebra, $B(n,\om,\gamma)$ is generated by $x^{\pm1}$, $g$ and $y$ with relations
$$\left\{\begin{array}{l}
    xx^{-1}=x^{-1}x=1,\;\;xg=gx,\;\;xy=yx,  \\
    yg=\gamma gy,  \\
    y^n=1-x^\om=1-g^n.
\end{array}\right.$$
Then $B(n,\om,\gamma)$ becomes a Hopf algebra with comultiplication, counit and antipode given by
\begin{align*}
& \Delta(x)=x\otimes x,\;\;\Delta(g)=g\otimes g,\;\;\Delta(y)=1\otimes y+y\otimes g, \\
& \varepsilon(x)=\varepsilon(g)=1,\;\;\varepsilon(y)=0,  \\
& S(x)=x^{-1},\;\;S(g)=g^{-1},\;\;S(y)=-\gamma^{-1}g^{-1}y.
\end{align*}
\end{definition}

\begin{definition}\label{def:D}
Let $m,d$ be positive integers such that $(1+m)d$ is even and $\xi$ a primitive $2m$th root of unity. Define
$$\om:=md,\;\; \gamma:=\xi^2.$$
As an algebra, $D(m,d,\xi)$ is generated by $x^{\pm1}$, $g$, $y$ and $u_0,u_1,\cdots,u_{m-1}$ with relations
$$\begin{array}{c}
    xx^{-1}=x^{-1}x=1,\;\;gx=xg,\;\;yx=xy,  \\
    yg=\gamma gy,\;\;y^m=1-x^\om=1-g^m, \\
    u_ix=x^{-1}u_i,\;\;yu_i=\phi_iu_{i+1}=\xi x^du_iy,\;\;u_ig=\gamma^ix^{-2d}gu_i,  \\
u_iu_j=
\left\{\begin{array}{ll}
  (-1)^{-j}\xi^{-j}\gamma^\frac{j(j+1)}{2}\frac{1}{m}
    x^{-\frac{1+m}{2}d}\phi_i\phi_{i+1}\cdots\phi_{m-2-j}y^{i+j}g &(i+j\leq m-2)  \\
  (-1)^{-j}\xi^{-j}\gamma^\frac{j(j+1)}{2}\frac{1}{m}x^{-\frac{1+m}{2}d}y^{i+j}g &(i+j=m-1)  \\
  (-1)^{-j}\xi^{-j}\gamma^\frac{j(j+1)}{2}\frac{1}{m}
    x^{-\frac{1+m}{2}d}\phi_i\cdots\phi_{m-1}\phi_0\cdots\phi_{m-2-j}y^{i+j-m}g &(i+j\geq m)  \\
\end{array}
\right.\end{array}$$
where $\phi_i:=1-\gamma^{-i-1}x^d$ and $i,j\in m$.

Then $D(m,d,\xi)$ becomes a Hopf algebra with comultiplication, counit and the antipode given by
$$\begin{array}{c}
    \Delta(x)=x\otimes x,\;\;\Delta(g)=g\otimes g,\;\;\Delta(y)=y\otimes g+1\otimes y, \\
    \Delta(u_i)=\sum_{j=0}^{m-1}\gamma^{j(i-j)}u_j\otimes x^{-jd}g^ju_{i-j};  \\
    \varepsilon(x)=\varepsilon(g)=\varepsilon(u_0)=1,\;\;\varepsilon(y)=\varepsilon(u_l)=0, \\
    S(x)=x^{-1},\;\;S(g)=g^{-1},\;\;
    S(y)=-yg^{-1}=-\gamma^{-1}g^{-1}y,  \\
    S(u_i)=(-1)^i\xi^{-i}\gamma^{-\frac{i(i+1)}{2}}x^{id+\frac{3}{2}(1-m)d}g^{m-i-1}u_i,
\end{array}$$
for $i\in m$ and $1\leq l\leq m-1$.
\end{definition}

\subsection{Finite Dual for Hopf algebras}

In this subsection, we recall the concept of finite dual of a Hopf algebra, and list some lemmas in order to compute finite duals for the cases we are concerned with. The definition of finite dual is well-known and can be found in \cite[Chapter 6]{Swe69} or \cite[Chapter 9]{Mon93}, for example.

\begin{definition}
Let $H$ be a Hopf algebra over $\k$ and denote its dual space by $H^\ast$. The finite dual of $H$ is defined as
$$H^\circ:=\{f\in H^\ast\mid f~\text{vanishes on a cofinite ideal}~I\lhd H\}.$$
\end{definition}

It is well-known that $H^{\circ}$ has a Hopf algebra structure naturally \cite[Theorem 9.1.3]{Mon93}. We want to simplify $H^\circ$ a little. It is not hard to see that $H^\circ$ has the following alternative description: Let $\mathcal{I}$ be the family of all the cofinite left ideals of $H$. Then by \cite[Lemma 9.1.1]{Mon93}, we have $H^\circ=\sum_{I\in\mathcal{I}} I^\perp$, where $I^\perp=\{f\in H^\ast \mid f(I)=0\}.$ Therefore, we have

\begin{lemma}\label{lem:I0}
Let $\mathcal{I}_0$ be a subset of $\mathcal{I}$. Suppose for any $I\in\mathcal{I}$, there exist some $I_1,I_2,\cdots,I_m\in\mathcal{I}_0$ such that $\bigcap_{i=1}^m I_i\subseteq I$. Then
$H^\circ=\sum_{I\in\mathcal{I}_0} I^{\perp}$.
\end{lemma}

\begin{proof}
The condition implies that for any $I\in\mathcal{I}$, there exist some $I_1,I_2,\cdots,I_m\in\mathcal{I}_0$ such that $\sum_{i=1}^m I_i^\perp=\left(\bigcap_{i=1}^m I_i\right)^\perp\supseteq I^\perp$. Thus
$$H^\circ=\sum_{I\in\mathcal{I}}I^\perp\subseteq\sum_{I\in\mathcal{I}_0}I^\perp\subseteq H^\circ,$$
and the claim is then verified.
\end{proof}

Now we turn to a common feature of Hopf algebras in our cases. Specifically:
\begin{itemize}
\item For the remaining of this subsection, we assume that $H$ is a Hopf algebra which is a finitely generated right module over some subalgebra $P$, such that $P$ is a polynomial algebra $\k[x]$ or a Laurent polynomial algebra $\k[x^{\pm1}]$.
\end{itemize}
We would note in subsequent sections that this assumption holds for all affine prime regular Hopf algebras of GK-dimension one. This assumption helps us to figure all the cofinite left ideals of $H$:

\begin{lemma}\label{lem:cofin}
For any left ideal $I\lhd_l H$, $I$ is cofinite in $H$ if and only if $I\cap P\neq 0$.
\end{lemma}

\begin{proof}
Note that whenever $P$ is $\k[x]$ or $\k[x^{\pm1}]$, $I\cap P$ is an ideal of $P$ and thus a principal ideal generated by some polynomial $q(x)\in P$. It follows that
$$I\cap P\neq 0\;\;\Longleftrightarrow\;\;q(x)\neq0\;\;
  \Longleftrightarrow\;\;I\cap P~\text{is cofinite in}~P.$$

If $I$ is cofinite in $H$, then the following relations of $\k$-vector spaces
$$H/I\supseteq (P+I)/I \cong P/(I\cap P)$$
implies that $P/(I\cap P)$ must be finite-dimensional. In other words, $I\cap P\neq 0$ holds.

Conversely if $I\cap P\neq 0$, then clearly $P/(I\cap P)$ is finite-dimensional, and so is $(P+I)/I$.
However, the assumption above this lemma implies that $H=\sum_{i=1}^m h_iP$ for some $h_1, h_2,\cdots,h_m\in H$. As $I$ is a left ideal of $H$, we can know that
$(P+I)/I$ is a finite-dimensional subspace of the quotient left $H$-module $H/I$, and
$$H/I=\sum_{i=1}^m h_i(P+I/I)$$
holds. It follows that $H/I$ is also finite-dimensional.
\end{proof}

Combining Lemmas \ref{lem:I0} and \ref{lem:cofin}, in order to compute $H^\circ$, it is sufficient to find enough appropriate polynomials $p(x)$'s, such that each $I\in\mathcal{I}$ must contain some principal left ideal $(p(x))$. Concerning  our cases  of $T_\infty(n,v,\xi)$, $B(n,\om,\gamma)$ and $D(m,d,\xi)$, we can choose such $p(x)$'s respectively as follows.

\begin{lemma}\label{lem:p(x)}
Let $H$ be an example from any of the above three classes of Hopf algebras.
Let $I\in\mathcal{I}$, and $m,n$ be positive integers.
\begin{itemize}
  \item[(1)] There is a polynomial $p(x)\in P=\k[x]$ with form:
    \begin{equation}\label{eqn:p(x)1}
    p(x)=\prod_{\alpha=1}^N (x^m-\lambda_\alpha)^{nr_\alpha}
    \end{equation}
    for some $N\geq 1$, $r_\alpha\in\N$ and distinct $\lambda_\alpha\in\k$, such that $I$ contains the cofinite left ideal generated by $p(x)$ in $H$;
  \item[(2)] There is a polynomial $p(x)\in P=\k[x^{\pm 1}]$ with form:
    \begin{equation}\label{eqn:p(x)2}
    p(x)=\prod_{\alpha=1}^N (x^n-\lambda_\alpha)^{r_\alpha}
    \end{equation}
    for some $N\geq 1$, $r_\alpha\in\N$ and distinct $\lambda_\alpha\in\k^\ast$, such that $I$ contains the cofinite left ideal generated by $p(x)$ in $H$;
  \item[(3)] There is a polynomial $p(x)\in P=\k[x^{\pm 1}]$ with form:
    \begin{equation}\label{eqn:p(x)3}
    p(x)=(x^m-1)^{r'}(x^m+1)^{r''}
           \prod_{\alpha=1}^N (x^m-\lambda_\alpha)^{r_\alpha}(x^m-\lambda_\alpha^{-1})^{r_\alpha}
    \end{equation}
    for some $N\geq 1$, $r',r'',r_\alpha\in\N$, and $\lambda_\alpha\in\k^\ast\setminus\{\pm1\}$ which are distinct and not inverses of each other, such that $I$ contains the cofinite left ideal generated by $p(x)$ in $H$.
\end{itemize}
\end{lemma}

\begin{proof}
We prove (2) for example, since (1) and (3) hold according to the same reason.
 It can be known from the proof of Lemma \ref{lem:cofin} that $I\cap P$ is a principal ideal of $P$ generated by some non-zero polynomial $q(x)$. Without loss of generality, assume that
      $$q(x)=(x-\mu_1)(x-\mu_2)\cdots(x-\mu_M)$$
      for some $M\geq1$ and $\mu_1,\mu_2,\cdots,\mu_M\in\k^\ast$. Now define a polynomial
      $$p(x)=(x^n-\mu_1^n)(x^n-\mu_2^n)\cdots(x^n-\mu_M^n),$$
      which gives the desired form (\ref{eqn:p(x)2}). Since $q(x)\mid p(x)$, $I$ contains the left ideal of $H$ generated by $p(x)$ in $H$.
\end{proof}

This lemma suggests us to consider the following subfamilies $\mathcal{I}_0\subseteq\mathcal{I}$ respectively.

\begin{corollary}\label{cor:circ}
Let $m,n$ be positive integers.
\begin{itemize}
  \item[(1)] For $P=\k[x]$, denote
      $$\mathcal{I}_0=\{((x^m-\lambda)^{nr})\lhd_l H\mid \lambda\in\k,\;r\in\N\}.$$
      If $\mathcal{I}_0$ satisfies the condition in Lemma \ref{lem:I0}, then
      $H^\circ=\sum_{I\in\mathcal{I}_0} (H/I)^\ast$;
  \item[(2)] For $P=\k[x^{\pm1}]$, denote
      $$\mathcal{I}_0=\{((x^n-\lambda)^r)\lhd_l H\mid \lambda\in\k^\ast,\;r\in\N\}.$$
      If $\mathcal{I}_0$ satisfies the condition in Lemma \ref{lem:I0}, then
      $H^\circ=\sum_{I\in\mathcal{I}_0} (H/I)^\ast$;
  \item[(3)] For $P=\k[x^{\pm1}]$, denote
      $$\mathcal{I}_0=\{((x^m-1)^r),((x^m+1)^r),((x^m-\lambda)^r(x^m-\lambda^{-1})^r)\lhd_l H
        \mid \lambda\in\k^\ast\setminus\{\pm1\},\;r\in\N\}.$$
      If $\mathcal{I}_0$ satisfies the condition in Lemma \ref{lem:I0}, then $H^\circ=\sum_{I\in\mathcal{I}_0} (H/I)^\ast$.
\end{itemize}
\end{corollary}

We end this subsection by providing the following lemma on the intersection of principal left ideals in $\mathcal{I}$:

\begin{lemma}\label{lem:intersection}
Let $N$ be a positive integer, and let $P$ be $\k[x]$ or $\k[x^{\pm1}]$. Suppose that $q_1(x),q_2(x),\cdots,q_N(x)\in P$ are coprime with each other. Then
$$\bigcap_{\alpha=1}^N\left(q_\alpha(x)\right)
=\Big(\prod_{\alpha=1}^Nq_\alpha(x)\Big)$$
as left ideals in $H$.
\end{lemma}

\begin{proof}
It is sufficient to prove the case when $N=2$. First there exist $p_1(x),p_2(x)\in P$ such that
$$q_1(x)p_1(x)+q_2(x)p_2(x)=1.$$
Now for any $h\in (q_1(x))\cap(q_2(x))$, we have $h=h_1q_1(x)=h_2q_2(x)$ for some $h_1,h_2\in H$.
Then
\begin{eqnarray*}
h_2 &=& h_2(q_1(x)p_1(x)+q_2(x)p_2(x))
~=~ h_2q_1(x)p_1(x)+h_1q_1(x)p_2(x)  \\
&=& (h_2p_1(x)+h_1p_2(x))q_1(x) ,
\end{eqnarray*}
and thus $h=h_2q_2(x)=(h_2p_1(x)+h_1p_2(x))q_1(x)q_2(x)\in(q_1(x)q_2(x))$.
\end{proof}

\subsection{Invertibility of Matrices}

The following primary fact on the Kronecker products is required, which is well-known in matrix theory.

\begin{lemma}\label{lem:Kron}
\begin{itemize}
\item[(1)]
Let $\Gamma_1$ and $\Gamma_2$ be finite totally ordered sets, and their Cartesian product $\Gamma_1\times \Gamma_2$ be ordered lexicographically. Then for any two square matrices
$$A=\left(a_{i,i'}\mid i,i'\in \Gamma_1\right)\;\;\;\;\text{and}
  \;\;\;\;B=\left(b_{j,j'}\mid j,j'\in \Gamma_2\right)$$
over $\k$, $A$ and $B$ are both invertible if and only if their Kronecker product
$$\left(a_{i,i'}b_{j,j'}\mid (i,j),(i',j')\in \Gamma_1\times \Gamma_2\right)$$
is invertible;
\item[(2)]
Let $n_1,\;n_2,\;\cdots,\;n_m$ be positive integers. Then the square matrices
$$A_k=\left(a^{(k)}_{i_k,i'_k}
           \mid i_k,i'_k\in n_k\right)\;\;\;\;(1\leq k\leq m)$$
over $\k$ are all invertible, if and only if their Kronecker product
$$\left(a^{(1)}_{i_1,i'_1}a^{(2)}_{i_2,i'_2}\cdots a^{(m)}_{i_m,i'_m}
  \mid (i_1,i_2,\cdots,i_m),(i'_1,i'_2,\cdots,i'_m)\in n_1\times n_2\times\cdots\times n_m\right)$$
is invertible.
\end{itemize}
\end{lemma}

A generalized version of above lemma is also required, and we provide a proof for convenience.

\begin{lemma}\label{lem:Kron2}
Let $\Gamma_1$ and $\Gamma_2$ be finite totally ordered sets, and their Cartesian product $\Gamma_1\times \Gamma_2$ be ordered lexicographically. Suppose $A:=\left(a_{i,i'}\mid i,i'\in \Gamma_1\right)$ is invertible. Then the following matrix
$$C:=\left(a_{i,i'}b_{i',j,j'}\mid (i,j),(i',j')\in \Gamma_1\times \Gamma_2\right)$$
is invertible, if and only if $B_{i'}:=\left(b_{i',j,j'}\mid j,j'\in\Gamma_2\right)$ is invertible for each $i'\in\Gamma_1$.
\end{lemma}

\begin{proof}
For clarity, denote the cardinalities of sets by $m:=|\Gamma_1|$ and $n:=|\Gamma_2|$. Clearly, $C$ can be divided into blocks of order $m$:
$$C=\left(a_{i,i'}Y_{i'}\mid i,i'\in\Gamma_1\right).$$
Denote $A^{-1}:=(\widetilde{a}_{i,i'}\mid i,i'\in\Gamma_1)$, then Kronecker product of $A^{-1}$ and $I_n$ (the identity matrix of order $n$) can be written as a block matrix
\begin{equation}\label{eqn:Xinv}
\left(\widetilde{a}_{i,i'}I_n\mid i,i'\in\Gamma_1\right),
\end{equation}
which is invertible according to Lemma \ref{lem:Kron}. One can directly verify that the product of (\ref{eqn:Xinv}) and $C$ is equal to
$$\left(\sum_{k\in\Gamma_1}\widetilde{a}_{i,k}a_{k,i'}B_{i'}\mid i,i'\in\Gamma_1\right)=
\left(\delta_{i,i'}B_{i'}\mid i,i'\in\Gamma_1\right),$$
and the claim is followed.
\end{proof}

Within the work of \cite{GL21}, they proved the invertibility of a certain form of matrices. This plays a crucial role on the process of constructing $(\k\mathbb{D}_\infty)^\circ$. Our construction for $D(m,d,\xi)^\circ$ in this paper would also rely on this fact. However, we provide here another form of that result for simplicity.

\begin{lemma}\emph{(}\cite[Corollary 2]{GL21}\emph{)}
Let $r$ be a positive integer, and $\la\in\k^\ast\setminus\{\pm1\}$. The $2r\times2r$ matrix
\begin{equation}\label{eqn:mat0}
\left(\la^{(-1)^e(2s'+e')}(2s'+e')^s\mid (e,s),(e',s')\in 2\times r\right)
\end{equation}
is invertible.
\end{lemma}

\begin{proof}
Denote $M:=2r-1$. It follows from \cite[Corollary 2]{GL21} that the $2r\times 2r$ matrix
\begin{equation}\label{eqn:matGL0}
\small{\left(\begin{array}{cccccccc}
  1 & 0 & \cdots & 0 & 1 & 0 & \cdots & 0  \\
  \la & \la & \cdots & \la & \la^{-1} & \la^{-1} & \cdots & \la^{-1}  \\
  \la^2 & 2\la^2 & \cdots & 2^{r-1}\la^2 & \la^{-2} & 2\la^{-2} & \cdots & 2^{r-1}\la^{-2}  \\
  \la^3 & 3\la^3 & \cdots & 3^{r-1}\la^3 & \la^{-3} & 3\la^{-3} & \cdots & 3^{r-1}\la^{-3}  \\
  \vdots & \vdots & \ddots & \vdots & \vdots & \vdots & \ddots \vdots  \\
  \la^M & M\la^M & \cdots & M^{r-1}\la^M & \la^{-M} & M\la^{-M} & \cdots & M^{r-1}\la^{-M}
\end{array}\right)}
\end{equation}
is invertible, which can be obtained from the transpose of (\ref{eqn:mat0}) through elementary operations of type 2. Thus (\ref{eqn:mat0}) is also invertible.
\end{proof}

This result has a slightly generalized version, which is required in this paper as well:

\begin{corollary}\label{cor:mat}
Let $r$ be a positive integer, and $\la\in\k^\ast\setminus\{\pm1\}$. Then:
\begin{itemize}
  \item[(1)] The $2r\times2r$ matrix
    \begin{equation}\label{eqn:mat1}
    \left(\la^{(-1)^e(2s'+e')}(a+2s'+e')^s\mid (e,s),(e',s')\in 2\times r\right)
    \end{equation}
    is invertible for any $a\in\k$;
  \item[(2)] The $2r\times2r$ matrix
    \begin{equation}\label{eqn:mat2}
    \left(\la^{(-1)^e(b+2s'+e')}(a+2s'+e')^s\mid (e,s),(e',s')\in 2\times r\right)
    \end{equation}
    is invertible for any $a\in\k$ and $b\in\mathbb{Q}$.
\end{itemize}
\end{corollary}

\begin{proof}
(1) Denote $M:=2r-1$. Firstly, (\ref{eqn:matGL0}) can be obtained from the following matrix through elementary column operations of type 3:
\begin{equation}\label{eqn:matGL1}
\footnotesize{\left(\begin{array}{cccccccc}
  1 & a & \cdots & a^{r-1} & 1 & a & \cdots & a^{r-1}  \\
  \la & (a+1)\la & \cdots & (a+1)^{r-1}\la &
    \la^{-1} & (a+1)\la^{-1} & \cdots & (a+1)^{r-1}\la^{-1}  \\
  \la^2 & (a+2)\la^2 & \cdots & (a+2)^{r-1}\la^2 &
    \la^{-2} & (a+2)\la^{-2} & \cdots & (a+2)^{r-1}\la^{-2}  \\
  \la^3 & (a+3)\la^3 & \cdots & (a+3)^{r-1}\la^3 &
    \la^{-3} & (a+3)\la^{-3} & \cdots & (a+3)^{r-1}\la^{-3}  \\
  \vdots & \vdots & \ddots & \vdots & \vdots & \vdots & \ddots \vdots  \\
  \la^M & (a+M)\la^M & \cdots & (a+M)^{r-1}\la^M &
    \la^{-M} & (a+M)\la^{-M} & \cdots & (a+M)^{r-1}\la^{-M}
\end{array}\right).}
\end{equation}
The process is by deleting powers of $\la$ or $\la^{-1}$ with coefficients containing $a$, from the $r$th column to the $2$nd column, and the $2r$th column to the $(r+1)$th column. Thus (\ref{eqn:matGL1}) is also invertible.

However, (\ref{eqn:mat1}) can be obtained from the transpose of (\ref{eqn:matGL1}) through elementary operations of type 2. This implies the invertibility of (\ref{eqn:mat1}).

(2) The similar argument follows that (\ref{eqn:mat2}) can be obtained from the transpose of the following matrix through elementary column operations:
\begin{equation*}\label{eqn:matGL2}
\tiny{\left(\begin{array}{cccccccc}
  \la^b & a\la^b & \cdots & a^{r-1}\la^b &
    \la^{-b} & a\la^{-b} & \cdots & a^{r-1}\la^{-b}  \\
  \la^{b+1} & (a+1)\la^{b+1} & \cdots & (a+1)^{r-1}\la^{b+1} &
    \la^{-(b+1)} & (a+1)\la^{-(b+1} & \cdots & (a+1)^{r-1}\la^{-(b+1)}  \\
  \la^{b+2} & (a+2)\la^{b+2} & \cdots & (a+2)^{r-1}\la^{b+2} &
    \la^{-(b+2)} & (a+2)\la^{-(b+2)} & \cdots & (a+2)^{r-1}\la^{-(b+2)}  \\
  \la^{b+3} & (a+3)\la^{b+3} & \cdots & (a+3)^{r-1}\la^{b+3} &
    \la^{-(b+3)} & (a+3)\la^{-(b+3)} & \cdots & (a+3)^{r-1}\la^{-(b+3)}  \\
  \vdots & \vdots & \ddots & \vdots & \vdots & \vdots & \ddots \vdots  \\
  \la^{b+M} & (a+M)\la^{b+M} & \cdots & (a+M)^{r-1}\la^{b+M} &
    \la^{-(b+M)} & (a+M)\la^{-(b+M)} & \cdots & (a+M)^{r-1}\la^{-(b+M)}
\end{array}\right)}
\end{equation*}
and this matrix has the same determinant as (\ref{eqn:matGL1}).
\end{proof}

\subsection{Combinatorial Notions}

We end this section by listing two other concepts for later use. The one is the well-known \textit{quantum binomial coefficients} for a parameter $q\in\k^\ast$, which is defined as
$$\binom{l}{k}_q:=\frac{l!_q}{k!_q (l-k)!_q}$$
for integers $l\geq k\geq 0$, where $l!_q:=1_q2_q\cdots l_q$ and $l_q:=1+q+\cdots+q^{l-1}$.

The other one is a part of the \textit{second Stirling number} (see (13.13) in \cite{LW01}), written as
\begin{equation}\label{eqn:2Stirling}
\sum_{t=0}^r\dbinom{r}{t}(-1)^{r-t}t^s,
\end{equation}
which is in fact zero if $0\leq s<r$.

\section{The Finite Dual of Infinite-Dimensional Taft Algebra $T_\infty(n,v,\xi)$}\label{section3}

This section is devoted to providing the structure of $T_\infty(n,v,\xi)^\circ$, which is isomorphic to the Hopf algebra $T_{\infty^\circ}(n,v,\xi)$ constructed by generators and relations in Subsection \ref{subsection:T1}. For the purpose, we choose certain elements in $T_\infty(n,v,\xi)^\circ$ and study their properties in Subsection \ref{subsection:T2}. The desired isomorphism is shown in Subsection \ref{subsection:T3} finally.

\subsection{The Hopf Algebra $T_{\infty^\circ}(n,v,\xi)$}\label{subsection:T1}

Let $n$ be a positive integer, $0\leq v\leq n-1$ and $\xi$ be a primitive $n$th root of $1$. For simplicity we denote $m:=\frac{n}{\gcd(n,v)}$. The Hopf algebra $T_{\infty^\circ}(n,v,\xi)$ is constructed as follows. As an algebra, it is generated by $\Psi_\la\;(\la\in\k),\;\Om,\;F_1,\;F_2$ with relations
\begin{align*}
& \Psi_{\la_1}\Psi_{\la_2}=\Psi_{\la_1+\la_2},\;\;\Psi_0=1,\;\;
\Om^n=1,\;\;F_1^m=0,  \\
& \Om\Psi_\la=\Psi_\la\Om,\;\;F_2\Om=\Om F_2,\;\;F_1\Om=\xi^v\Om F_1,  \\
& F_2\Psi_\la=\Psi_\la F_2,\;\;F_1\Psi_\la=\Psi_\la F_1,\;\;F_1F_2=F_2F_1
\end{align*}
for all $\la,\la_1,\la_2\in\k$. The comultiplication, counit and antipode are given by
\begin{align*}
& \Delta(\Om)=\Om\otimes\Om,\;\; \Delta(F_1)=1\otimes F_1+F_1\otimes\Om,  \\
& \Delta(F_2)=1\otimes F_2+F_2\otimes\Om^m+\sum_{k=1}^{m-1}F_1^{[k]}\otimes\Om^kF_1^{[m-k]},  \\
& \Delta(\Psi_\la)=\sum_{c=0}^{(n/m)-1}(\Psi_{\la\xi^{mc}}\otimes\Psi_\la\varsigma_c)
  (1\otimes 1+\la\sum_{k=1}^{m-1}F_1^{[k]}\otimes\Om^kF_1^{[m-k]}),  \\
& \varepsilon(\Om)=\varepsilon(\Psi_\la)=1,\;\;\varepsilon(F_1)=\varepsilon(F_2)=0,  \\
& S(\Om)=\Om^{n-1},\;\;S(F_1)=-\xi^{-v}\Om^{n-1}F_1,\;\;S(F_2)=-F_2,\;\;
  S(\Psi_\la)=\sum_{c=0}^{(n/m)-1}\Psi_{-\la\xi^{-mc}}\varsigma_c,
\end{align*}
for $\la\in\k$, where $F_1^{[k]}:=\frac{1}{k!_{\xi^v}}F_1^k$ for $1\leq k\leq m-1$, and
$$\varsigma_c:=\frac{m}{n}
               (1+\xi^{-mc}\Om^m+\xi^{-2mc}\Om^{2m}+\cdots+\xi^{-(n-m)c}\Om^{n-m})\;\;\;\;(c\in n/m).$$

The fact that $T_{\infty^\circ}(n,v,\xi)$ is a Hopf algebra would be stated and proved in Subsection \ref{subsection:T3} as Theorem \ref{thm:T}(1). Here we note that
$$\{\Psi_\la\Om^j F_2^s F_1^l \mid \la\in\k,\;j\in n,\;s\in\N,\;l\in m\}$$
is a linear basis, due to an application of the Diamond Lemma \cite{Ber78}.

\subsection{Certain Elements in $T_\infty(n,v,\xi)^\circ$}\label{subsection:T2}

Let $n$ be a positive integer, $0\leq v\leq n-1$, and $\xi$ be a primitive $n$th root of $1$.
According to the definition of $T_\infty(n,v,\xi)$ recalled in Definition \ref{def:T}, we know that it has a linear basis $\{g^jx^l\mid j\in n,\;l\in\N\}$. Denote $m:=\frac{n}{\gcd(n,v)}$, and then $\xi^v$ is a primitive $m$th root of $1$. We define following elements in $T_\infty(n,v,\xi)^\ast$:
\begin{align}\label{eqn:Tgenerator}
& \psi_\la:g^jx^l\mapsto \sum_{u=0}^\infty \delta_{l,um}\la^u,\;\;
  \om:g^jx^l\mapsto \delta_{l,0}\xi^j,\;\;E_1:g^jx^l\mapsto \delta_{l,1},\;\;
  E_2:g^jx^l\mapsto \delta_{l,m}
\end{align}
for any $j\in n$, $l\in\N$ and $\la\in\k$. We remark that these definitions make sense for all $j\in\Z$ as well due to a direct computation.

One can verify that $\psi_\la$ vanishes on the principal left ideal $(x^m-\la)$ for any $\la\in\k$; $\om$ vanishes on $(x)$; $E_2$ vanishes on $(x^{m+1})$; $E_1$ vanishes on $(x^2)$. Clearly, these left ideals are cofinite in $T_\infty(n,v,\xi)$, and hence all the elements in \eqref{eqn:Tgenerator}  lie in $T_\infty(n,v,\xi)^\circ$.

\begin{lemma}\label{lem:Talg}
Following equations hold in $T_\infty(n,v,\xi)^\circ$:
\begin{align*}
& \psi_{\la_1}\psi_{\la_2}=\psi_{\la_1+\la_2},\;\;\psi_0=1,\;\;
\om^n=1,\;\;E_1^m=0,  \\
& \om\psi_\la=\psi_\la\om,\;\;E_2\om=\om E_2,\;\;E_1\om=\xi^v\om E_1,  \\
& E_2\psi_\la=\psi_\la E_2,\;\;E_1\psi_\la=\psi_\la E_1,\;\;E_1E_2=E_2E_1
\end{align*}
for all $\la,\la_1,\la_2\in\k$.
\end{lemma}

\begin{proof}
Note that
$$\Delta(g^jx^l)=(g\otimes g)^j(1\otimes x+x\otimes g^v)^l
  =\sum_{k=0}^l\binom{l}{k}_{\xi^v} g^jx^k\otimes g^{j+kv}x^{l-k}.$$
We prove the lemma by checking values at the basis elements $g^jx^l\;\;(j\in n,\;l\in\N)$:
\begin{eqnarray*}
\langle\psi_{\la_1}\psi_{\la_2},g^jx^l\rangle
&=& \sum_{u=0}^\infty\delta_{l,um}\sum_{i=0}^u\binom{um}{im}_{\xi^v}
    \langle\psi_{\la_1},g^jx^{im}\rangle\langle\psi_{\la_2},g^{j+imv}x^{(u-i)m}\rangle  \\
&=& \sum_{u=0}^\infty\delta_{l,um}\sum_{i=0}^u\binom{u}{i}\la_1^i\la_2^{u-i}
~=~ \langle\psi_{\la_1+\la_2},g^jx^l\rangle,  \\
\langle\psi_0,g^jx^l\rangle
&=& \sum_{u=0}^\infty\delta_{l,um}0^u
~=~ \delta_{l,0}
~=~ \langle\varepsilon,g^jx^l\rangle,  \\
\langle\om\psi_\la,g^jx^l\rangle
&=& \sum_{u=0}^\infty\delta_{l,um}\langle\om,g^j\rangle\langle\psi_\la,g^jx^{um}\rangle  \\
&=& \sum_{u=0}^\infty\delta_{l,um}\langle\psi_\la,g^jx^{um}\rangle\langle\om,g^{j+umv}\rangle
~=~ \langle\psi_\la\om,g^jx^l\rangle,  \\
\langle E_2\om,g^jx^l\rangle
&=& \delta_{l,m}\langle E_2,g^jx^m\rangle\langle\om,g^{j+mv}\rangle
~=~ \delta_{l,m}\langle\om,g^j\rangle\langle E_2,g^jx^m\rangle
~=~ \langle \om E_2,g^jx^l\rangle,  \\
\langle E_1\om,g^jx^l\rangle
&=& \delta_{l,1}\langle E_1,g^jx\rangle\langle\om,g^{j+v}\rangle
~=~ \delta_{l,1}\xi^v\langle\om,g^j\rangle\langle E_1,g^jx\rangle
~=~ \langle\xi^v\om E_1,g^jx^l\rangle,  \\
\langle E_2\psi_\la,g^jx^l\rangle
&=& \sum_{u=0}^\infty \delta_{l,(u+1)m}\binom{(u+1)m}{m}_{\xi^v}
    \langle E_2,g^jx^m\rangle\langle \psi_\la,g^{j+mv}x^{um}\rangle  \\
&=& \sum_{u=0}^\infty \delta_{l,(u+1)m}\binom{(u+1)m}{um}_{\xi^v}
    \langle \psi_\la,g^jx^{um}\rangle\langle E_2,g^{j+umv}x^m\rangle
~=~ \langle\psi_\la E_2,g^jx^l\rangle,  \\
\langle E_1\psi_\la,g^jx^l\rangle
&=& \sum_{u=0}^\infty \delta_{l,um+1}\binom{um+1}{1}_{\xi^v}
    \langle E_1,g^jx\rangle\langle\psi_\la,g^{j+v}x^{um}\rangle  \\
&=& \sum_{u=0}^\infty \delta_{l,um+1}\binom{um+1}{um}_{\xi^v}
    \langle\psi_\la,g^jx^{um}\rangle\langle E_1,g^{j+umv}x\rangle
~=~ \langle \psi_\la E_1,g^jx^l\rangle.
\end{eqnarray*}
Also, we find by induction on $k\in n$ and $k'\in m$ that
\begin{eqnarray*}
\langle\om^k,g^jx^l\rangle
&=& \delta_{l,0}\langle\om^k,g^j\rangle
~=~ \delta_{l,0}\langle\om,g^j\rangle^k
~=~ \delta_{l,0}\xi^{jk},  \\
\langle E_1^{k'},g^jx^l\rangle
&=& \delta_{l,k'}\binom{k'}{1}_{\xi^v}\langle E_1,g^jx\rangle\langle E_1^{k'-1},g^{j+v}x^{k'-1}\rangle  \\
&=& \delta_{l,k'}\frac{k'!_{\xi^v}}{(k'-1)!_{\xi^v}}\langle E_1^{k'-1},g^{j+v}x^{k'-1}\rangle  \\
&=& \delta_{l,k'}\frac{k'!_{\xi^v}}{(k'-1)!_{\xi^v}}\frac{(k'-1)!_{\xi^v}}{(k'-2)!_{\xi^v}}
    \langle E_1^{k'-2},g^{j+2v}x^{k'-2}\rangle
~=~ \cdots
~=~ \delta_{l,k'}k'!_{\xi^v}.
\end{eqnarray*}
Thus $\om^n=1$ and $E_1^m=0$ hold.
\end{proof}

Furthermore if denote $E_1^{[k']}:=\frac{1}{k'!_{\xi^v}}E_1^{k'}$, we verify that for any $k'\in m$, $k,j\in n$ and $l\in \N$,
\begin{eqnarray*}
\langle\om^k E_1^{[k']},g^jx^l\rangle
&=& \delta_{l,k'}\langle\om^k,g^j\rangle\langle E_1^{[k']},g^jx^{k'}\rangle
~=~ \delta_{l,k'}\xi^{jk},  \\
\langle\psi_\la\om^k E_1^{[k']},g^jx^l\rangle
&=& \sum_{u=0}^\infty \delta_{l,um+k'}\binom{um+k'}{um}_{\xi^v}
    \langle\psi_\la,g^jx^{um}\rangle\langle \om^k E_1^{[k']},g^{j+umv}x^{k'}\rangle  \\
&=& \sum_{u=0}^\infty \delta_{l,um+k'}\la^u\xi^{jk},
\end{eqnarray*}
If we denote $E_2^{[s]}:=\frac{1}{s!}E_2^s$ for $s\geq 1$, followings can be obtained by induction:
\begin{eqnarray*}
\langle E_2^s,g^jx^l\rangle
&=& \delta_{l,sm}
    \binom{sm}{m}_{\xi^v}\langle E_2,g^jx^m\rangle\langle E_2^{s-1},g^{j+mv}x^{(s-1)m}\rangle  \\
&=& \delta_{l,sm}s\langle E_2^{s-1},g^{j+mv}x^{(s-1)m}\rangle  \\
&=& \delta_{l,sm}s\binom{(s-1)m}{m}_{\xi^v}
    \langle E_2,g^{j+mv}x^m\rangle\langle E_2^{s-2},g^{j+2mv}x^{(s-2)m}\rangle  \\
&=& \delta_{l,sm}s(s-1)\langle E_2^{s-2},g^{j+2mv}x^{(s-2)m}\rangle
~=~ \cdots
~=~ \delta_{l,sm}s!,  \\
\langle E_2^{[s]},g^jx^l\rangle
&=& \delta_{l,sn},  \\
\langle E_2^{[s]}E_1^{[k']},g^jx^l\rangle
&=& \delta_{l,sm+k'}\binom{sm+k'}{sm}_{\xi^v}
    \langle E_2^{[s]},g^jx^{sm}\rangle\langle E_1^{[k']},g^{j+smv}x^{k'}\rangle
~=~ \delta_{l,sm+k'},  \\
\langle\om^kE_2^{[s]}E_1^{[k']},g^jx^l\rangle
&=& \delta_{l,sm+k'}\langle\om^k,g^j\rangle\langle E_2^{[s]}E_1^{[k']},g^jx^{sm+k'}\rangle
~=~ \delta_{l,sm+k'}\xi^{jk}.
\end{eqnarray*}

\begin{lemma}
Following equations hold in $T_\infty(n,v,\xi)^\circ$:
\begin{align*}
& \Delta(\om)=\om\otimes\om,\;\; \Delta(E_1)=1\otimes E_1+E_1\otimes\om,  \\
& \Delta(E_2)=1\otimes E_2+E_2\otimes\om^m+\sum_{k=1}^{m-1}E_1^{[k]}\otimes\om^kE_1^{[m-k]},  \\
& \Delta(\psi_\la)=\sum_{c=0}^{(n/m)-1}(\psi_{\la\xi^{mc}}\otimes\psi_\la\sigma_c)
  (1\otimes 1+\la\sum_{k=1}^{m-1}E_1^{[k]}\otimes\om^kE_1^{[m-k]}),  \\
& \varepsilon(\om)=\varepsilon(\psi_\la)=1,\;\;\varepsilon(E_1)=\varepsilon(E_2)=0,  \\
& S(\om)=\om^{n-1},\;\;S(E_1)=-\xi^{-v}\om^{n-1}E_1,\;\;S(E_2)=-E_2,\;\;
  S(\psi_\la)=\sum_{c=0}^{(n/m)-1}\psi_{-\la\xi^{-mc}}\sigma_c,
\end{align*}
for $\la\in\k$, where $E_1^{[k]}:=\frac{1}{k!_{\xi^v}}E_1^k$ for $1\leq k\leq m-1$, and
$$\sigma_c:=\frac{m}{n}
            (1+\xi^{-mc}\om^m+\xi^{-2mc}\om^{2m}+\cdots+\xi^{-(n-m)c}\om^{n-m})\;\;\;\;(c\in n/m).$$
\end{lemma}

\begin{proof}
Note that $(g^jx^l)(g^{j'}x^{l'})=\xi^{j'l}g^{j+j'}x^{l+l'}\;(j,j'\in n,\;l,l'\in\N)$. We also prove the lemma by checking values at each $g^jx^l\otimes g^{j'}x^{l'}$.
\begin{eqnarray*}
\langle\om,\xi^{j'l}g^{j+j'}x^{l+l'}\rangle
&=& \delta_{l+l',0}\xi^{j'l+j+j'}
~=~ \delta_{l,0}\xi^j\delta_{l',0}\xi^{j'}
~=~ \langle\om\otimes\om,g^jx^l\otimes g^{j'}x^{l'}\rangle,  \\
\langle E_1,\xi^{j'l}g^{j+j'}x^{l+l'}\rangle
&=& \delta_{l+l',1}\xi^{j'l}
~=~ \delta_{l,0}\delta_{l',1}+\delta_{l,1}\delta_{l',0}\xi^{j'}  \\
&=& \langle 1\otimes E_1+E_1\otimes\om,g^jx^l\otimes g^{j'}x^{l'}\rangle,  \\
\langle E_2,\xi^{j'l}g^{j+j'}x^{l+l'}\rangle
&=& \delta_{l+l',m}\xi^{j'l}
~=~ \delta_{l,0}\delta_{l',m}+\delta_{l,m}\delta_{l',0}\xi^{j'm}
    +\sum_{k=1}^{m-1}\delta_{l,k}\delta_{l',m-k}\xi^{j'k}  \\
&=& \langle 1\otimes E_2+E_2\otimes\om^m+\sum_{k=1}^{m-1}E_1^{[k]}\otimes\om^kE_1^{[m-k]},
    g^jx^l\otimes g^{j'}x^{l'}\rangle.
\end{eqnarray*}
It is not hard to find that for each $c\in n/m$, $j\in n$ and $l\in\N$:
$$\langle\psi_\la\sigma_c,g^jx^l\rangle=
\left\{\begin{array}{ll}
  \sum_{u=0}^\infty\delta_{l,um}\la^u, & \text{if}\;\;j\equiv c~(\mathrm{mod}~n/m);  \\
  0, & \text{otherwise}.
\end{array}
\right.$$
Now compute that
\begin{eqnarray*}
&& \langle \psi_\la,\xi^{j'l}g^{j+j'}x^{l+l'}\rangle
~=~ \sum_{u=0}^\infty\delta_{l+l',um}\la^u\xi^{j'l}  \\
&=& \sum_{u,u'=0}^\infty (\delta_{l,um}\delta_{l',u'm}\la^{u+u'}\xi^{j'um}
    +\sum_{k=1}^{m-1}\delta_{l,um+k}\delta_{l',u'm+(m-k)}\la^{u+u'+1}\xi^{j'(um+k)})  \\
&=& \sum_{u=0}^\infty \delta_{l,um}(\la\xi^{j'm})^u\sum_{u'=0}^\infty\delta_{l',u'm}\la^{u'}
  +\la\sum_{k=1}^{m-1}\sum_{u=0}^\infty\delta_{l,um+k}(\la\xi^{j'm})^u
      \sum_{u'=0}^\infty\delta_{l',u'm+(m-k)}\la^{u'}\xi^{j'k}  \\
&=& \langle\psi_{\la\xi^{j'm}}\otimes \psi_\la
    +\la\sum_{k=1}^{m-1}\psi_{\la\xi^{j'm}} E_1^{[k]}\otimes\psi_\la\om^kE_1^{[m-k]},
    g^jx^l\otimes g^{j'}x^{l'}\rangle  \\
&=& \langle\sum_{c=0}^{(n/m)-1}(\psi_{\la\xi^{mc}}\otimes\psi_\la\sigma_c
    +\la\sum_{k=1}^{m-1}\psi_{\la\xi^{mc}}E_1^{[k]}\otimes\psi_\la\sigma_c\om^kE_1^{[m-k]}),
    g^jx^l\otimes g^{j'}x^{l'}\rangle  \\
&=& \langle\sum_{c=0}^{(n/m)-1}(\psi_{\la\xi^{mc}}\otimes\psi_\la\sigma_c)
           (1\otimes 1+\la\sum_{k=1}^{m-1}E_1^{[k]}\otimes\om^kE_1^{[m-k]}),
           g^jx^l\otimes g^{j'}x^{l'}\rangle.
\end{eqnarray*}
The expressions of the counit and the antipode are clear.
\end{proof}

\subsection{The Generation Problem}\label{subsection:T3}

\begin{proposition}\label{p.3.3}
As an algebra, $T_\infty(n,v,\xi)^\circ$ is generated by $\psi_\la$, $\om$, $E_2$ and $E_1$ for $\la\in\k$.
\end{proposition}

\begin{proof}
Clearly $T_\infty(n,v,\xi)$ has a polynomial subalgebra $P=\k[x]$. Choose
$$\mathcal{I}_0=\{((x^m-\la)^{nr})\lhd_l T_\infty(n,v,\xi) \mid \la\in\k,\;r\in\N\}.$$
First for any cofinite left ideal $I$ of $T_\infty(n,v,\xi)$, according to Lemma \ref{lem:p(x)}(1), there must be some non-zero polynomial
$$p(x)=\prod_{\alpha=1}^N (x^m-\la_\alpha)^{nr_\alpha}$$
for some $N\geq1$, $r_\alpha\in\N$ and distinct $\la_\alpha\in\k$, which generates a cofinite left ideal contained in $I$.
It can be known by Lemma \ref{lem:intersection} that
$$\bigcap_{\alpha=1}^N((x^m-\la_\alpha)^{nr_\alpha})= (p(x))\subseteq I$$
as left ideals of $T_\infty(n,v,\xi)$,
and thus Lemma \ref{lem:I0} or Corollary \ref{cor:circ}(1) can be applied to obtain that $T_\infty(n,v,\xi)^\circ=\sum_{I\in\mathcal{I}_0}(T_\infty(n,v,\xi)/I)^\ast$.

Now we try to prove that for each $I\in\mathcal{I}_0$, the subspace $(T_\infty(n,v,\xi)/I)^\ast$ can be spanned by some products of $\psi_\la,\om,E_2,E_1\;(\la\in\k)$. Specifically, suppose $I=((x^m-\la)^{nr})$ for some $r\in\N$ and $\la\in\k$, and we aim to show that
\begin{equation}\label{eqn:T}
\{\psi_\la\om^j E_2^s E_1^l \mid j\in n,\;s\in nr,\;l\in m\}
\end{equation}
is a linear basis of $n^2mr$-dimensional space $(T_\infty(n,v,\xi)/I)^\ast$.

Evidently $I=\k\{g^jx^l(x^m-\la)^{nr}\mid j\in n,\;l\in\N\}$. Also, $T_\infty(n,v,\xi)/I$ has a linear basis
$$\{g^jx^l(x^m-\la)^{s}\mid j\in n,\;l\in m,\;s\in nr\},$$
and hence $\dim((T_\infty(n,v,\xi)/I)^\ast)=\dim(T_\infty(n,v,\xi)/I)=n^2mr$. Next we show that all elements in (\ref{eqn:T}) vanish on $I$.
Recall that
$$E_2^{[s]}E_1^{[l]}=\frac{1}{s!\cdot l!_{\xi^v}}E_2^sE_1^l:
  g^{j'}x^{l'}\mapsto\delta_{l',sm+l}\;\;\;(l\in m,\;s,l'\in\N,\;j'\in n).$$
Therefore, for any $j,j'\in n$, $s\in nr$, $l\in m$ and $l'\in\N$, we have
\begin{eqnarray*}
& & \langle\psi_\la\om^j E_2^{[s]} E_1^{[l]},g^{j'}x^{l'}(x^m-\la)^{nr}\rangle  \\
&=& \langle\psi_\la\om^j E_2^{[s]} E_1^{[l]},
           \sum_{t=0}^{nr}\binom{nr}{t}g^{j'}x^{l'+tm}(-\la)^{nr-t}\rangle  \\
&=& \sum_{t=0}^{nr}\binom{nr}{t}(-\la)^{nr-t}
    \langle\psi_\la\otimes\om^jE_2^{[s]}E_1^{[l]},\Delta(g^{j'}x^{l'+tm})\rangle  \\
&=& \sum_{t=s}^{nr}\binom{nr}{t}(-\la)^{nr-t}\delta_{l',l}
    \binom{l+tm}{(t-s)m}_{\xi^v}\langle\psi_\la,g^{j'}x^{(t-s)m}\rangle
    \langle\om^jE_2^{[s]}E_1^{[l]},g^{j'+(t-s)mv}x^{l+sm}\rangle  \\
&=& \sum_{t=s}^{nr}\binom{nr}{t}(-\la)^{nr-t}\delta_{l',l}
    \binom{t}{t-s}\la^{t-s}\xi^{(j'+(t-s)mv)j}  \\
&=& \la^{nr-s}\delta_{l',l}\xi^{j'j}
    \sum_{t=s}^{nr}\binom{nr}{t}(-1)^{nr-t}\binom{t}{s}  \\
&=& \la^{nr-s}\delta_{l',l}\xi^{j'j}
    \sum_{t=s}^{nr}\binom{nr-s}{t-s}(-1)^{nr-t}
~=~ 0.
\end{eqnarray*}
In other words, $\langle\psi_\la\om^jE_2^{[s]}E_1^{[l]},I\rangle=0$ for any $j\in n$, $s\in nr$ and $l\in m$, which follows that the elements in (\ref{eqn:T}) belong to $I^\perp=(T_\infty(n,v,\xi)/I)^\ast$.

Finally we prove that the elements in (\ref{eqn:T}) are linearly independent. Choose a lexicographically ordered set of elements $\{g^{j'}x^{l'+s'm}\in T_\infty(n,v,\xi)\mid (j',s',l')\in n\times nr\times m\}$. Our goal is to show that the $n^2mr\times n^2mr$ square matrix
$$A:=\left(\langle\psi_\la\om^jE_2^{[s]}E_1^{[l]},g^{j'}x^{l'+s'm}\rangle
         \mid (j,s,l),(j',s',l')\in n\times nr\times m\right)$$
is invertible, which would imply the linear independence of (\ref{eqn:T}). We try to compute
\begin{eqnarray*}
\langle\psi_\la\om^jE_2^{[s]}E_1^{[l]},g^{j'}x^{l'+s'm}\rangle
&=&  \langle\psi_\la\otimes\om^jE_2^{[s]}E_1^{[l]},\Delta(g^{j'}x^{l'+s'm})\rangle.
\end{eqnarray*}

One finds that $\langle\psi_\la\om^jE_2^{[s]}E_1^{[l]},g^{j'}x^{l'+s'm}\rangle=0$ if $s>s'$. Otherwise when $s\leq s'$,
\begin{eqnarray*}
&& \langle\psi_\la\om^jE_2^{[s]}E_1^{[l]},g^{j'}x^{l'+s'm}\rangle  \\
&=& \delta_{l',l}\binom{l+s'm}{(s'-s)m}_{\xi^v}\langle\psi_\la,g^{j'}x^{(s'-s)m}\rangle
    \langle\om^jE_2^{[s]}E_1^{[l]}g^{j'+(s'-s)mv}x^{l+sm}\rangle  \\
&=& \delta_{l',l}\binom{s'}{s}\la^{s'-s}\xi^{j'j}.
\end{eqnarray*}
The computation follows that the matrix $A$ is the Kronecker product of following three invertible matrices:
\begin{itemize}
  \item A Vandermonde matrix $(\xi^{j'j}\mid j,j'\in n)$,
  \item An $nr\times nr$ upper triangular matrix with entries $\binom{s'}{s}\la^{s'-s}$ for $s\leq s'$, and
  \item An identity matrix matrix $(\delta_{l',l}\mid l,l'\in m)$.
\end{itemize}
Thus $A$ is invertible by Lemma \ref{lem:Kron}, and hence (\ref{eqn:T}) is a basis of $(T_\infty(n,v,\xi)/I)^\ast$. We conclude that
$$T_\infty(n,v,\xi)^\circ=\sum_{I\in\mathcal{I}_0}(T_\infty(n,v,\xi)/I)^\ast
  =\k\{\psi_\la\om^j E_2^s E_1^l \mid \la\in\k,\;j\in n,\;s\in\N,\;l\in m\}.$$
\end{proof}

\begin{theorem}\label{thm:T}
\begin{itemize}
  \item[(1)] $T_{\infty^\circ}(n,v,\xi)$ constructed in Subsection \ref{subsection:T1} is a Hopf algebra;
  \item[(2)] As a Hopf algebra, $T_\infty(n,v,\xi)^\circ$ is isomorphic to $T_{\infty^\circ}(n,v,\xi)$.
\end{itemize}
\end{theorem}

\begin{proof}
Consider the following map
$$\Theta:T_{\infty^\circ}(n,v,\xi)\rightarrow T_\infty(n,v,\xi)^\circ,\;\;
  \Psi_\la\mapsto\psi_\la,\;\Om\mapsto\om,\;
  F_1\mapsto E_1,\;F_2\mapsto E_2,$$
where $\la\in\k.$ This is an epimorphism of algebras by Lemma \ref{lem:Talg} and Proposition \ref{p.3.3}. Furthermore, $\Theta$ would become an isomorphism of Hopf algebras with desired coalgebra structure and antipode, as long as it is injective (since $T_\infty(n,v,\xi)^\circ$ is in fact a Hopf algebra).

In order to show that $\Theta$ is injective, it is enough to show the linear independence of
$$\{\psi_\la\om^j E_2^s E_1^l \mid \la\in\k,\;j\in n,\;s\in\N,\;l\in m\}$$
in $T_\infty(n,v,\xi)^\circ$. By the linear independence of elements in \eqref{eqn:T}, we only need to show that any finite sum of form
\begin{equation}\label{eqn:Tinj}
\sum_{\alpha=1}^N(T_\infty(n,v,\xi)/((x^m-\la_\alpha)^{nr_\alpha}))^\ast
=\sum_{\alpha=1}^N((x^m-\la_\alpha)^{nr_\alpha})^\perp
\end{equation}
is direct as long as $\la_\alpha$'s are distinct from each other. But this is clear:
\begin{eqnarray*}
& &((x^m-\la_{\alpha'})^{nr})^\perp
   \cap[\sum_{\alpha\neq\alpha'}((x^m-\la_\alpha)^{nr})^\perp]
~=~ ((x^m-\la_{\alpha'})^{nr})^\perp
    \cap[\bigcap_{\alpha\neq\alpha'}((x^m-\la_\alpha)^{nr})]^\perp  \\
&=& ((x^m-\la_{\alpha'})^{nr})^\perp
    \cap(\prod_{\alpha\neq\alpha'}(x^m-\la_\alpha)^{nr})^\perp
~=~ [((x^m-\la_{\alpha'})^{nr})
    +(\prod_{\alpha\neq\alpha'}(x^m-\la_\alpha)^{nr})]^\perp  \\
&=& T_\infty(n,v,\xi)^\perp
~=~ 0.
\end{eqnarray*}
\end{proof}

\begin{remark}
\emph{As a special case, the connected algebraic groups of dimension one $H_1=\k[x]$ is equal to $T_\infty(1,0,1)$. Therefore, Theorem \ref{thm:T} gives ${H_1}^\circ$, which is well-known (see \cite[Example 9.1.7]{Mon93} for example).}

\end{remark}

\section{The Finite Dual of Generalized Liu's Algebra $B(n,\om,\gamma)$}\label{section4}

The concept and definition of a generalized Liu's algebra $B(n,\om,\gamma)$ is given in \cite[Section 3.4]{BZ10}. In this section, we use the same procedure as Section \ref{section3} to determine the structure of $B(n,\om,\gamma)^\circ$.

\subsection{The Hopf Algebra $B_\circ(n,\om,\gamma)$}\label{subsection:B1}

Let $n$ and $\om$ be positive integers, and $\gamma$ be a primitive $n$th root of $1$. For convenience, $\la^\frac{1}{m}$ always denotes an $m$th root of $\la\in\k^\ast$ for a positive integer $m$ throughout the paper.

The Hopf algebra $B_\circ(n,\om,\gamma)$ is constructed as follows. As an algebra, it is generated by
$$\Psi_{\la^\frac{1}{\om},\la^\frac{1}{n}},\;F_1,\;F_2\;\;(\la^\frac{1}{\om},\la^\frac{1}{n}\in\k^\ast)$$
(where the parameters of $\Psi$ are meant to range through
$\{(\alpha,\beta)\in\k^\ast\times\k^\ast\mid \alpha^\omega=\beta^n\}$)
with relations
\begin{align*}
& \Psi_{\la_1^\frac{1}{\om},\la_1^\frac{1}{n}}\Psi_{\la_2^\frac{1}{\om},\la_2^\frac{1}{n}}
  =\Psi_{(\la_1^\frac{1}{\om}\la_2^\frac{1}{\om}),(\la_1^\frac{1}{n}\la_2^\frac{1}{n})},\;\;
  \Psi_{1,1}=1,\;\;F_1^n=0,  \\
& F_2\Psi_{\la^\frac{1}{\om},\la^\frac{1}{n}}=\Psi_{\la^\frac{1}{\om},\la^\frac{1}{n}}F_2,\;\;
  F_1\Psi_{\la^\frac{1}{\om},\la^\frac{1}{n}}
  =\la^\frac{1}{n}\Psi_{\la^\frac{1}{\om},\la^\frac{1}{n}}F_1,\;\;
  F_1F_2=F_2F_1+\frac{1}{n}F_1
\end{align*}
for $\la^\frac{1}{\om},\la^\frac{1}{n},\la_1^\frac{1}{\om},\la_1^\frac{1}{n},
\la_2^\frac{1}{\om},\la_2^\frac{1}{n}\in\k^\ast$. The comultiplication, counit and antipode are given by
\begin{align*}
& \Delta(F_1)=1\otimes F_1+F_1\otimes\Psi_{1,\gamma},\;\;
  \Delta(F_2)=1\otimes F_2+F_2\otimes 1
  -\sum_{k=1}^{n-1}F_1^{[k]}\otimes\Psi_{1,\gamma}^kF_1^{[n-k]},  \\
& \Delta(\Psi_{\la^\frac{1}{\om},\la^\frac{1}{n}})
  =(\Psi_{\la^\frac{1}{\om},\la^\frac{1}{n}}\otimes\Psi_{\la^\frac{1}{\om},\la^\frac{1}{n}})
   (1\otimes 1+(1-\la)\sum_{k=1}^{n-1}F_1^{[k]}\otimes\Psi_{1,\gamma}^kF_1^{[n-k]}),  \\
& \varepsilon(\Psi_{\la^\frac{1}{\om},\la^\frac{1}{n}})=1,\;\;\varepsilon(F_1)=\varepsilon(F_2)=0,  \\
& S(F_1)=-\gamma^{n-1}\Psi_{1,\gamma}^{n-1}F_1,\;\;S(F_2)=-F_2,\;\;
  S(\Psi_{\la^\frac{1}{\om},\la^\frac{1}{n}})=\Psi_{\la^\frac{-1}{\om},\la^\frac{-1}{n}}
\end{align*}
for $\la^\frac{1}{\om},\la^\frac{1}{n}\in\k^\ast$, where $F_1^{[k]}:=\frac{1}{k!_\gamma}F_1^k$ for $1\leq k\leq n-1$.

The fact that $B_\circ(n,\om,\gamma)$ is a Hopf algebra would be stated and proved in Subsection \ref{subsection:B3} as Theorem \ref{thm:B}(1). Let $\eta$ be a primitive $\omega$th root of $1$. Here we note that
$$\{\Psi_{\la^\frac{1}{\om},\la^\frac{1}{n}}\Psi_{\eta,1}^i\Psi_{1,\gamma}^jF_2^sF_1^l
  \mid \la^\frac{1}{\om},\la^\frac{1}{n}~\text{are fixed roots of}~\la\in\k^\ast,\;
  i\in\om,\;j\in n,\;s\in\N,\;l\in n\}$$
is a linear basis, due to an application of the Diamond Lemma \cite{Ber78}.

\subsection{Certain Elements in $B(n,\om,\gamma)^\circ$}\label{subsection:B2}

Let $n$ and $\om$ be positive integers, and $\gamma$ be a primitive $n$th root of $1$.
According to the definition of $B(n,\om,\gamma)$ recalled in Definition \ref{def:B}, we know that it has a linear basis $\{x^ig^jy^l\mid i\in\om,\;j\in\Z,\;l\in n\}$. Define following elements in $B(n,\om,\gamma)^\ast$:
\begin{align}\label{eqn:Bgenerator}
& \psi_{\la^\frac{1}{\om},\la^\frac{1}{n}}:
    x^ig^jy^l\mapsto \delta_{l,0}\la^\frac{i}{\om}\la^\frac{j}{n},\;\;E_1:x^ig^jy^l\mapsto \delta_{l,1},\;\;
  E_2:x^ig^jy^l\mapsto \delta_{l,0}(\frac{i}{\om}+\frac{j}{n})
\end{align}
for any $i\in\om$, $j\in\Z$, $l\in n$ and $\la^\frac{1}{\om},\la^\frac{1}{n}\in\k^\ast$. We remark that these definitions make sense for all $i\in\Z$ as well.

One can verify that $\psi_{\la^\frac{1}{\om},\la^\frac{1}{n}}$ vanishes on the principal left ideal $(g^n-\la)$ for any $\la^\frac{1}{\om},\la^\frac{1}{n}\in\k^\ast$; $E_2$ vanishes on $((g^n-1)^2)$; $E_1$ vanishes on $(g^n-1)$. Therefore, these elements lie in $B(n,\om,\gamma)^\circ$.

\begin{lemma}\label{lem:Balg}
Following equations hold in $B(n,\om,\gamma)^\circ$:
\begin{align*}
& \psi_{\la_1^\frac{1}{\om},\la_1^\frac{1}{n}}\psi_{\la_2^\frac{1}{\om},\la_2^\frac{1}{n}}
  =\psi_{(\la_1^\frac{1}{\om}\la_2^\frac{1}{\om}),(\la_1^\frac{1}{n}\la_2^\frac{1}{n})},\;\;
  \psi_{1,1}=1,\;\;E_1^n=0,  \\
& E_2\psi_{\la^\frac{1}{\om},\la^\frac{1}{n}}=\psi_{\la^\frac{1}{\om},\la^\frac{1}{n}}E_2,\;\;
  E_1\psi_{\la^\frac{1}{\om},\la^\frac{1}{n}}
  =\la^\frac{1}{n}\psi_{\la^\frac{1}{\om},\la^\frac{1}{n}}E_1,\;\;
  E_1E_2=E_2E_1+\frac{1}{n}E_1
\end{align*}
for all $\la^\frac{1}{\om},\la^\frac{1}{n},\la_1^\frac{1}{\om},\la_1^\frac{1}{n},
\la_2^\frac{1}{\om},\la_2^\frac{1}{n}\in\k^\ast$.
\end{lemma}

\begin{proof}
Note that
$$\Delta(x^ig^jy^l)=(x\otimes x)^i(g\otimes g)^j(1\otimes y+y\otimes g)^l
  =\sum_{k=0}^l\binom{l}{k}_\gamma x^ig^jy^k\otimes x^ig^{j+k}y^{l-k}.$$
We prove the lemma through checking their values on the basis $x^ig^jy^l\;\;(i\in\om,\;j\in\Z,\;l\in n)$:
\begin{eqnarray*}
\langle\psi_{\la_1^\frac{1}{\om},\la_1^\frac{1}{n}}\psi_{\la_2^\frac{1}{\om},\la_2^\frac{1}{n}},
       x^ig^jy^l\rangle
&=& \delta_{l,0}\langle\psi_{\la_1^\frac{1}{\om},\la_1^\frac{1}{n}},x^ig^j\rangle
    \langle\psi_{\la_2^\frac{1}{\om},\la_2^\frac{1}{n}},x^ig^j\rangle
~=~ \delta_{l,0}\la_1^\frac{i}{\om}\la_2^\frac{i}{\om}\la_1^\frac{j}{n}\la_2^\frac{j}{n}  \\
&=& \langle\psi_{\la_1^\frac{1}{\om}\la_2^\frac{1}{\om},\la_1^\frac{1}{n}\la_2^\frac{1}{n}},
           x^ig^jy^l\rangle,  \\
\langle\psi_{1,1},x^ig^jy^l\rangle
&=& \delta_{l,0}1^i1^j
~=~ \langle\varepsilon,x^ig^jy^l\rangle,  \\
\langle E_2\psi_{\la^\frac{1}{\om},\la^\frac{1}{n}},x^ig^jy^l\rangle
&=& \delta_{l,0}\langle E_2,x^ig^j\rangle\langle\psi_{\la^\frac{1}{\om},\la^\frac{1}{n}},x^ig^j\rangle
~=~ \delta_{l,0}(\frac{i}{\om}+\frac{j}{n})\la^\frac{i}{\om}\la^\frac{j}{n}  \\
&=& \delta_{l,0}\langle\psi_{\la^\frac{1}{\om},\la^\frac{1}{n}},x^ig^j\rangle\langle E_2,x^ig^j\rangle
~=~ \langle\psi_{\la^\frac{1}{\om},\la^\frac{1}{n}}E_2,x^ig^jy^l\rangle,  \\
\langle E_1\psi_{\la^\frac{1}{\om},\la^\frac{1}{n}},x^ig^jy^l\rangle
&=& \delta_{l,1}\langle E_1,x^ig^jy\rangle
    \langle\psi_{\la^\frac{1}{\om},\la^\frac{1}{n}},x^ig^{j+1}\rangle
~=~ \delta_{l,1}\la^\frac{i}{\om}\la^\frac{j+1}{n}  \\
&=& \la^\frac{1}{n}\delta_{l,1}\la^\frac{i}{\om}\la^\frac{j}{n}
~=~ \la^\frac{1}{n}\delta_{l,1}\langle\psi_{\la^\frac{1}{\om},\la^\frac{1}{n}},x^ig^j\rangle
    \langle E_1,x^ig^jy\rangle  \\
&=& \langle\la^\frac{1}{n}\psi_{\la^\frac{1}{\om},\la^\frac{1}{n}}E_1,x^ig^jy^l\rangle,  \\
\langle E_1E_2,x^ig^jy^l\rangle
&=& \delta_{l,1}\langle E_1,x^ig^jy\rangle\langle E_2,x^ig^{j+1}\rangle
~=~ \delta_{l,1}(\frac{i}{\om}+\frac{j+1}{n})  \\
&=& \delta_{l,1}(\frac{i}{\om}+\frac{j}{n})+\delta_{l,1}\frac{1}{n}  \\
&=& \delta_{l,1}\langle E_2,x^ig^j\rangle\langle E_1,x^ig^jy\rangle
    +\delta_{l,1}\frac{1}{n}\langle E_1,x^ig^jy\rangle  \\
&=& \langle E_2E_1+\frac{1}{n}E_1,x^ig^jy^l\rangle.
\end{eqnarray*}
Similar to the case in Lemma \ref{lem:Talg}, it can be found by induction on $k\in n$ that
$$\langle E_1^k,x^ig^jy^l\rangle=\delta_{l,k}k!_\gamma,$$
and thus $E_1^n=0$.
\end{proof}

Furthermore if denote $E_1^{[k]}:=\frac{1}{k!_\gamma}E_1^k$, we can also verify that for any $\la^\frac{1}{\om},\la^\frac{1}{n}\in\k^\ast$ and $k\in n$,
\begin{equation}\label{eqn:Bmore}
\langle\psi_{\la^\frac{1}{\om},\la^\frac{1}{n}} E_1^{[k]},x^ig^jy^l\rangle
=\delta_{l,k}\langle\psi_{\la^\frac{1}{\om},\la^\frac{1}{n}},x^ig^j\rangle
 \langle E_1^{[k]},x^ig^jy^k\rangle
=\delta_{l,k}\la^\frac{i}{\om}\la^\frac{j}{n}.
\end{equation}

\begin{lemma}
Following equations hold in $B(n,\om,\gamma)^\circ$:
\begin{align*}
& \Delta(E_1)=1\otimes E_1+E_1\otimes\psi_{1,\gamma},\;\;
  \Delta(E_2)=1\otimes E_2+E_2\otimes 1
  -\sum_{k=1}^{n-1}E_1^{[k]}\otimes\psi_{1,\gamma}^kE_1^{[n-k]},  \\
& \Delta(\psi_{\la^\frac{1}{\om},\la^\frac{1}{n}})
  =(\psi_{\la^\frac{1}{\om},\la^\frac{1}{n}}\otimes\psi_{\la^\frac{1}{\om},\la^\frac{1}{n}})
   (1\otimes 1+(1-\la)\sum_{k=1}^{n-1}E_1^{[k]}\otimes\psi_{1,\gamma}^kE_1^{[n-k]}),  \\
& \varepsilon(\psi_{\la^\frac{1}{\om},\la^\frac{1}{n}})=1,\;\;\varepsilon(E_1)=\varepsilon(E_2)=0,  \\
& S(E_1)=-\gamma^{n-1}\psi_{1,\gamma}^{n-1}E_1,\;\;S(E_2)=-E_2,\;\;
  S(\psi_{\la^\frac{1}{\om},\la^\frac{1}{n}})=\psi_{\la^\frac{-1}{\om},\la^\frac{-1}{n}}
\end{align*}
for $\la^\frac{1}{\om},\la^\frac{1}{n}\in\k^\ast$, where $E_1^{[k]}:=\frac{1}{k!_\gamma}E_1^k$ for $1\leq k\leq n-1$.
\end{lemma}

\begin{proof}
Note that $(x^ig^jy^l)(x^{i'}g^{j'}y^{l'})=\gamma^{j'l}x^{i+i'}g^{j+j'}y^{l+l'}\;(i,i'\in\om,\;j,j'\in\Z,\;l,l'\in n)$. We also prove the lemma through checking their values on the basis $x^ig^jy^l\otimes x^{i'}g^{j'}y^{l'}$.
\begin{eqnarray*}
&& \langle E_1,\gamma^{j'l}x^{i+i'}g^{j+j'}y^{l+l'}\rangle  \\
&=& \delta_{l+l',1}\gamma^{j'l}\langle E_1,x^{i+i'}g^{j+j'}y\rangle
    +\delta_{l+l',n+1}\gamma^{j'l}\langle E_1,x^{i+i'}g^{j+j'}y^{n+1}\rangle  \\
&=& \delta_{l+l',1}\gamma^{j'l}
~=~ \delta_{l,0}\delta_{l',1}+\delta_{l,1}\delta_{l',0}\gamma^{j'}  \\
&=& \langle 1\otimes E_1+E_1\otimes\psi_{1,\gamma},x^ig^jy^l\otimes x^{i'}g^{j'}y^{l'}\rangle,  \\
&& \langle E_2,\gamma^{j'l}x^{i+i'}g^{j+j'}y^{l+l'}\rangle  \\
&=& \gamma^{j'l}(\delta_{l+l',0}\langle E_2,x^{i+i'}g^{j+j'}\rangle
    +\delta_{l+l',n}\langle E_2,x^{i+i'}g^{j+j'}-x^{i+i'}g^{j+j'+n}\rangle)  \\
&=& \delta_{l,0}\delta_{l',0}(\frac{i+i'}{\om}+\frac{j+j'}{n})  \\
& & +\delta_{l+l',n}\gamma^{j'l}(\frac{i+i'}{\om}+\frac{j+j'}{n}-\frac{i+i'}{\om}-\frac{j+j'+n}{n})  \\
&=& \delta_{l,0}\delta_{l',0}(\frac{i+i'}{\om}+\frac{j+j'}{n})
    +\sum_{k=1}^{n-1}\delta_{l,k}\delta_{l',n-k}\gamma^{j'k}(-1)  \\
&=& \delta_{l,0}\delta_{l',0}(\frac{i'}{\om}+\frac{j'}{n})
    +\delta_{l,0}(\frac{i}{\om}+\frac{j}{n})\delta_{l',0}
    -\sum_{k=1}^{n-1}\delta_{l,k}\delta_{l',n-k}\gamma^{j'k}  \\
&=& \langle 1\otimes E_2+E_2\otimes 1-\sum_{k=1}^{n-1}E_1^{[k]}\otimes\psi_{1,\gamma}^kE_1^{[n-k]},
    x^ig^jy^l\otimes x^{i'}g^{j'}y^{l'}\rangle,
\end{eqnarray*}
and
\begin{eqnarray*}
& & \langle\psi_{\la^\frac{1}{\om},\la^\frac{1}{n}},\gamma^{j'l}x^{i+i'}g^{j+j'}y^{l+l'}\rangle  \\
&=& \delta_{l+l',0}\langle\psi_{\la^\frac{1}{\om},\la^\frac{1}{n}},x^{i+i'}g^{j+j'}\rangle
    +\delta_{l+l',n}\gamma^{j'l}
     \langle\psi_{\la^\frac{1}{\om},\la^\frac{1}{n}},x^{i+i'}g^{j+j'}-x^{i+i'}g^{j+j'+n}\rangle)  \\
&=& \delta_{l+l',0}\la^\frac{i+i'}{\om}\la^\frac{j+j'}{n}
    +\delta_{l+l',n}\gamma^{j'l}
     (\la^\frac{i+i'}{\om}\la^\frac{j+j'}{n}-\la^\frac{i+i'}{\om}\la^\frac{j+j'+n}{n})  \\
&=& \delta_{l,0}\delta_{l',0}\la^\frac{i+i'}{\om}\la^\frac{j+j'}{n}
    +\delta_{l+l',n}\gamma^{j'l}(1-\la)\la^\frac{i+i'}{\om}\la^\frac{j+j'}{n}  \\
&=& (\delta_{l,0}\la^\frac{i}{\om}\la^\frac{j}{n})(\delta_{l',0}\la^\frac{i'}{\om}\la^\frac{j'}{n})
    +(1-\la)\sum_{k=1}^{n-1}
     (\delta_{l,k}\la^\frac{i}{\om}\la^\frac{j}{n})
     (\delta_{l',n-k}\la^\frac{i'}{\om}\la^\frac{j'}{n}\gamma^{j'k})  \\
&=& \langle\psi_{\la^\frac{1}{\om},\la^\frac{1}{n}}\otimes\psi_{\la^\frac{1}{\om},\la^\frac{1}{n}}
    +(1-\la)\sum_{k=1}^{n-1}\psi_{\la^\frac{1}{\om},\la^\frac{1}{n}}E_1^{[k]}
     \otimes\psi_{\la^\frac{1}{\om},\la^\frac{1}{n}}\psi_{1,\gamma}^kE_1^{[n-k]},
     x^ig^jy^l\otimes x^{i'}g^{j'}y^{l'}\rangle.
\end{eqnarray*}
The expressions of the counit and the antipode are clear.
\end{proof}

\subsection{The Generation Problem}\label{subsection:B3}

\begin{proposition}\label{prop:B}
As an algebra, $B(n,\om,\gamma)^\circ$ is generated by $\psi_{\la^\frac{1}{\om},\la^\frac{1}{n}}$, $E_1$ and $E_2$ for $\la^\frac{1}{\om},\la^\frac{1}{n}\in\k^\ast$.
\end{proposition}

\begin{proof}
Clearly $B(n,\om,\gamma)$ has a Laurent polynomial subalgebra $P=\k[g^{\pm1}]$. Choose
$$\mathcal{I}_0=\{((g^n-\la)^r)\lhd_l B(n,\om,\gamma) \mid \la\in\k^\ast,\;r\in\N\}.$$
At first for any cofinite left ideal $I$ of $B(n,\om,\gamma)$, according to Lemma \ref{lem:p(x)}(2), there must be some non-zero polynomial
$$p(g)=\prod_{\alpha=1}^N(g^n-\la_\alpha)^{r_\alpha}$$
for some $N\geq1$, $r_\alpha\in\N$ and distinct $\la_\alpha\in\k^\ast$, which generates a cofinite left ideal contained in $I$.
It can be known by Lemma \ref{lem:intersection} that
$$\bigcap_{\alpha=1}^N((g^n-\la_\alpha)^{r_\alpha})= (p(g)) \subseteq I$$
as left ideals of $B(n,\om,\gamma)$,
and thus Lemma \ref{lem:I0} or Corollary \ref{cor:circ}(2) can be applied to obtain that $B(n,\om,\gamma)^\circ=\sum_{I\in\mathcal{I}_0}(B(n,\om,\gamma)/I)^\ast$.

Now we try to prove that for each $I\in\mathcal{I}_0$, the subspace $(B(n,\om,\gamma)/I)^\ast$ can be spanned by some products of
$\psi_{\la^\frac{1}{\om},\la^\frac{1}{n}},E_2,E_1\;(\la^\frac{1}{\om},\la^\frac{1}{n}\in\k^\ast)$.

For the purpose, suppose $I=((g^n-\la)^r)$ for some $r\in\N$ and $\la\in\k^\ast$. Let $\eta$ be a primitive $\omega$th root of $1$, and $\la^\frac{1}{\om},\la^\frac{1}{n}$ be fixed roots of $\la$ respectively. We aim to show that
\begin{equation}\label{eqn:B}
\{\psi_{\la^\frac{1}{\om},\la^\frac{1}{n}}\psi_{\eta,1}^i\psi_{1,\gamma}^jE_2^sE_1^l
  \mid i\in\om,\;j\in n,\;s\in r,\;l\in n\}
\end{equation}
is a linear basis of $\om n^2r$-dimensional space $(B(n,\om,\gamma)/I)^\ast$.

Evidently $I=\k\{x^ig^j(g^n-\la)^ry^l\mid i\in\om,\;j\in\Z,\;l\in n\}$, since $g^n$ is central. Also, $B(n,\om,\gamma)/I$ has a linear basis
$$\{x^ig^j(g^n-\la)^sy^l\mid i\in\om,\;j\in n,\;s\in r,\;l\in n\},$$
and hence $\dim((B(n,\om,\gamma)/I)^\ast)=\dim(B(n,\om,\gamma)/I)=\om n^2r$. Next we show that all elements in (\ref{eqn:B}) vanish on $I$.
Recall that
$$E_1^{[l]}=\frac{1}{l!_\gamma}E_1^l:x^{i'}g^{j'}y^{l'}\mapsto\delta_{l',l}\;\;\;(i',j'\in\Z,\;l,l'\in n).$$
Therefore, for any $s\in r$, $l,l'\in n$, $i'\in\om$ and $j'\in\Z$, and for arbitrary roots $\la^\frac{1}{\om}$ and $\la^\frac{1}{n}$ of $\la$, we have
\begin{eqnarray}\label{eqn:Bvanish}
&& \langle\psi_{\la^\frac{1}{\om},\la^\frac{1}{n}}E_2^sE_1^{[l]},x^{i'}g^{j'}(g^n-\la)^ry^{l'}\rangle  \\
&=& \langle\psi_{\la^\frac{1}{\om},\la^\frac{1}{n}}E_2^sE_1^{[l]},
           \sum_{t=0}^r\binom{r}{t}x^{i'}g^{j'+nt}(-\la)^{r-t}y^{l'}\rangle  \nonumber  \\
&=& \sum_{t=0}^r\binom{r}{t}(-\la)^{r-t}
    \langle\psi_{\la^\frac{1}{\om},\la^\frac{1}{n}}E_2^s\otimes E_1^{[l]},
           \Delta(x^{i'}g^{j'+nt}y^{l'})\rangle  \nonumber  \\
&=& \delta_{l',l}\sum_{t=0}^r\binom{r}{t}(-\la)^{r-t}
    \langle\psi_{\la^\frac{1}{\om},\la^\frac{1}{n}}E_2^s,x^{i'}g^{j'+nt}\rangle
    \langle E_1^{[l]},x^{i'}g^{j'+nt}y^{l}\rangle  \nonumber  \\
&=& \delta_{l',l}\sum_{t=0}^r\binom{r}{t}(-\la)^{r-t}
    \la^\frac{i'}{\om}\la^\frac{j'+nt}{n}(\frac{i'}{\om}+\frac{j'+nt}{n})^s  \nonumber  \\
&=& \delta_{l',l}\la^r\la^\frac{i'}{\om}\la^\frac{j'}{n}
    \sum_{t=0}^r\binom{r}{t}(-1)^{r-t}
    \sum_{u=0}^s\binom{s}{u}(\frac{i'}{\om}+\frac{j'}{n})^ut^{s-u}  \nonumber  \\
&=& \delta_{l',l}\la^r\la^\frac{i'}{\om}\la^\frac{j'}{n}
    \sum_{u=0}^s\binom{s}{u}(\frac{i'}{\om}+\frac{j'}{n})^u
    \sum_{t=0}^r\binom{r}{t}(-1)^{r-t}t^{s-u}  \nonumber  \\
&=& 0,  \nonumber
\end{eqnarray}
since $\sum_{t=0}^r\binom{r}{t}(-1)^{r-t}t^{s-u}=0$ when $s-u<r$ is a part of the second Stirling number (\ref{eqn:2Stirling}).

In other words, $\langle\psi_{\la^\frac{1}{\om},\la^\frac{1}{n}}E_2^sE_1^{[l]},I\rangle=0$ for any $s\in r$, $l\in n$, and arbitrary roots $\la^\frac{1}{\om},\la^\frac{1}{n}$ of $\la$. It follows that the elements in (\ref{eqn:B}) belong to $I^\perp=(B(n,\om,\gamma)/I)^\ast$, as $\la^\frac{1}{\om}\eta^i$ and $\la^\frac{1}{n}\gamma^j$ are still roots of $\la$.

Finally we prove that (\ref{eqn:B}) are linearly independent. Choose a lexicographically ordered set of elements $\{x^{i'}g^{j'+s'n}y^{l'}\in B(n,\om,\gamma)\mid (i',j',s',l')\in \om\times n\times r\times n\}$. Our goal is to show that the $n^2r\times n^2r$ square matrix
$$A:=\left(\langle\psi_{\la^\frac{1}{\om},\la^\frac{1}{n}}\psi_{\eta,1}^i\psi_{1,\gamma}^jE_2^sE_1^{[l]},
           x^{i'}g^{j'+s'n}y^{l'}\rangle
           \mid (i,j,s,l),(i,'j',s',l')\in \om\times n\times r\times n \right)$$
is invertible for fixed $\la^\frac{1}{\om},\la^\frac{1}{n}$, which would imply the linear independence of (\ref{eqn:B}). We try to compute
\begin{eqnarray*}
\langle\psi_{\la^\frac{1}{\om},\la^\frac{1}{n}}\psi_{\eta,1}^i\psi_{1,\gamma}^jE_2^sE_1^{[l]},
  x^{i'}g^{j'+s'n}y^{l'}\rangle
&=& \langle\psi_{\la^\frac{1}{\om}\eta^i,\la^\frac{1}{n}\gamma^j}E_2^s\otimes E_1^{[l]},
    \Delta(x^{i'}g^{j'+s'n}y^{l'})\rangle  \\
&=& \delta_{l',l}\langle\psi_{\la^\frac{1}{\om}\eta^i,\la^\frac{1}{n}\gamma^j}E_2^s,
    x^{i'}g^{j'+s'n}\rangle\langle E_1^{[l]},x^{i'}g^{j'+s'n}y^l\rangle  \\
&=& \delta_{l',l}(\la^\frac{1}{\om}\eta^i)^{i'}(\la^\frac{1}{n}\gamma^j)^{j'+s'n}
    (\frac{i'}{\om}+\frac{j'+s'n}{n})^s  \\
&=& \delta_{l',l}\la^{\frac{i'}{\om}+\frac{j'}{n}+s'}\eta^{ii'}\gamma^{jj'}
    (\frac{i'}{\om}+\frac{j'}{n}+s')^s,
\end{eqnarray*}
and observe that $\det(A)$ equals to
$$(\la^\frac{1}{\om})^{nr\frac{(\om-1)\om}{2}}(\la^\frac{1}{n})^{\om r\frac{(n-1)n}{2}}
  \la^{\om n\frac{(r-1)r}{2}}
  \det\left(\delta_{l',l}\eta^{ii'}\gamma^{jj'}(\frac{i'}{\om}+\frac{j'}{n}+s')^s\right).$$
However, there are following two facts:
\begin{itemize}
  \item The matrix $(\delta_{l',l}\eta^{ii'}\gamma^{jj'}\mid (i,j,l),(i',j',l')\in \om\times n\times n)$ is invertible by Lemma \ref{lem:Kron}, since it is the Kronecker product of three invertible matrices;
  \item For each $(i',j')\in \om\times n$, the matrix
      $((\frac{i'}{\om}+\frac{j'}{n}+s')^s\mid s,s'\in r)$
      is an invertible Vandermonde matrix.
\end{itemize}
Thus
$(\delta_{l',l}\eta^{ii'}\gamma^{jj'}(\frac{i'}{\om}+\frac{j'}{n}+s')^s
  \mid (i,j,s,l),(i,'j',s',l')\in \om\times n\times r\times n)$
is invertible according to Lemma \ref{lem:Kron2}, which follows that $A$ is invertible. As a consequence, (\ref{eqn:B}) is a basis of $(B(n,\om,\gamma)/I)^\ast$.

Therefore,
$$B(n,\om,\gamma)^\circ=\sum_{I\in\mathcal{I}_0}(B(n,\om,\gamma)/I)^\ast
  =\k\{\psi_{\la^\frac{1}{\om},\la^\frac{1}{n}}E_2^sE_1^l
       \mid \la^\frac{1}{\om},\la^\frac{1}{n}\in\k^\ast,\;s\in\N,\;l\in n\}.$$
\end{proof}

\begin{theorem}\label{thm:B}
\begin{itemize}
  \item[(1)] $B_\circ(n,\om,\gamma)$ constructed in Subsection \ref{subsection:B1} is a Hopf algebra;
  \item[(2)] As a Hopf algebra, $B(n,\om,\gamma)^\circ$ is isomorphic to $B_\circ(n,\om,\gamma)$.
\end{itemize}
\end{theorem}

\begin{proof}
Consider the following map
$$\Theta:B_\circ(n,\om,\gamma)\rightarrow B(n,\om,\gamma)^\circ,\;\;
  \Psi_{\la^\frac{1}{\om},\la^\frac{1}{n}}\mapsto\psi_{\la^\frac{1}{\om},\la^\frac{1}{n}},\;
  \;F_1\mapsto E_1,\;F_2\mapsto E_2,$$
where $\la^\frac{1}{\om},\la^\frac{1}{n}\in\k^\ast$. This is an epimorphism of algebras by Lemma \ref{lem:Balg} and Proposition \ref{prop:B}. Furthermore, $\Theta$ would become an isomorphism of Hopf algebras with desired coalgebra structure and antipode, as long as it is injective (since $B(n,\om,\gamma)^\circ$ is in fact a Hopf algebra).

In order to show that $\Theta$ is injective, we aim to show the linear independence of
$$\{\psi_{\la^\frac{1}{\om},\la^\frac{1}{n}}\psi_{\eta,1}^i\psi_{1,\gamma}^jE_2^sE_1^l
  \mid \la\in\k^\ast,\;i\in\om,\;j\in n,\;s\in\N,\;l\in n\}$$
in $B(n,\om,\gamma)^\circ$, where $\la^\frac{1}{\om}$ and $\la^\frac{1}{n}$ are fixed roots of each $\la\in\k^\ast$. By linear independence of elements in \eqref{eqn:B}, we only need to show that any finite sum of form
\begin{equation}\label{eqn:Binj}
\sum_{\alpha=1}^N(B(n,\om,\gamma)/((g^n-\la_\alpha)^r))^\ast
=\sum_{\alpha=1}^N((g^n-\la_\alpha)^r))^\perp
\end{equation}
is direct, as long as $\la_\alpha$'s are distinct from each other. But this is clear:
\begin{eqnarray}\label{eqn:Binj2}
& &((g^n-\la_{\alpha'})^r)^\perp\cap[\sum_{\alpha\neq\alpha'}((g^n-\la_\alpha)^r)^\perp]  \\
&=& ((g^n-\la_{\alpha'})^r)^\perp
    \cap[\bigcap_{\alpha\neq\alpha'}((g^n-\la_\alpha)^r)]^\perp  \nonumber \\
&=& ((g^n-\la_{\alpha'})^r)^\perp
    \cap(\prod_{\alpha\neq\alpha'}(g^n-\la_\alpha)^r)^\perp
~=~ [((g^n-\la_{\alpha'})^r)+(\prod_{\alpha\neq\alpha'}(g^n-\la_\alpha)^r))]^\perp  \nonumber  \\
&=& B(n,\om,\gamma)^\perp
~=~ 0.  \nonumber
\end{eqnarray}
\end{proof}

\begin{remark}
\emph{As a special case, the connected algebraic groups of dimension one $H_2=\k[g^{\pm1}]$ is equal to $B(1,1,1)$. Therefore, Theorem \ref{thm:B} gives ${H_2}^\circ$ which is well-known too (see \cite[Example 9.2.7]{Mon93} for example).}
\end{remark}

\begin{remark}
\emph{
Let us contrast the structure of $B(n,\omega,\gamma)^\circ$ with the Hopf algebra introduced in \cite{ACE15}: Choose a set
$I:=\{(\alpha,\beta)\in\k^\ast\times\k^\ast\mid \alpha^\omega=\beta^n\}$. Then a particular case of the Hopf algebra $$\mathfrak{D}:=\mathfrak{D}(n,\gamma,(\la^\frac{1}{n})_{(\la^\frac{1}{\om},\la^\frac{1}{n})\in I},1)$$
defined in \cite[Theorem 3.4]{ACE15} is generated by
$u,\;x,\;a_{\la^\frac{1}{\om},\la^\frac{1}{n}}^{\pm1}\;\;(\la^\frac{1}{\om},\la^\frac{1}{n}\in \k^\ast)$ with relations
\begin{align*}
& u^n=1,\;\;x^n=0,\;\;
a_{\la^\frac{1}{\om},\la^\frac{1}{n}}^{\pm1}a_{\la^\frac{1}{\om},\la^\frac{1}{n}}^{\mp1}=1,
\;\;ux=\gamma xu,\;\;  \\
& ua_{\la^\frac{1}{\om},\la^\frac{1}{n}}=a_{\la^\frac{1}{\om},\la^\frac{1}{n}}u,\;\;
a_{\la^\frac{1}{\om},\la^\frac{1}{n}}x=\la^\frac{1}{n}xa_{\la^\frac{1}{\om},\la^\frac{1}{n}},\;\;  \\
& a_{\la_1^\frac{1}{\om},\la_1^\frac{1}{n}}a_{\la_2^\frac{1}{\om},\la_2^\frac{1}{n}}
=a_{\la_2^\frac{1}{\om},\la_2^\frac{1}{n}}a_{\la_1^\frac{1}{\om},\la_1^\frac{1}{n}}
\end{align*}
as an algebra for $\la^\frac{1}{\om},\la^\frac{1}{n},\la_1^\frac{1}{\om},\la_1^\frac{1}{n},
\la_2^\frac{1}{\om},\la_2^\frac{1}{n}\in\k^\ast$. The comultiplication, counit and antipode on $\mathfrak{D}$ is defined by
\begin{align*}
& \Delta(u)=u\otimes u,\;\;\Delta(x)=u\otimes x+x\otimes 1,  \\
& \Delta(a_{\la^\frac{1}{\om},\la^\frac{1}{n}}^{\pm1})
=(1\otimes 1
+(1-\la^{\pm1})\sum_{k=1}^{n-1}
  \frac{1}{k!_\gamma(n-k)!_\gamma}x^{n-k}u^k\otimes x^k)
  (a_{\la^\frac{1}{\om},\la^\frac{1}{n}}^{\pm1}\otimes a_{\la^\frac{1}{\om},\la^\frac{1}{n}}^{\pm1})  \\
& \varepsilon(u)=1,\;\;\varepsilon(x)=0,\;\;
\varepsilon(a_{\la^\frac{1}{\om},\la^\frac{1}{n}}^{\pm1})=1,\;\;  \\
& S(u)=u^{n-1},\;\;S(x)=-u^{n-1}x,\;\;
S(a_{\la^\frac{1}{\om},\la^\frac{1}{n}}^{\pm1})=a_{\la^\frac{1}{\om},\la^\frac{1}{n}}^{\mp1}
\end{align*}
for $\la^\frac{1}{\om},\la^\frac{1}{n}\in\k^\ast$.
Denote by $\overline{\mathfrak{D}}$ the quotient Hopf algebra of $\mathfrak{D}$ modulo the ideal generated by
$$
a_{1,1}-1,\;\;u-a_{1,\gamma},\;\;
a_{\la_1^\frac{1}{\om},\la_1^\frac{1}{n}}a_{\la_2^\frac{1}{\om},\la_2^\frac{1}{n}}
-a_{\la_1^\frac{1}{\om}\la_2^\frac{1}{\om},\la_1^\frac{1}{n}\la_2^\frac{1}{n}}\;\;
(\la_1^\frac{1}{\om},\la_1^\frac{1}{n},\la_2^\frac{1}{\om},\la_2^\frac{1}{n}\in\k^\ast).
$$
It can be found that $\overline{\mathfrak{D}}^{\mathrm{op\;cop}}$ is isomorphic to the Hopf subalgebra of $B(n,\omega,\gamma)^\circ$ generated by
$\psi_{\la^\frac{1}{\om},\la^\frac{1}{n}},\;E_1\;(\la^\frac{1}{\om},\la^\frac{1}{n}\in\k^\ast)$ via the isomorphism
$$\psi_{\la^\frac{1}{\om},\la^\frac{1}{n}}\mapsto a_{\la^\frac{1}{\om},\la^\frac{1}{n}},\;\;
x\mapsto E_1\;\;(\la^\frac{1}{\om},\la^\frac{1}{n}\in\k^\ast).$$
}
\end{remark}

\section{The Finite Dual of $D(m,d,\xi)$}\label{section5}

We apply the same way used Sections \ref{section3} and \ref{section4} to give structure of $D(m,d,\xi)^\circ$ in this section.

\subsection{The Hopf Algebra $D_\circ(m,d,\xi)$}\label{subsection:D1}

Let $m$ and $d$ be positive integers such that $(1+m)d$ is even, and $\xi$ be a primitive $2m$th root of $1$. Define
$$\om:=md,\;\;\gamma:=\xi^2.$$
The Hopf algebra $D_\circ(m,d,\xi)$ is constructed as follows. As an algebra, it is generated by
$$\Zeta_{\la^\frac{1}{\om},\la^\frac{1}{m}},\;\Chi_{\la^\frac{1}{\om},\la^\frac{1}{m}},\;
  F_1,\;F_2,\;\;(\la^\frac{1}{\om},\la^\frac{1}{m}\in\k^\ast)$$
(where the parameters of $\Zeta$ and $\Chi$ are meant to range through
$\{(\alpha,\beta)\in\k^\ast\times\k^\ast\mid \alpha^\omega=\beta^m\}$)
with relations
\begin{align*}
& \Zeta_{\la_1^\frac{1}{\om},\la_1^\frac{1}{m}}\Zeta_{\la_2^\frac{1}{\om},\la_2^\frac{1}{m}}
  =\Zeta_{(\la_1^\frac{1}{\om}\la_2^\frac{1}{\om}),(\la_1^\frac{1}{m}\la_2^\frac{1}{m})},\;\;
  \Chi_{\la_1^\frac{1}{\om},\la_1^\frac{1}{m}}\Chi_{\la_2^\frac{1}{\om},\la_2^\frac{1}{m}}
  =\Chi_{(\la_1^\frac{1}{\om}\la_2^\frac{1}{\om}),(\la_1^\frac{1}{m}\la_2^\frac{1}{m})},  \\
& \Zeta_{\la_1^\frac{1}{\om},\la_1^\frac{1}{m}}\Chi_{\la_2^\frac{1}{\om},\la_2^\frac{1}{m}}
  =\Chi_{\la_1^\frac{1}{\om},\la_1^\frac{1}{m}}\Zeta_{\la_2^\frac{1}{\om},\la_2^\frac{1}{m}}=0,\;\;
  \Zeta_{1,1}+\Chi_{1,1}=1,\;\;
  F_1^m=\frac{1}{(1-\gamma)^m}\Chi_{1,1},  \\
& F_2\Zeta_{\la^\frac{1}{\om},\la^\frac{1}{m}}=\Zeta_{\la^\frac{1}{\om},\la^\frac{1}{m}}F_2,\;\;
  F_1\Zeta_{\la^\frac{1}{\om},\la^\frac{1}{m}}
  =\la^\frac{1}{m}\Zeta_{\la^\frac{1}{\om},\la^\frac{1}{m}}F_1,  \\
& F_2\Chi_{\la^\frac{1}{\om},\la^\frac{1}{m}}
  =\Chi_{\la^\frac{1}{\om},\la^\frac{1}{m}}F_2,\;\;
  F_1\Chi_{\la^\frac{1}{\om},\la^\frac{1}{m}}
  =\la^\frac{-d}{\om}\la^\frac{1}{m}\Chi_{\la^\frac{1}{\om},\la^\frac{1}{m}}F_1,  \\
& F_1F_2=F_2F_1+\frac{1}{m}\Zeta_{1,1}F_1
\end{align*}
for $\la^\frac{1}{\om},\la^\frac{1}{m},\la_1^\frac{1}{\om},\la_1^\frac{1}{m},
\la_2^\frac{1}{\om},\la_2^\frac{1}{m}\in\k^\ast$. Denote $F_1^{[k]}:=\frac{1}{k!_\gamma}F_1^k$ for $1\leq k\leq m-1$. The comultiplication is given by:
\begin{align*}
&\Delta(F_1)=1\otimes F_1+F_1\otimes(\Zeta_{1,\gamma}+\xi\Chi_{1,\gamma}),  \\
&\Delta(F_2)
  =(\Zeta_{1,1}-\Chi_{1,1})\otimes F_2+F_2\otimes 1
   -\sum_{k=1}^{m-1}
    (\Zeta_{1,1}-\Chi_{1,1})F_1^{[k]}\otimes(\Zeta_{1,\gamma}+\xi\Chi_{1,\gamma})^{-m+k}
    F_1^{[m-k]},\\
&\Delta(\Zeta_{\la^\frac{1}{\om},\la^\frac{d}{\om}})
= \Zeta_{\la^\frac{1}{\om},\la^\frac{d}{\om}}\otimes\Zeta_{\la^\frac{1}{\om},\la^\frac{d}{\om}}
    +(1-\la)\sum_{k=1}^{m-1}\Zeta_{\la^\frac{1}{\om},\la^\frac{d}{\om}}F_1^{[k]}
                \otimes\Zeta_{\la^\frac{1}{\om},\la^\frac{d}{\om}}\Zeta_{1,\gamma}^kF_1^{[m-k]}  \\
&\;\;\;\;\;\;\;\;\;\;\;\;\;\;\;\;\;\;\;\;\;\;+\la^\frac{(1-m)d/2}{\om}(1-\la)(
     \ta_0^{-1}\Chi_{\la^\frac{1}{\om},\la^\frac{d}{\om}}\otimes
     \Chi_{\la^\frac{-1}{\om},\la^\frac{-d}{\om}}\\
    &\;\;\;\;\;\;\;\;\;\;\;\;\;\;\;\;\;\;\;\;\;\;\;
    \;\;\;\;\;\;\;\;\;\;\;\;\;\;\;\;\;\;\;\;\;
    \;\;\;\;\;\;\;\;+\sum_{k=1}^{m-1}\ta_{m-k}^{-1}
      \Chi_{\la^\frac{1}{\om},\la^\frac{d}{\om}}F_1^{[k]}
      \otimes\Chi_{\la^\frac{-1}{\om},\la^\frac{-d}{\om}}\xi^k
      \Chi_{1,\gamma}^kF_1^{[m-k]})\\
&\Delta(\Chi_{\la^\frac{1}{\om},\la^\frac{d}{\om}})= \Zeta_{\la^\frac{1}{\om},\la^\frac{d}{\om}}\otimes\Chi_{\la^\frac{1}{\om},\la^\frac{d}{\om}}
      -\ta_0\sum_{k=1}^{m-1}\ta_1\cdots\ta_{k-1}
         \Zeta_{\la^\frac{1}{\om},\la^\frac{d}{\om}}E_1^{[k]}
           \otimes\Chi_{\la^\frac{1}{\om},\la^\frac{d}{\om}}\xi^k\Chi_{1,\gamma}^kF_1^{[m-k]}  \\& \;\;\;\;\;\;\;\;\;\;\;\;\;\;\;\;\;\;\;\;\;\; +\Chi_{\la^\frac{1}{\om},\la^\frac{d}{\om}}\otimes\Zeta_{\la^\frac{-1}{\om},\la^\frac{-d}{\om}}  \\
&\;\;\;\;\;\;\;\;\;\;\;\;\;\;\;\;\;\;\;\;\;\;  -\ta_0\sum_{k=1}^{m-1}\la^\frac{-(m-k)d}{\om}\ta_1\cdots\ta_{m-k-1}
           \Chi_{\la^\frac{1}{\om},\la^\frac{d}{\om}}E_1^{[k]}
             \otimes\Zeta_{\la^\frac{-1}{\om},\la^\frac{-d}{\om}}\Zeta_{1,\gamma}^kF_1^{[m-k]}.
\end{align*}
where $\la^\frac{1}{\om}\in\k^\ast$ and $\theta_0=\frac{1-\la^\frac{d}{\om}}{1/m},\;
  \theta_k=\frac{1-\gamma^k\la^\frac{d}{\om}}{1-\gamma^k}\;(1\leq k\leq m-1).$ Note that $\theta_0\theta_1\cdots\theta_{m-1}=1-\la$ holds (see \cite[Proposition IV.2.7]{Kas95}) and thus $(1-\la)\theta_k^{-1}$ for $0\leq k\leq m-1$ is well-defined.
The coproducts on $\Zeta_{\lambda^\frac{1}{\om},\lambda^\frac{1}{m}}$ and $\Chi_{\lambda^\frac{1}{\om},\lambda^\frac{1}{m}}$ for arbitrary $\lambda^\frac{1}{\om}$ and $\lambda^\frac{1}{m}$ are given by $$\Delta(\Zeta_{\lambda^\frac{1}{\om},\lambda^\frac{1}{m}})=\Delta(\Zeta_{\lambda^\frac{1}{\om},\lambda^\frac{d}{\om}})
\Delta(\Zeta_{1,\gamma}+\xi\Chi_{1,\gamma})^k,$$
$$\Delta(\Chi_{\lambda^\frac{1}{\om},\lambda^\frac{1}{m}})=\xi^{-k}\Delta(\Chi_{\lambda^\frac{1}{\om},\lambda^\frac{d}{\om}})
\Delta(\Zeta_{1,\gamma}+\xi\Chi_{1,\gamma})^k,$$ where $k$ is a non-negative integer such that $\la^\frac{1}{m}=\la^\frac{d}{\om}\gamma^k$ and $\Zeta_{1,\gamma}+\xi\Chi_{1,\gamma}$ is defined to be a group-like element.

The counit is given by
\begin{align*}
& \varepsilon(\Zeta_{\lambda^\frac{1}{\om},\lambda^\frac{1}{m}})=1,\;\;
\varepsilon(\Chi_{\lambda^\frac{1}{\om},\lambda^\frac{1}{m}})=0,\;\;
\varepsilon(F_1)=\varepsilon(F_2)=0.
\end{align*}
The antipode is given by
\begin{align*}
& S(F_1)=-\gamma^{-1}(\Zeta_{1,\gamma^{-1}}+\xi^{-1}\Chi_{1,\gamma^{-1}})F_1,  \\
& S(F_2)=-\Zeta_{1,1}F_2+\Chi_{1,1}F_2+\frac{1-m}{2m}\Chi_{1,1},  \\
& S(\Zeta_{\la^\frac{1}{\om},\la^\frac{1}{m}})=\Zeta_{\la^\frac{-1}{\om},\la^\frac{-1}{m}},  \\
& S(\Chi_{\la^\frac{1}{\om},\la^\frac{1}{m}})
  =\la^\frac{(1-m)d/2}{\om}\gamma^{-k}\Chi_{\la^\frac{1}{\om},\la^\frac{d}{\om}\gamma^{-k}}\;\;
   \text{when}\;\;\la^\frac{1}{m}=\la^\frac{d}{\om}\gamma^k.
\end{align*}

The fact that $D_\circ(m,d,\xi)$ is a Hopf algebra would be stated and proved in Subsection \ref{subsection:D3} as Theorem \ref{thm:D}(1). Let $\eta$ be an primitive $\om$th root of $1$. Here we note that
\begin{align*}
& \{\Zeta_{\la^\frac{1}{\om},\la^\frac{1}{m}}\Zeta_{\eta,1}^i\Zeta_{1,\gamma}^jF_2^sF_1^l,\;
    \Chi_{\la^\frac{1}{\om},\la^\frac{1}{m}}\Chi_{\eta,1}^i\Chi_{1,\gamma}^jF_2^sF_1^l
       \\
& \;\;\mid  \la^\frac{1}{\om},\la^\frac{1}{n}~\text{are fixed roots of}~\la\in\k^\ast,\;
    i\in\om,\;j\in m,\;s\in\N,\;l\in m\}
\end{align*}
is a linear basis, due to an application of the Diamond Lemma \cite{Ber78}.

\begin{remark}\label{rmk:cocyletrivial}
\emph{It is asked in \cite[Remarks 7.1(3)]{BCJ} whether the cocycle $\sigma$ non-trivial in the formula of \cite[Proposition 7.2(V)]{BCJ}:
$$D(m,d,\xi)^\circ=(\k C_2\#_\sigma T_f(m,1,\xi^2))\# (\k[f]\otimes\k(\k^\times)).$$
Our results can be used to find that $\sigma$ is trivial in fact.}
\end{remark}

\subsection{Certain Elements in $D(m,d,\xi)^\circ$}\label{subsection:D2}

Let $m,d$ be positive integers such that $(1+m)d$ is even and $\xi$ a primitive $2m$th root of unity. Define
$$\om:=md,\;\; \gamma:=\xi^2.$$
According to the definition of $D(m,d,\xi)$ recalled in Definition \ref{def:D}, we know that it has a linear basis $\{x^ig^jy^l,\;x^ig^ju_l\mid i\in\om,\;j\in\Z,\;l\in m\}$ by \cite[Lemma 3.3]{Wu16} and \cite[Equation 4.7]{WLD16}. Define following elements in $D(m,d,\xi)^\ast$:
\begin{align}\label{eqn:Dgenerator}
& \zeta_{\lambda^\frac{1}{\om},\lambda^\frac{1}{m}}:\left\{\begin{array}{ll}
    x^ig^jy^l\mapsto \delta_{l,0}\lambda^\frac{i}{\om}\lambda^\frac{j}{m}  \\  x^ig^ju_l\mapsto 0
  \end{array}\right.,\;\;
  \chi_{\lambda^\frac{1}{\om},\lambda^\frac{1}{m}}:\left\{\begin{array}{ll}
    x^ig^jy^l\mapsto 0  \\  x^ig^ju_l\mapsto \delta_{l,0}\lambda^\frac{i}{\om}\lambda^\frac{j}{m}
  \end{array}\right.,  \nonumber  \\
& E_1:\left\{\begin{array}{ll}
    x^ig^jy^l\mapsto \delta_{l,1}  \\ x^ig^ju_l\mapsto \frac{\xi}{1-\gamma^{-1}}\delta_{l,1}
  \end{array}\right.,\;\;
  E_2:\left\{\begin{array}{ll}
    x^ig^jy^l\mapsto \delta_{l,0}(\frac{i}{\om}+\frac{j}{m})  \\
    x^ig^ju_l\mapsto \delta_{l,0}(\frac{i}{\om}+\frac{j}{m})
  \end{array}\right.,
\end{align}
for any $i\in\om$, $j\in\Z$, $l\in m$ and $\la^\frac{1}{\om},\la^\frac{1}{m}\in\k^\ast$. We remark that these definitions make sense for all $i\in\Z$ as well, due to direct computations.

One can verify that $\zeta_{\la^\frac{1}{\om},\la^\frac{1}{m}}$ and $\chi_{\la^\frac{1}{\om},\la^\frac{1}{m}}$ both vanish on the principal left ideal $((g^m-\la)(g^m-\la^{-1}))$ for any $\la^\frac{1}{\om},\la^\frac{1}{m}\in\k^\ast$; $E_2$ vanishes on $((g^m-1)^2)$; $E_1$ vanishes on $(g^m-1)$. Therefore, these elements belong to $D(m,d,\xi)^\circ$.

\begin{lemma}\label{lem:Dalg}
Following equations hold in $D(m,d,\xi)^\circ$:
\begin{align*}
& \zeta_{\la_1^\frac{1}{\om},\la_1^\frac{1}{m}}\zeta_{\la_2^\frac{1}{\om},\la_2^\frac{1}{m}}
  =\zeta_{(\la_1^\frac{1}{\om}\la_2^\frac{1}{\om}),(\la_1^\frac{1}{m}\la_2^\frac{1}{m})},\;\;
  \chi_{\la_1^\frac{1}{\om},\la_1^\frac{1}{m}}\chi_{\la_2^\frac{1}{\om},\la_2^\frac{1}{m}}
  =\chi_{(\la_1^\frac{1}{\om}\la_2^\frac{1}{\om}),(\la_1^\frac{1}{m}\la_2^\frac{1}{m})},  \\
& \zeta_{\la_1^\frac{1}{\om},\la_1^\frac{1}{m}}\chi_{\la_2^\frac{1}{\om},\la_2^\frac{1}{m}}
  =\chi_{\la_1^\frac{1}{\om},\la_1^\frac{1}{m}}\zeta_{\la_2^\frac{1}{\om},\la_2^\frac{1}{m}}=0,\;\;
  \zeta_{1,1}+\chi_{1,1}=1,\;\;
  E_1^m=\frac{1}{(1-\gamma)^m}\chi_{1,1},  \\
& E_2\zeta_{\la^\frac{1}{\om},\la^\frac{1}{m}}=\zeta_{\la^\frac{1}{\om},\la^\frac{1}{m}}E_2,\;\;
  E_1\zeta_{\la^\frac{1}{\om},\la^\frac{1}{m}}
  =\la^\frac{1}{m}\zeta_{\la^\frac{1}{\om},\la^\frac{1}{m}}E_1,  \\
& E_2\chi_{\la^\frac{1}{\om},\la^\frac{1}{m}}
  =\chi_{\la^\frac{1}{\om},\la^\frac{1}{m}}E_2,\;\;
  E_1\chi_{\la^\frac{1}{\om},\la^\frac{1}{m}}
  =\la^\frac{-d}{\om}\la^\frac{1}{m}\chi_{\la^\frac{1}{\om},\la^\frac{1}{m}}E_1,  \\
& E_1E_2=E_2E_1+\frac{1}{m}\zeta_{1,1}E_1
\end{align*}
for all $\la^\frac{1}{\om},\la^\frac{1}{m},\la_1^\frac{1}{\om},\la_1^\frac{1}{m},
\la_2^\frac{1}{\om},\la_2^\frac{1}{m}\in\k^\ast$.
\end{lemma}

\begin{proof}
Note that
\begin{align*}
& \Delta(x^ig^jy^l)=\sum_{k=0}^l\binom{l}{k}_\gamma x^ig^jy^k\otimes x^ig^{j+k}y^{l-k},  \\
& \Delta(x^ig^ju_l)=\sum_{k=0}^{m-1}\gamma^{k(l-k)}x^ig^ju_k\otimes x^{i-kd}g^{j+k}u_{l-k}.
\end{align*}
We prove the lemma through checking their values on the basis
$x^ig^jy^l,\;x^ig^ju_l\;\;(i\in\om,\;j\in\Z,\;l\in m)$:
\begin{eqnarray*}
\langle\zeta_{\la_1^\frac{1}{\om},\la_1^\frac{1}{m}}\zeta_{\la_2^\frac{1}{\om},\la_2^\frac{1}{m}},
       x^ig^jy^l\rangle
&=& \delta_{l,0}\langle\zeta_{\la_1^\frac{1}{\om},\la_1^\frac{1}{m}},x^ig^j\rangle
    \langle\zeta_{\la_2^\frac{1}{\om},\la_2^\frac{1}{m}},x^ig^j\rangle
~=~ \delta_{l,0}\la_1^\frac{i}{\om}\la_2^\frac{i}{\om}\la_1^\frac{j}{m}\la_2^\frac{j}{m}  \\
&=& \langle\zeta_{\la_1^\frac{1}{\om}\la_2^\frac{1}{\om},\la_1^\frac{1}{m}\la_2^\frac{1}{m}},
           x^ig^jy^l\rangle,  \\
\langle\zeta_{\la_1^\frac{1}{\om},\la_1^\frac{1}{m}}\zeta_{\la_2^\frac{1}{\om},\la_2^\frac{1}{m}},
       x^ig^ju_l\rangle
&=& 0
~=~ \langle\zeta_{\la_1^\frac{1}{\om}\la_2^\frac{1}{\om},\la_1^\frac{1}{m}\la_2^\frac{1}{m}},
           x^ig^ju_l\rangle,  \\
\langle\chi_{\la_1^\frac{1}{\om},\la_1^\frac{1}{m}}\chi_{\la_2^\frac{1}{\om},\la_2^\frac{1}{m}},
       x^ig^jy^l\rangle
&=& 0
~=~ \langle\chi_{\la_1^\frac{1}{\om}\la_2^\frac{1}{\om},\la_1^\frac{1}{m}\la_2^\frac{1}{m}},
           x^ig^jy^l\rangle,  \\
\langle\chi_{\la_1^\frac{1}{\om},\la_1^\frac{1}{m}}\chi_{\la_2^\frac{1}{\om},\la_2^\frac{1}{m}},
       x^ig^ju_l\rangle
&=& \delta_{l,0}\langle\chi_{\la_1^\frac{1}{\om},\la_1^\frac{1}{m}},x^ig^ju_0\rangle
    \langle\chi_{\la_2^\frac{1}{\om},\la_2^\frac{1}{m}},x^ig^ju_0\rangle
~=~ \delta_{l,0}\la_1^\frac{i}{\om}\la_2^\frac{i}{\om}\la_1^\frac{j}{m}\la_2^\frac{j}{m}  \\
&=& \langle\zeta_{\la_1^\frac{1}{\om}\la_2^\frac{1}{\om},\la_1^\frac{1}{m}\la_2^\frac{1}{m}},
           x^ig^ju_l\rangle,  \\
\langle\zeta_{\la_1^\frac{1}{\om},\la_1^\frac{1}{m}}\chi_{\la_2^\frac{1}{\om},\la_2^\frac{1}{m}},
       x^ig^jy^l\rangle
&=& \langle\zeta_{\la_1^\frac{1}{\om},\la_1^\frac{1}{m}}\chi_{\la_2^\frac{1}{\om},\la_2^\frac{1}{m}},
       x^ig^ju_0\rangle
~=~ 0,  \\
\langle\chi_{\la_1^\frac{1}{\om},\la_1^\frac{1}{m}}\zeta_{\la_2^\frac{1}{\om},\la_2^\frac{1}{m}},
       x^ig^jy^l\rangle
&=& \langle\chi_{\la_1^\frac{1}{\om},\la_1^\frac{1}{m}}\zeta_{\la_2^\frac{1}{\om},\la_2^\frac{1}{m}},
       x^ig^ju_0\rangle
~=~ 0,  \\
\langle\zeta_{1,1}+\chi_{1,1},x^ig^jy^l\rangle
&=& \delta_{l,0}1^i1^j
~=~ \langle\varepsilon,x^ig^jy^l\rangle,  \\
\langle\zeta_{1,1}+\chi_{1,1},x^ig^ju_l\rangle
&=& \delta_{l,0}1^i1^j
~=~ \langle\varepsilon,x^ig^ju_l\rangle,  \\
\langle E_2\zeta_{\la^\frac{1}{\om},\la^\frac{1}{m}},x^ig^jy^l\rangle
&=& \delta_{l,0}\langle E_2,x^ig^j\rangle\langle\zeta_{\la^\frac{1}{\om},\la^\frac{1}{m}},x^ig^j\rangle
~=~ \delta_{l,0}(\frac{i}{\om}+\frac{j}{m})\la^\frac{i}{\om}\la^\frac{j}{m}  \\
&=& \delta_{l,0}\langle\zeta_{\la^\frac{1}{\om},\la^\frac{1}{m}},x^ig^j\rangle\langle E_2,x^ig^j\rangle
~=~ \langle\zeta_{\la^\frac{1}{\om},\la^\frac{1}{m}}E_2,x^ig^jy^l\rangle,  \\
\langle E_2\zeta_{\la^\frac{1}{\om},\la^\frac{1}{m}},x^ig^ju_l\rangle
&=& 0
~=~ \langle\zeta_{\la^\frac{1}{\om},\la^\frac{1}{m}}E_2,x^ig^ju_l\rangle,  \\
\langle E_2\chi_{\la^\frac{1}{\om},\la^\frac{1}{m}},x^ig^jy^l\rangle
&=& 0
~=~ \langle\chi_{\la^\frac{1}{\om},\la^\frac{1}{m}}E_2,x^ig^jy^l\rangle,  \\
\langle E_2\chi_{\la^\frac{1}{\om},\la^\frac{1}{m}},x^ig^ju_l\rangle
&=& \delta_{l,0}\langle E_2,x^ig^ju_0\rangle
    \langle\zeta_{\la^\frac{1}{\om},\la^\frac{1}{m}},x^ig^ju_0\rangle
~=~ \delta_{l,0}(\frac{i}{\om}+\frac{j}{m})\la^\frac{i}{\om}\la^\frac{j}{m}  \\
&=& \delta_{l,0}\langle\zeta_{\la^\frac{1}{\om},\la^\frac{1}{m}},x^ig^ju_0\rangle
    \langle E_2,x^ig^ju_0\rangle
~=~ \langle\zeta_{\la^\frac{1}{\om},\la^\frac{1}{m}}E_2,x^ig^ju_l\rangle,  \\
\langle E_1\zeta_{\la^\frac{1}{\om},\la^\frac{1}{m}},x^ig^jy^l\rangle
&=& \delta_{l,1}\langle E_1,x^ig^jy\rangle
    \langle\zeta_{\la^\frac{1}{\om},\la^\frac{1}{m}},x^ig^{j+1}\rangle
~=~ \delta_{l,1}\la^\frac{i}{\om}\la^\frac{j+1}{m}  \\
&=& \la^\frac{1}{m}\delta_{l,1}\la^\frac{i}{\om}\la^\frac{j}{m}
~=~ \la^\frac{1}{m}\delta_{l,1}\langle\zeta_{\la^\frac{1}{\om},\la^\frac{1}{m}},x^ig^j\rangle
    \langle E_1,x^ig^jy\rangle  \\
&=& \langle\la^\frac{1}{m}\zeta_{\la^\frac{1}{\om},\la^\frac{1}{m}}E_1,x^ig^jy^l\rangle,  \\
\langle E_1\zeta_{\la^\frac{1}{\om},\la^\frac{1}{m}},x^ig^ju_l\rangle
&=& 0
~=~ \langle \la^\frac{1}{m}\zeta_{\la^\frac{1}{\om},\la^\frac{1}{m}}E_1,x^ig^ju_l\rangle,  \\
\langle E_1\chi_{\la^\frac{1}{\om},\la^\frac{1}{m}},x^ig^jy^l\rangle
&=& 0
~=~ \langle\la^\frac{-d}{\om}\la^\frac{1}{m}\chi_{\la^\frac{1}{\om},\la^\frac{1}{m}}E_1,
           x^ig^jy^l\rangle, \\
\langle E_1\chi_{\la^\frac{1}{\om},\la^\frac{1}{m}},x^ig^ju_l\rangle
&=& \delta_{l,1}\langle E_1,x^ig^ju_1\rangle
    \langle\chi_{\la^\frac{1}{\om},\la^\frac{1}{m}},x^{i-d}g^{j+1}u_0\rangle  \\
&=& \delta_{l,1}\frac{\xi}{1-\gamma^{-1}}\la^\frac{i-d}{\om}\la^\frac{j+1}{m}  \\
&=& \la^\frac{-d}{\om}\la^\frac{1}{m}
    \delta_{l,1}\la^\frac{i}{\om}\la^\frac{j}{m}\frac{\xi}{1-\gamma^{-1}}  \\
&=& \la^\frac{-d}{\om}\la^\frac{1}{m}\delta_{l,1}
    \langle\chi_{\la^\frac{1}{\om},\la^\frac{1}{m}},x^ig^ju_0\rangle\langle E_1,x^ig^ju_1\rangle  \\
&=& \langle\la^\frac{-d}{\om}\la^\frac{1}{m}\chi_{\la^\frac{1}{\om},\la^\frac{1}{m}}E_1,
           x^ig^ju_l\rangle, \\
\langle E_1E_2,x^ig^jy^l\rangle
&=& \delta_{l,1}\langle E_1,x^ig^jy\rangle\langle E_2,x^ig^{j+1}\rangle
~=~ \delta_{l,1}(\frac{i}{\om}+\frac{j+1}{m})  \\
&=& \delta_{l,1}(\frac{i}{\om}+\frac{j}{m})+\delta_{l,1}\frac{1}{m}  \\
&=& \delta_{l,1}\langle E_2,x^ig^j\rangle\langle E_1,x^ig^jy\rangle
    +\delta_{l,1}\frac{1}{m}\langle E_1,x^ig^jy\rangle  \\
&=& \langle E_2E_1+\frac{1}{m}E_1,x^ig^jy^l\rangle
~=~ \langle E_2E_1+\frac{1}{m}\zeta_{1,1}E_1,x^ig^jy^l\rangle,  \\
\langle E_1E_2,x^ig^ju_l\rangle
&=& \delta_{l,1}\langle E_1,x^ig^ju_1\rangle\langle E_2,x^{i-d}g^{j+1}u_0\rangle  \\
&=& \delta_{l,1}\frac{\xi}{1-\gamma^{-1}}(\frac{i-d}{\om}+\frac{j+1}{m})  \\
&=& \delta_{l,1}(\frac{i}{\om}+\frac{j}{m})\frac{\xi}{1-\gamma^{-1}}
~=~ \delta_{l,1}\langle E_2,x^ig^ju_0\rangle\langle E_1,x^ig^ju_1\rangle  \\
&=& \langle E_2E_1,x^ig^ju_l\rangle
~=~ \langle E_2E_1+\frac{1}{m}\zeta_{1,1}E_1,x^ig^ju_l\rangle,
\end{eqnarray*}
Similarly to the case in Lemmas \ref{lem:Talg} and \ref{lem:Balg}, it can be found by induction on $k\in n$ that $\langle E_1^k,x^ig^jy^l\rangle=\delta_{l,k}k!_\gamma$ and
\begin{eqnarray*}
\langle E_1^k,x^ig^ju_l\rangle
&=& \delta_{l,k}\gamma^{k-1}\langle E_1,x^ig^ju_1\rangle\langle E_1^{k-1},x^{i-d}g^{j+1}u_{k-1}\rangle  \\  &=& \delta_{l,k}\gamma^{k-1}\frac{\xi}{1-\gamma^{-1}}\langle E_1^{k-1},x^{i-d}g^{j+1}u_{k-1}\rangle  \\
&=& \delta_{l,k}\gamma^{(k-1)+(k-2)}(\frac{\xi}{1-\gamma^{-1}})^2
    \langle E_1^{k-2},x^{i-2d}g^{j+2}u_{k-2}\rangle
~=~ \cdots  \\
&=& \delta_{l,k}\gamma^{(k-1)+(k-2)+\cdots+1}(\frac{\xi}{1-\gamma^{-1}})^{k-1}
    \langle E_1,x^{i-(k-1)d}g^{j+(k-1)}u_1\rangle  \\
&=& \delta_{l,k}\gamma^{(k-1)+(k-2)+\cdots+1}(\frac{\xi}{1-\gamma^{-1}})^k
~=~ \delta_{l,k}\frac{\xi^{k^2}}{(1-\gamma^{-1})^k}.
\end{eqnarray*}
Thus $E_1^m=\frac{\xi^{m^2}}{(1-\gamma^{-1})^m}\chi_{1,1}=\frac{(-1)^m}{(1-\gamma^{-1})^m}\chi_{1,1}
=\frac{1}{(1-\gamma)^m}\chi_{1,1}$.
\end{proof}

It can be found in the proof above that
$$E_1^{[k]}:=\frac{1}{k!_\gamma}E_1^k:\left\{\begin{array}{ll}
    x^ig^jy^l\mapsto \delta_{l,k}  \\
    x^ig^ju_l\mapsto \frac{1}{k!_\gamma}\frac{\xi^{k^2}}{(1-\gamma^{-1})^k}\delta_{l,k}
                     =\frac{\xi^k}{(1-\gamma^{-1})(1-\gamma^{-2})\cdots(1-\gamma^{-k})}\delta_{l,k}
  \end{array}\right..$$
Similarly to Equation (\ref{eqn:Bmore}), we can also verify that for any $\la^\frac{1}{\om},\la^\frac{1}{m}\in\k^\ast$ and $k\in m$,
\begin{eqnarray*}
\zeta_{\la^\frac{1}{\om},\la^\frac{1}{m}}E_1^{[k]} &:&
  \left\{\begin{array}{ll}
    x^ig^jy^l\mapsto \delta_{l,k}\la^\frac{i}{\om}\la^\frac{j}{m}  \\
    x^ig^ju_l\mapsto 0
  \end{array}\right.,\;  \\
\chi_{\la^\frac{1}{\om},\la^\frac{1}{m}}E_1^{[k]} &:&
  \left\{\begin{array}{ll}
    x^ig^jy^l\mapsto 0  \\
    x^ig^ju_l\mapsto \frac{\xi^k}{(1-\gamma^{-1})(1-\gamma^{-2})\cdots(1-\gamma^{-k})}
                     \delta_{l,k}\la^\frac{i}{\om}\la^\frac{j}{m}
  \end{array}\right..
\end{eqnarray*}

\begin{lemma}
Following equations hold in $D(m,d,\xi)^\circ$:
\begin{align*}
& \Delta(E_1)=1\otimes E_1+E_1\otimes(\zeta_{1,\gamma}+\xi\chi_{1,\gamma}),  \\
& \Delta(E_2)
  =(\zeta_{1,1}-\chi_{1,1})\otimes E_2+E_2\otimes 1
   -\sum_{k=1}^{m-1}
    (\zeta_{1,1}-\chi_{1,1})E_1^{[k]}\otimes(\zeta_{1,\gamma}+\xi\chi_{1,\gamma})^{-m+k}E_1^{[m-k]},  \\
& \varepsilon(E_1)=\varepsilon(E_2)=0,
\end{align*}
where $E_1^{[k]}:=\frac{1}{k!_\gamma}E_1^k$ for $1\leq k\leq m-1$. We remark that
$\zeta_{1,1}-\chi_{1,1}=(\zeta_{1,\gamma}+\xi\chi_{1,\gamma})^m$.
\end{lemma}

\begin{proof}
Note that for each $i,i'\in\om$, $j,j'\in\Z$ and $l,l'\in m$,
\begin{eqnarray*}
(x^ig^jy^l)(x^{i'}g^{j'}y^{l'}) &=& \gamma^{j'l}x^{i+i'}g^{j+j'}y^{l+l'},  \\
(x^ig^ju_l)(x^{i'}g^{j'}u_{l'}) &=& \gamma^{j'l}x^{i-i'-2dj'}g^{j+j'}u_lu_{l'},  \\
(x^ig^jy^l)(x^{i'}g^{j'}u_{l'})
&=& \gamma^{j'l}x^{i+i'}\phi_{l'}\phi_{l'+1}\cdots\phi_{l'+l-1}g^{j+j'}u_{l+l'},  \\
(x^ig^ju_l)(x^{i'}g^{j'}y^{l'})
&=& \xi^{-l'}\gamma^{j'l}x^{i-i'-2dj'-dl'}\phi_l\phi_{l+1}\cdots\phi_{l+l'-1}g^{j+j'}u_{l+l'}.
\end{eqnarray*}
We also prove the lemma by through checking their values on the basis.
\begin{eqnarray*}
&& \langle E_1,(x^ig^jy^l)(x^{i'}g^{j'}y^{l'})\rangle  \\
&=& \langle E_1,\gamma^{j'l}x^{i+i'}g^{j+j'}y^{l+l'}\rangle  \\
&=& \delta_{l+l',1}\gamma^{j'l}\langle E_1,x^{i+i'}g^{j+j'}y\rangle
    +\delta_{l+l',m+1}\gamma^{j'l}\langle E_1,x^{i+i'}g^{j+j'}y^{m+1}\rangle  \\
&=& \delta_{l+l',1}\gamma^{j'l}
~=~ \delta_{l,0}\delta_{l',1}+\delta_{l,1}\delta_{l',0}\gamma^{j'},  \\
&=& \langle\zeta_{1,1}\otimes \zeta_{1,1}E_1+\zeta_{1,1}E_1\otimes\zeta_{1,\gamma},
           x^ig^jy^l\otimes x^{i'}g^{j'}y^{l'}\rangle,  \\
&& \langle E_1,(x^ig^ju_l)(x^{i'}g^{j'}u_{l'})\rangle  \\
&=& \langle E_1,\gamma^{j'l}x^{i-i'-2dj'}g^{j+j'}u_lu_{l'}\rangle
~=~ \delta_{l+l',1}\gamma^{j'l}\langle E_1,x^{i-i'-2dj'}g^{j+j'}u_lu_{l'}\rangle  \\
&=& \delta_{l,0}\delta_{l',1}\langle E_1,x^{i-i'-2dj'}g^{j+j'}u_0u_1\rangle
    +\delta_{l,1}\delta_{l',0}\gamma^{j'}\langle E_1,x^{i-i'-2dj'}g^{j+j'}u_1u_0\rangle  \\
&=& \delta_{l,0}\delta_{l',1}\langle E_1,u_0u_1\rangle
    +\delta_{l,1}\delta_{l',0}\gamma^{j'}\langle E_1,u_1u_0\rangle  \\
&=& \delta_{l,0}\delta_{l',1}\langle E_1,\frac{-\xi}{m}\phi_0\phi_1\cdots\phi_{m-3}yg\rangle
    +\delta_{l,1}\delta_{l',0}\gamma^{j'}\langle E_1,\frac{1}{m}\phi_1\phi_2\cdots\phi_{m-2}yg\rangle  \\
&=& \delta_{l,0}\delta_{l',1}\frac{-\xi\gamma}{m}
      (1-\gamma^{-1})(1-\gamma^{-2})\cdots(1-\gamma^{-(m-2)})  \\
& & +\delta_{l,1}\delta_{l',0}\gamma^{j'}\frac{\gamma}{m}
      (1-\gamma^{-2})(1-\gamma^{-3})\cdots(1-\gamma^{-(m-1)})  \\
&=& \delta_{l,0}\delta_{l',1}\frac{-\xi\gamma}{1-\gamma}
    +\delta_{l,1}\delta_{l',0}\gamma^{j'}\frac{\gamma}{1-\gamma^{-1}}  \\
&=& \delta_{l,0}(\frac{\xi}{1-\gamma^{-1}}\delta_{l',1})
    +\xi(\frac{\xi}{1-\gamma^{-1}}\delta_{l,1})(\delta_{l',0}\gamma^{j'})  \\
&=& \langle\chi_{1,1}\otimes \chi_{1,1}E_1+\chi_{1,1}E_1\otimes\xi\chi_{1,\gamma},
           x^ig^ju_l\otimes x^{i'}g^{j'}u_{l'}\rangle,  \\
&& \langle E_1,(x^ig^jy^l)(x^{i'}g^{j'}u_{l'})\rangle  \\
&=& \langle E_1,\gamma^{j'l}x^{i+i'}\phi_{l'}\phi_{l'+1}\cdots\phi_{l'+l-1}g^{j+j'}u_{l+l'}\rangle  \\
&=& \delta_{l,0}\delta_{l',1}\langle E_1,x^{i+i'}g^{j+j'}u_1\rangle
    +\delta_{l,1}\delta_{l',0}\gamma^{j'}\langle E_1,x^{i+i'}\phi_0g^{j+j'}u_1\rangle  \\
&=& \delta_{l,0}\delta_{l',1}\langle E_1,u_1\rangle
    +\delta_{l,1}\delta_{l',0}\gamma^{j'}\langle E_1,\phi_0u_1\rangle  \\
&=& \delta_{l,0}\delta_{l',1}\frac{\xi}{1-\gamma^{-1}}
    +\delta_{l,1}\delta_{l',0}\gamma^{j'}(1-\gamma^{-1})\frac{\xi}{1-\gamma^{-1}}  \\
&=& \delta_{l,0}(\frac{\xi}{1-\gamma^{-1}}\delta_{l',1})
    +\xi\delta_{l,1}(\delta_{l',0}\gamma^{j'})  \\
&=& \langle\zeta_{1,1}\otimes \chi_{1,1}E_1+\zeta_{1,1}E_1\otimes\xi\chi_{1,\gamma},
           x^ig^jy^l\otimes x^{i'}g^{j'}u_{l'}\rangle,  \\
&& \langle E_1,(x^ig^ju_l)(x^{i'}g^{j'}y^{l'})\rangle  \\
&=& \langle E_1,\xi^{-l'}\gamma^{j'l}x^{i-i'-2dj'-dl'}
                \phi_l\phi_{l+1}\cdots\phi_{l+l'-1}g^{j+j'}u_{l+l'}\rangle  \\
&=& \delta_{l,0}\delta_{l',1}\xi^{-1}\langle E_1,x^{i-i'-2dj'-d}\phi_0g^{j+j'}u_1\rangle
    +\delta_{l,1}\delta_{l',0}\gamma^{j'}\langle E_1,x^{i-i'-2dj'}g^{j+j'}u_1\rangle  \\
&=& \delta_{l,0}\delta_{l',1}\xi^{-1}\langle E_1,\phi_0u_1\rangle
    +\delta_{l,1}\delta_{l',0}\gamma^{j'}\langle E_1,u_1\rangle  \\
&=& \delta_{l,0}\delta_{l',1}\xi^{-1}(1-\gamma^{-1})\frac{\xi}{1-\gamma^{-1}}
    +\delta_{l,1}\delta_{l',0}\gamma^{j'}\frac{\xi}{1-\gamma^{-1}}  \\
&=& \delta_{l,0}\delta_{l',1}+(\frac{\xi}{1-\gamma^{-1}}\delta_{l,1})(\delta_{l',0}\gamma^{j'})  \\
&=& \langle\chi_{1,1}\otimes\zeta_{1,1}E_1+\chi_{1,1}E_1\otimes\zeta_{1,\gamma},
           x^ig^ju_l\otimes x^{i'}g^{j'}y^{l'}\rangle.
\end{eqnarray*}
It can be concluded that $\Delta(E_1)=1\otimes E_1+E_1\otimes(\zeta_{1,\gamma}+\xi\chi_{1,\gamma})$.
\begin{eqnarray*}
&& \langle E_2,(x^ig^jy^l)(x^{i'}g^{j'}y^{l'})\rangle  \\
&=& \langle E_2,\gamma^{j'l}x^{i+i'}g^{j+j'}y^{l+l'}\rangle  \\
&=& \gamma^{j'l}(\delta_{l+l',0}\langle E_2,x^{i+i'}g^{j+j'}\rangle
    +\delta_{l+l',m}\langle E_2,x^{i+i'}g^{j+j'}-x^{i+i'}g^{j+j'+m}\rangle)  \\
&=& \delta_{l,0}\delta_{l',0}(\frac{i+i'}{\om}+\frac{j+j'}{m})  \\
& & +\delta_{l+l',m}\gamma^{j'l}(\frac{i+i'}{\om}+\frac{j+j'}{m}-\frac{i+i'}{\om}-\frac{j+j'+m}{m})  \\
&=& \delta_{l,0}\delta_{l',0}(\frac{i+i'}{\om}+\frac{j+j'}{m})
    +\sum_{k=1}^{m-1}\delta_{l,k}\delta_{l',m-k}\gamma^{j'k}(-1)  \\
&=& \delta_{l,0}\delta_{l',0}(\frac{i'}{\om}+\frac{j'}{m})
    +\delta_{l,0}(\frac{i}{\om}+\frac{j}{m})\delta_{l',0}
    -\sum_{k=1}^{m-1}\delta_{l,k}\delta_{l',m-k}\gamma^{j'k}  \\
&=& \langle\zeta_{1,1}\otimes\zeta_{1,1}E_2+\zeta_{1,1}E_2\otimes\zeta_{1,1}
    -\sum_{k=1}^{m-1}\zeta_{1,1}E_1^{[k]}\otimes\zeta_{1,\gamma}^kE_1^{[m-k]},
           x^ig^jy^l\otimes x^{i'}g^{j'}y^{l'}\rangle,  \\
&& \langle E_2,(x^ig^ju_l)(x^{i'}g^{j'}u_{l'})\rangle  \\
&=& \langle E_2,\gamma^{j'l}x^{i-i'-2dj'}g^{j+j'}u_lu_{l'}\rangle  \\
&=& \delta_{l,0}\delta_{l',0}
    \langle E_2,\frac{1}{m}x^{i-i'-2dj'-\frac{(1+m)}{2}d}\phi_0\phi_1\cdots\phi_{m-2}g^{j+j'+1}\rangle  \\
& & +\sum_{k=1}^{m-1}\delta_{l,k}\delta_{l',m-k}\gamma^{j'k}  \\
& &  \;\;\;\;\;\;\;\;\;\;\langle E_2,(-1)^{-m+k}\xi^{(m-k)^2}\frac{1}{m}x^{i-i'-2dj'-\frac{(1+m)}{2}d}
            \phi_k\cdots\phi_{m-1}\phi_0\cdots\phi_{k-2}g^{j+j'+1}\rangle  \\
&=& \delta_{l,0}\delta_{l',0}\frac{1}{m}
    [(\frac{i-i'-2dj'}{\om}-\frac{(1+m)}{2\om}d+\frac{j+j'+1}{m})m+\frac{m-1}{2}]  \\
& & +\sum_{k=1}^{m-1}\delta_{l,k}\delta_{l',m-k}\gamma^{j'k} (-1)^{-m+k}\xi^{(m-k)^2}\frac{1}{m}  \\
& & \;\;\;\;\;\;\;\;\;\;
         [-\frac{1}{m}(1-\gamma^{-k-1})\cdots(1-\gamma^{-m+1})(1-\gamma^{-1})\cdots(1-\gamma^{-k+1})]  \\
&=& \delta_{l,0}\delta_{l',0}(\frac{i-i'}{\om}+\frac{j-j'}{m})  \\
& & +\sum_{k=1}^{m-1}\delta_{l,k}\delta_{l',m-k}\gamma^{j'k} (-1)^{-m+k}\xi^{(m-k)^2}\frac{1}{m}
     (-\frac{1}{m}\frac{m}{1-\gamma^{-k}})  \\
&=& -\delta_{l,0}\delta_{l',0}(\frac{i'}{\om}+\frac{j'}{m})
    +\delta_{l,0}(\frac{i}{\om}+\frac{j}{m})\delta_{l',0}  \\
& & -\sum_{k=1}^{m-1}
     \frac{\xi^k}{(1-\gamma^{-1})\cdots(1-\gamma^{-k})}\delta_{l,k}
     \xi^k\frac{\xi^{m-k}}{(1-\gamma^{-1})\cdots(1-\gamma^{-(m-k)})}\delta_{l',m-k}\gamma^{j'k}  \\
&=& \langle-\chi_{1,1}\otimes\chi_{1,1}E_2+\chi_{1,1}E_2\otimes\chi_{1,1}
    -\sum_{k=1}^{m-1}\chi_{1,1}E_1^{[k]}\otimes\xi^k\chi_{1,\gamma}^kE_1^{[m-k]},
      x^ig^ju_l\otimes x^{i'}g^{j'}u_{l'}\rangle,  \\
&& \langle E_2,(x^ig^jy^l)(x^{i'}g^{j'}u_{l'})\rangle  \\
&=& \langle E_2,\gamma^{j'l}x^{i+i'}\phi_{l'}\phi_{l'+1}\cdots\phi_{l'+l-1}g^{j+j'}u_{l+l'}\rangle  \\
&=& \delta_{l,0}\delta_{l',0}\langle E_2,x^{i+i'}g^{j+j'}u_0\rangle  \\
& & +\sum_{k=1}^{m-1}\delta_{l,k}\delta_{l',m-k}\gamma^{j'k}
     \langle E_2,x^{i+i'}\phi_{m-k}\phi_{m-k+1}\cdots\phi_{m-1}g^{j+j'}u_0\rangle  \\
&=& \delta_{l,0}\delta_{l',0}(\frac{i+i'}{\om}+\frac{j+j'}{m})  \\
& & +\sum_{k=1}^{m-1}\delta_{l,k}\delta_{l',m-k}\gamma^{j'k}
     (-\frac{1}{m})(1-\gamma^{k-1})(1-\gamma^{k-2})\cdots(1-\gamma)  \\
&=& \delta_{l,0}\delta_{l',0}(\frac{i+i'}{\om}+\frac{j+j'}{m})  \\
& & -\sum_{k=1}^{m-1}\delta_{l,k}\delta_{l',m-k}\gamma^{j'k}\frac{1}{m}
     \frac{m}{(1-\gamma^{-1})(1-\gamma^{-2})\cdots(1-\gamma^{-(m-k)})}  \\
&=& \delta_{l,0}\delta_{l',0}(\frac{i'}{\om}+\frac{j'}{m})
    +\delta_{l,0}(\frac{i}{\om}+\frac{j}{m})\delta_{l',0}  \\
& & +\sum_{k=1}^{m-1}\xi^k\delta_{l,k}
    [\frac{\xi^{m-k}}{(1-\gamma^{-1})(1-\gamma^{-2})\cdots(1-\gamma^{-(m-k)})}\delta_{l',m-k}\gamma^{j'k}]\\
&=& \langle\zeta_{1,1}\otimes\chi_{1,1}E_2+\zeta_{1,1}E_2\otimes\chi_{1,1}
    +\sum_{k=1}^{m-1}\zeta_{1,1}E_1^{[k]}\otimes\xi^k\chi_{1,\gamma}^kE_1^{[m-k]},
      x^ig^jy^l\otimes x^{i'}g^{j'}u_{l'}\rangle,  \\
&& \langle E_2,(x^ig^ju_l)(x^{i'}g^{j'}y^{l'})\rangle  \\
&=& \langle E_2,
        \xi^{-l'}\gamma^{j'l}x^{i-i'-2dj'-dl'}\phi_l\phi_{l+1}\cdots\phi_{l+l'-1}g^{j+j'}u_{l+l'}\rangle \\
&=& \delta_{l,0}\delta_{l',0}\langle E_2,x^{i-i'-2dj'}g^{j+j'}u_0\rangle  \\
& & +\sum_{k=1}^{m-1}\delta_{l,k}\delta_{l',m-k}\xi^{-(m-k)}\gamma^{j'k}
     \langle E_2,x^{i-i'-2dj'-d(m-k)}\phi_k\phi_{k+1}\cdots\phi_{m-1}g^{j+j'}u_0\rangle  \\
&=& \delta_{l,0}\delta_{l',0}(\frac{i-i'-2dj'}{\om}+\frac{j+j'}{m})  \\
& & +\sum_{k=1}^{m-1}\delta_{l,k}\delta_{l',m-k}\xi^{-(m-k)}\gamma^{j'k}
     (-\frac{1}{m})(1-\gamma^{m-k-1})(1-\gamma^{m-k-2})\cdots(1-\gamma)  \\
&=& \delta_{l,0}\delta_{l',0}(\frac{i-i'}{\om}+\frac{j-j'}{m})  \\
& & -\sum_{k=1}^{m-1}\delta_{l,k}\delta_{l',m-k}\xi^{-(m-k)}\gamma^{j'k}\frac{1}{m}
     \frac{m}{(1-\gamma^{-1})(1-\gamma^{-2})\cdots(1-\gamma^{-k})}  \\
&=& -\delta_{l,0}\delta_{l',0}(\frac{i'}{\om}+\frac{j'}{m})
    +\delta_{l,0}(\frac{i}{\om}+\frac{j}{m})\delta_{l',0}  \\
& & +\sum_{k=1}^{m-1}
     [\frac{\xi^k}{(1-\gamma^{-1})(1-\gamma^{-2})\cdots(1-\gamma^{-k})}\delta_{l,k}]
     (\delta_{l',m-k}\gamma^{j'k})  \\
&=& \langle-\chi_{1,1}\otimes\zeta_{1,1}E_2+\chi_{1,1}E_2\otimes\zeta_{1,1}
           +\sum_{k=1}^{m-1}\chi_{1,1}E_1^{[k]}\otimes\zeta_{1,\gamma}^kE_1^{[m-k]},
             x^ig^ju_l\otimes x^{i'}g^{j'}y^{l'}\rangle.
\end{eqnarray*}
It can be concluded that
$$\Delta(E_2)=(\zeta_{1,1}-\chi_{1,1})\otimes E_2+E_2\otimes 1
  -\sum_{k=1}^{m-1}
   (\zeta_{1,1}-\chi_{1,1})E_1^{[k]}\otimes(\zeta_{1,\gamma}+\xi\chi_{1,\gamma})^{-m+k}E_1^{[m-k]}.$$

The counit is clear.
\end{proof}

\begin{lemma}\label{lem:D.comult}
Following equations hold in $D(m,d,\xi)^\circ$:
$$\Delta(\zeta_{1,\gamma}+\xi\chi_{1,\gamma})
  =(\zeta_{1,\gamma}+\xi\chi_{1,\gamma})\otimes(\zeta_{1,\gamma}+\xi\chi_{1,\gamma}),\;\;
  \varepsilon(\zeta_{1,\gamma}+\xi\chi_{1,\gamma})=1,$$
and
\begin{align*}
& \Delta(\zeta_{\la^\frac{1}{\om},\la^\frac{d}{\om}})
= \zeta_{\la^\frac{1}{\om},\la^\frac{d}{\om}}\otimes\zeta_{\la^\frac{1}{\om},\la^\frac{d}{\om}}
    +(1-\la)\sum_{k=1}^{m-1}\zeta_{\la^\frac{1}{\om},\la^\frac{d}{\om}}E_1^{[k]}
                \otimes\zeta_{\la^\frac{1}{\om},\la^\frac{d}{\om}}\zeta_{1,\gamma}^kE_1^{[m-k]}  \\
&  \;\;\;\;\;\;\;\;\;\;\;\;\;\;\;\;\;\;\;\;\;
   +\la^\frac{(1-m)d/2}{\om}(1-\la)
   (\ta_0^{-1}\chi_{\la^\frac{1}{\om},\la^\frac{d}{\om}}\otimes\chi_{\la^\frac{-1}{\om},\la^\frac{-d}{\om}}\\
&  \;\;\;\;\;\;\;\;\;\;\;\;\;\;\;\;\;\;\;\;\;\;\;\;\;\;\;\;\;\;
   \;\;\;\;\;\;\;\;\;\;\;\;\;\;\;\;\;\;\;\;\;\;
    +\sum_{k=1}^{m-1}\ta_{m-k}^{-1}
      \chi_{\la^\frac{1}{\om},\la^\frac{d}{\om}}E_1^{[k]}
      \otimes\chi_{\la^\frac{-1}{\om},\la^\frac{-d}{\om}}\xi^k\chi_{1,\gamma}^kE_1^{[m-k]})  \\
& \Delta(\chi_{\la^\frac{1}{\om},\la^\frac{d}{\om}})
= \zeta_{\la^\frac{1}{\om},\la^\frac{d}{\om}}\otimes\chi_{\la^\frac{1}{\om},\la^\frac{d}{\om}}
      -\ta_0\sum_{k=1}^{m-1}\ta_1\cdots\ta_{k-1}
         \zeta_{\la^\frac{1}{\om},\la^\frac{d}{\om}}E_1^{[k]}
           \otimes\chi_{\la^\frac{1}{\om},\la^\frac{d}{\om}}\xi^k\chi_{1,\gamma}^kE_1^{[m-k]}  \\
&  \;\;\;\;\;\;\;\;\;\;\;\;\;\;\;\;\;\;\;\;\;\;
   +\chi_{\la^\frac{1}{\om},\la^\frac{d}{\om}}\otimes\zeta_{\la^\frac{-1}{\om},\la^\frac{-d}{\om}}  \\
&  \;\;\;\;\;\;\;\;\;\;\;\;\;\;\;\;\;\;\;\;\;\;
   -\ta_0\sum_{k=1}^{m-1}\la^\frac{-(m-k)d}{\om}\ta_1\cdots\ta_{m-k-1}
           \chi_{\la^\frac{1}{\om},\la^\frac{d}{\om}}E_1^{[k]}
             \otimes\zeta_{\la^\frac{-1}{\om},\la^\frac{-d}{\om}}\zeta_{1,\gamma}^kE_1^{[m-k]},  \\
& \varepsilon(\zeta_{\la^\frac{1}{\om},\la^\frac{d}{\om}})=1,\;\;
  \varepsilon(\chi_{\la^\frac{1}{\om},\la^\frac{d}{\om}})=0,
\end{align*}
where $\la^\frac{1}{\om}\in\k^\ast$, $E_1^{[k]}:=\frac{1}{k!_\gamma}E_1^k$, and $\theta_0=\frac{1-\la^\frac{d}{\om}}{1/m},\;
  \theta_k=\frac{1-\gamma^k\la^\frac{d}{\om}}{1-\gamma^k}\;(1\leq k\leq m-1).$
\end{lemma}

\begin{proof}
We also prove the lemma through checking their values on the basis.
\begin{eqnarray*}
&&\langle\zeta_{1,\gamma}+\xi\chi_{1,\gamma},(x^ig^jy^l)(x^{i'}g^{j'}y^{l'})\rangle  \\
&=& \langle\zeta_{1,\gamma},\gamma^{j'l}x^{i+i'}g^{j+j'}y^{l+l'}\rangle  \\
&=& \delta_{l+l',0}\langle\zeta_{1,\gamma},x^{i+i'}g^{j+j'}\rangle
    +\delta_{l+l',m}\langle\zeta_{1,\gamma},\gamma^{j'l}x^{i+i'}g^{j+j'}-x^{i+i'}g^{j+j'+m}\rangle  \\
&=& \delta_{l,0}\delta_{l',0}\gamma^{j+j'}
~=~ (\delta_{l,0}\gamma^j)(\delta_{l',0}\gamma^{j'})  \\
&=& \langle\zeta_{1,\gamma}\otimes\zeta_{1,\gamma},x^ig^jy^l\otimes x^{i'}g^{j'}y^{l'}\rangle,  \\
&&\langle\zeta_{1,\gamma}+\xi\chi_{1,\gamma},(x^ig^ju_l)(x^{i'}g^{j'}u_{l'})\rangle  \\
&=& \langle\zeta_{1,\gamma},\gamma^{j'l}x^{i-i'-2dj'}g^{j+j'}u_lu_{l'}\rangle  \\
&=& \delta_{l,0}\delta_{l',0}\langle\zeta_{1,\gamma},
      \frac{1}{m}x^{i-i'-2dj'-\frac{(1+m)}{2}d}\phi_0\phi_1\cdots\phi_{m-2}g^{j+j'+1}\rangle  \\
& & +\delta_{l+l',m}\gamma^{j'l}
     \langle\zeta_{1,\gamma},(-1)^{-l'}\xi^{{l'}^2}\frac{1}{m}x^{i-i'-2dj'-\frac{(1+m)}{2}d}
            \phi_l\cdots\phi_{m-1}\phi_0\cdots\phi_{l-2}g^{j+j'+1}\rangle  \\
&=& \delta_{l,0}\delta_{l',0}\frac{1}{m}\langle\zeta_{1,\gamma},
      x^{i-i'-2dj'-\frac{(1+m)}{2}d}\phi_0\phi_1\cdots\phi_{m-2}g^{j+j'+1}\rangle  \\
&=& \delta_{l,0}\delta_{l',0}\gamma^{j+j'+1}
~=~ (\xi\delta_{l,0}\gamma^j)(\xi\delta_{l',0}\gamma^{j'})  \\
&=& \langle\xi\chi_{1,\gamma}\otimes\xi\chi_{1,\gamma},x^ig^ju_l\otimes x^{i'}g^{j'}u_{l'}\rangle,  \\
&&\langle\zeta_{1,\gamma}+\xi\chi_{1,\gamma},(x^ig^jy^l)(x^{i'}g^{j'}u_{l'})\rangle  \\
&=& \langle\xi\chi_{1,\gamma},
           \gamma^{j'l}x^{i+i'}\phi_{l'}\phi_{l'+1}\cdots\phi_{l'+l-1}g^{j+j'}u_{l+l'}\rangle  \\
&=& \delta_{l,0}\delta_{l',0}\xi\langle\chi_{1,\gamma},x^{i+i'}g^{j+j'}u_0\rangle
    +\delta_{l+l',m}\xi\langle\chi_{1,\gamma},
                       \gamma^{j'l}x^{i+i'}\phi_{l'}\phi_{l'+1}\cdots\phi_{m-1}g^{j+j'}u_m\rangle  \\
&=& \delta_{l,0}\delta_{l',0}\xi\gamma^{j+j'}
~=~ (\delta_{l,0}\gamma^j)(\xi\delta_{l',0}\gamma^{j'})  \\
&=& \langle\zeta_{1,\gamma}\otimes\xi\chi_{1,\gamma},x^ig^jy^l\otimes x^{i'}g^{j'}u_{l'}\rangle,  \\
&&\langle\zeta_{1,\gamma}+\xi\chi_{1,\gamma},(x^ig^ju_l)(x^{i'}g^{j'}y^{l'})\rangle  \\
&=& \langle\xi\chi_{1,\gamma},
    \xi^{-l'}\gamma^{j'l}x^{i-i'-2dj'-dl'}\phi_l\phi_{l+1}\cdots\phi_{l+l'-1}g^{j+j'}u_{l+l'}\rangle  \\
&=& \delta_{l,0}\delta_{l',0}\xi\langle\chi_{1,\gamma},x^{i-i'-2dj'}g^{j+j'}u_0\rangle  \\
& & +\delta_{l+l',m}\xi\langle\chi_{1,\gamma},
    \xi^{-l'}\gamma^{j'l}x^{i-i'-2dj'-dl'}\phi_l\phi_{l+1}\cdots\phi_{m-1}g^{j+j'}u_m\rangle  \\
&=& \delta_{l,0}\delta_{l',0}\xi\gamma^{j+j'}
~=~ (\xi\delta_{l,0}\gamma^j)(\delta_{l',0}\gamma^{j'})  \\
&=& \langle\xi\chi_{1,\gamma}\otimes\zeta_{1,\gamma},x^ig^ju_l\otimes x^{i'}g^{j'}y^{l'}\rangle.
\end{eqnarray*}
It can be concluded that
$$\Delta(\zeta_{1,\gamma}+\xi\chi_{1,\gamma})
  =(\zeta_{1,\gamma}+\xi\chi_{1,\gamma})\otimes(\zeta_{1,\gamma}+\xi\chi_{1,\gamma}).$$

Now we compute that for any $\la^\frac{1}{\om}\in\k^\ast$,
\begin{eqnarray*}
&& \langle\zeta_{\la^\frac{1}{\om},\la^\frac{d}{\om}},(x^ig^jy^l)(x^{i'}g^{j'}y^{l'})\rangle  \\
&=& \langle\zeta_{\la^\frac{1}{\om},\la^\frac{d}{\om}},\gamma^{j'l}x^{i+i'}g^{j+j'}y^{l+l'}\rangle  \\
&=& \delta_{l,0}\delta_{l',0}\langle\zeta_{\la^\frac{1}{\om},\la^\frac{d}{\om}},x^{i+i'}g^{j+j'}\rangle
    +\sum_{k=1}^{m-1}\delta_{l,k}\delta_{l',m-k}\langle\zeta_{\la^\frac{1}{\om},\la^\frac{d}{\om}},
         \gamma^{j'k}(x^{i+i'}g^{j+j'}-x^{i+i'}g^{j+j'+m})\rangle  \\
&=& \delta_{l,0}\delta_{l',0}\la^\frac{i+i}{\om}\la^\frac{(j+j')d}{\om}
    +\sum_{k=1}^{m-1}\delta_{l,k}\delta_{l',m-k}\gamma^{j'k}
         (\la^\frac{i+i'}{\om}\la^\frac{(j+j')d}{\om}-\la^\frac{i+i'}{\om}\la^\frac{(j+j'+m)d}{\om})  \\
&=& \delta_{l,0}\delta_{l',0}\la^\frac{i+i'}{\om}\la^\frac{(j+j')d}{\om}
    +(1-\la)\sum_{k=1}^{m-1}\delta_{l,k}\delta_{l',m-k}\gamma^{j'k}
     \la^\frac{i+i'}{\om}\la^\frac{(j+j')d}{\om}  \\
&=& (\delta_{l,0}\la^\frac{i}{\om}\la^\frac{jd}{\om})(\delta_{l',0}\la^\frac{i'}{\om}\la^\frac{j'd}{\om})
    +(1-\la)\sum_{k=1}^{m-1}(\delta_{l,k}\la^\frac{i}{\om}\la^\frac{jd}{\om})
     (\delta_{l',m-k}\gamma^{j'k}\la^\frac{i'}{\om}\la^\frac{j'd}{\om})  \\
&=& \langle\zeta_{\la^\frac{1}{\om},\la^\frac{d}{\om}}\otimes\zeta_{\la^\frac{1}{\om},\la^\frac{d}{\om}}
           +(1-\la)\sum_{k=1}^{m-1}\zeta_{\la^\frac{1}{\om},\la^\frac{d}{\om}}E_1^{[k]}
                \otimes\zeta_{\la^\frac{1}{\om},\la^\frac{d}{\om}}\zeta_{1,\gamma}^kE_1^{[m-k]},
           x^ig^jy^l\otimes x^{i'}g^{j'}y^{l'}\rangle,  \\
&& \langle\zeta_{\la^\frac{1}{\om},\la^\frac{d}{\om}},(x^ig^ju_l)(x^{i'}g^{j'}u_{l'})\rangle  \\
&=& \langle\zeta_{\la^\frac{1}{\om},\la^\frac{d}{\om}},\gamma^{j'l}x^{i-i'-2dj'}g^{j+j'}u_lu_{l'}\rangle  \\
&=& \delta_{l,0}\delta_{l',0}\langle\zeta_{\la^\frac{1}{\om},\la^\frac{d}{\om}},
      \frac{1}{m}x^{i-i'-2dj'-\frac{(1+m)}{2}d}\phi_0\phi_1\cdots\phi_{m-2}g^{j+j'+1}\rangle  \\
& & +\sum_{k=1}^{m-1}\delta_{l,k}\delta_{l',m-k}\gamma^{j'k}  \\
& & \;\;\;\;\;\;\;\;\;\;
    \langle\zeta_{\la^\frac{1}{\om},\la^\frac{1}{m}},
           (-1)^{-(m-k)}\xi^{(m-k)^2}\frac{1}{m}x^{i-i'-2dj'-\frac{(1+m)}{2}d}
            \phi_k\cdots\phi_{m-1}\phi_0\cdots\phi_{k-2}g^{j+j'+1}\rangle  \\
&=& \delta_{l,0}\delta_{l',0}\frac{1}{m}\la^{\frac{i-i'-2dj'}{\om}-\frac{(1+m)d}{2\om}}
      (1-\gamma^{-1}\la^\frac{d}{\om})(1-\gamma^{-2}\la^\frac{d}{\om})
      \cdots(1-\gamma^{-(m-1)}\la^\frac{d}{\om})\la^\frac{(j+j'+1)d}{\om}  \\
& & +\sum_{k=1}^{m-1}\delta_{l,k}\delta_{l',m-k}\gamma^{j'k}
      (-1)^{-(m-k)}\xi^{(m-k)^2}\frac{1}{m}\la^{\frac{i-i'-2dj'}{\om}-\frac{(1+m)d}{2\om}}  \\
& &   \;\;\;\;\;\;\;\;\;\;(1-\gamma^{-(k+1)}\la^\frac{d}{\om})\cdots(1-\gamma^{-m}\la^\frac{d}{\om})
      (1-\gamma^{-1}\la^\frac{d}{\om})\cdots(1-\gamma^{-(k-1)}\la^\frac{d}{\om})\la^\frac{(j+j'+1)d}{\om}\\
&=& \delta_{l,0}\delta_{l',0}\frac{1}{m}\la^\frac{i-i'}{\om}\la^\frac{(j-j')d}{\om}\la^\frac{(1-m)d/2}{\om}
      \frac{1-\la}{1-\la^\frac{d}{\om}}  \\
& & +\sum_{k=1}^{m-1}\delta_{l,k}\delta_{l',m-k}\gamma^{j'k}
      (-1)^{-(m-k)}\xi^{(m-k)^2}\frac{1}{m}\la^\frac{i-i'}{\om}\la^\frac{(j-j')d}{\om}
      \la^\frac{(1-m)d/2}{\om}\frac{1-\la}{1-\gamma^{-k}\la^\frac{d}{\om}}  \\
&=& (1-\la)\la^\frac{(1-m)d/2}{\om}
    [\frac{1/m}{1-\la^\frac{d}{\om}}\delta_{l,0}\delta_{l',0}\la^\frac{i-i'}{\om}\la^\frac{(j-j')d}{\om}  \\
& & \;\;\;\;\;\;\;\;\;\;\;\;\;\;\;\;\;\;\;\;\;\;\;\;\;\;
    +\sum_{k=1}^{m-1}\frac{1-\gamma^{-k}}{1-\gamma^{-k}\la^\frac{d}{\om}}
         \frac{\xi^k}{(1-\gamma^{-1})\cdots(1-\gamma^{-k})}\delta_{l,k}  \\
& & \;\;\;\;\;\;\;\;\;\;\;\;\;\;\;\;\;\;\;\;\;\;\;\;\;\;\;\;\;\;\;\;\;\;\;\;
         \frac{\xi^{m-k}}{(1-\gamma^{-1})\cdots(1-\gamma^{-(m-k)})}\delta_{l',m-k}
         \la^\frac{i-i'}{\om}\la^\frac{(j-j')d}{\om}\xi^k\gamma^{j'k}]  \\
&=& \langle\la^\frac{(1-m)d/2}{\om}(1-\la)(\theta_0^{-1}
     \chi_{\la^\frac{1}{\om},\la^\frac{d}{\om}}\otimes\chi_{\la^\frac{-1}{\om},\la^\frac{-d}{\om}}  \\
&& \;\;\;\;\;\;\;\;\;\;\;\;\;\;\;\;\;\;\;\;\;\;\;\;\;\;\;\;+\sum_{k=1}^{m-1}\theta_{m-k}^{-1}
      \chi_{\la^\frac{1}{\om},\la^\frac{d}{\om}}E_1^{[k]}
      \otimes\chi_{\la^\frac{-1}{\om},\la^\frac{-d}{\om}}\xi^k\chi_{1,\gamma}^kE_1^{[m-k]}),  \\
& &\;\;  x^ig^ju_l\otimes x^{i'}g^{j'}u_{l'}\rangle,  \\
&& \langle\zeta_{\la^\frac{1}{\om},\la^\frac{d}{\om}},(x^ig^jy^l)(x^{i'}g^{j'}u_{l'})\rangle  \\
&=& \langle\zeta_{\la^\frac{1}{\om},\la^\frac{d}{\om}},
           \gamma^{j'l}x^{i+i'}\phi_{l'}\phi_{l'+1}\cdots\phi_{l'+l-1}g^{j+j'}u_{l+l'}\rangle
~=~ 0,  \\
&& \langle\zeta_{\la^\frac{1}{\om},\la^\frac{d}{\om}},(x^ig^ju_l)(x^{i'}g^{j'}y^{l'})\rangle  \\
&=& \langle\zeta_{\la^\frac{1}{\om},\la^\frac{d}{\om}},
           \xi^{-l'}\gamma^{j'l}x^{i-i'-2dj'-dl'}\phi_l\phi_{l+1}\cdots\phi_{l+l'-1}g^{j+j'}u_{l+l'}\rangle
~=~ 0.
\end{eqnarray*}
It can be concluded that
\begin{align*}
& \Delta(\zeta_{\la^\frac{1}{\om},\la^\frac{d}{\om}})
= \zeta_{\la^\frac{1}{\om},\la^\frac{d}{\om}}\otimes\zeta_{\la^\frac{1}{\om},\la^\frac{d}{\om}}
    +(1-\la)\sum_{k=1}^{m-1}\zeta_{\la^\frac{1}{\om},\la^\frac{d}{\om}}E_1^{[k]}
                \otimes\zeta_{\la^\frac{1}{\om},\la^\frac{d}{\om}}\zeta_{1,\gamma}^kE_1^{[m-k]}  \\
&  \;\;\;\;\;\;\;\;\;\;\;\;\;\;\;\;\;\;\;\;\;
   +\la^\frac{(1-m)d/2}{\om}(1-\la)
   (\ta_0^{-1}\chi_{\la^\frac{1}{\om},\la^\frac{d}{\om}}\otimes\chi_{\la^\frac{-1}{\om},\la^\frac{-d}{\om}}\\
&  \;\;\;\;\;\;\;\;\;\;\;\;\;\;\;\;\;\;\;\;\;\;\;\;\;\;\;\;\;\;
   \;\;\;\;\;\;\;\;\;\;\;\;\;\;\;\;\;\;\;\;\;\;
    +\sum_{k=1}^{m-1}\ta_{m-k}^{-1}
      \chi_{\la^\frac{1}{\om},\la^\frac{d}{\om}}E_1^{[k]}
      \otimes\chi_{\la^\frac{-1}{\om},\la^\frac{-d}{\om}}\xi^k\chi_{1,\gamma}^kE_1^{[m-k]}).
\end{align*}

On the other hand,
\begin{eqnarray*}
&& \langle\chi_{\la^\frac{1}{\om},\la^\frac{d}{\om}},(x^ig^jy^l)(x^{i'}g^{j'}y^{l'})\rangle
~=~ 0,  \\
&& \langle\chi_{\la^\frac{1}{\om},\la^\frac{d}{\om}},(x^ig^ju_l)(x^{i'}g^{j'}u_{l'})\rangle
~=~ 0,  \\
&& \langle\chi_{\la^\frac{1}{\om},\la^\frac{d}{\om}},(x^ig^jy^l)(x^{i'}g^{j'}u_{l'})\rangle  \\
&=& \langle\chi_{\la^\frac{1}{\om},\la^\frac{d}{\om}},
           \gamma^{j'l}x^{i+i'}\phi_{l'}\phi_{l'+1}\cdots\phi_{l'+l-1}g^{j+j'}u_{l+l'}\rangle  \\
&=& \delta_{l,0}\delta_{l',0}
    \langle\chi_{\la^\frac{1}{\om},\la^\frac{d}{\om}},x^{i+i'}g^{j+j'}u_0\rangle  \\
& & +\sum_{k=1}^{m-1}\delta_{l,k}\delta_{l',m-k}\langle\chi_{\la^\frac{1}{\om},\la^\frac{d}{\om}},
         \gamma^{j'k}x^{i+i'}\phi_{m-k}\phi_{m-k+1}\cdots\phi_{m-1}g^{j+j'}u_0\rangle  \\
&=& \delta_{l,0}\delta_{l',0}\la^\frac{i+i'}{\om}\la^\frac{(j+j')d}{\om}
    +\sum_{k=1}^{m-1}\delta_{l,k}\delta_{l',m-k}\gamma^{j'k}\la^\frac{i+i'}{\om}
         (1-\gamma^{k-1}\la^\frac{d}{\om})(1-\gamma^{k-2}\la^\frac{d}{\om})\cdots(1-\la^\frac{d}{\om})
         \la^\frac{(j+j')d}{m}  \\
&=& \delta_{l,0}\delta_{l',0}\la^\frac{i+i'}{\om}\la^\frac{(j+j')d}{\om}  \\
& & -\sum_{k=1}^{m-1}
     (1-\la^\frac{d}{\om})(1-\gamma\la^\frac{d}{\om})\cdots(1-\gamma^{k-1}\la^\frac{d}{\om})
     \frac{m}{(1-\gamma)\cdots(1-\gamma^{k-1})}  \\
& &  \;\;\;\;\;\;\;\;\;\;
     \delta_{l,k}\frac{\xi^{m-k}}{(1-\gamma^{-1})(1-\gamma^{-2})\cdots(1-\gamma^{-(m-k)})}
     \delta_{l',m-k}\xi^k\gamma^{j'k}\la^\frac{i+i'}{\om}\la^\frac{(j+j')d}{\om}  \\
&=& \langle\zeta_{\la^\frac{1}{\om},\la^\frac{d}{\om}}\otimes\chi_{\la^\frac{1}{\om},\la^\frac{d}{\om}}
      -\ta_0\sum_{k=1}^{m-1}\ta_1\cdots\ta_{k-1}
      \zeta_{\la^\frac{1}{\om},\la^\frac{d}{\om}}E_1^{[k]}
           \otimes\chi_{\la^\frac{1}{\om},\la^\frac{d}{\om}}\xi^k\chi_{1,\gamma}^kE_1^{[m-k]},  \\
& & \;\;     x^ig^jy^l\otimes x^{i'}g^{j'}u_{l'}\rangle,  \\
&& \langle\chi_{\la^\frac{1}{\om},\la^\frac{d}{\om}},(x^ig^ju_l)(x^{i'}g^{j'}y^{l'})\rangle  \\
&=& \langle\chi_{\la^\frac{1}{\om},\la^\frac{d}{\om}},
      \xi^{-l'}\gamma^{j'l}x^{i-i'-2dj'-dl'}\phi_l\phi_{l+1}\cdots\phi_{l+l'-1}g^{j+j'}u_{l+l'}\rangle  \\
&=& \delta_{l,0}\delta_{l',0}
    \langle\chi_{\la^\frac{1}{\om},\la^\frac{d}{\om}},x^{i-i'-2dj'}g^{j+j'}u_0\rangle  \\
& & +\sum_{k=1}^{m-k}\delta_{l,k}\delta_{l',m-k}\langle\chi_{\la^\frac{1}{\om},\la^\frac{d}{\om}},
       \xi^{-(m-k)}\gamma^{j'k}x^{i-i'-2dj'-d(m-k)}\phi_k\phi_{k+1}\cdots\phi_{m-1}g^{j+j'}u_0\rangle  \\
&=& \delta_{l,0}\delta_{l',0}\la^\frac{i-i'-2dj'}{\om}\la^\frac{(j+j')d}{\om}
    +\sum_{k=1}^{m-k}\delta_{l,k}\delta_{l',m-k}
       \xi^{-(m-k)}\gamma^{j'k}\la^\frac{i-i'-2dj'-d(m-k)}{\om}  \\
& & \;\;\;\;\;\;\;\;\;\;\;\;\;\;\;\;\;\;\;\;\;\;\;\;\;\;\;\;\;\;\;\;\;\;\;\;\;\;\;\;\;\;\;\;\;
       (1-\gamma^{m-k-1}\la^\frac{d}{\om})(1-\gamma^{m-k-2}\la^\frac{d}{\om})\cdots(1-\la^\frac{d}{\om})
       \la^\frac{(j+j')d}{\om}  \\
&=& \delta_{l,0}\delta_{l',0}\la^\frac{i-i'}{\om}\la^\frac{(j-j')d}{\om}
    +\sum_{k=1}^{m-k}\delta_{l,k}\delta_{l',m-k}
       \xi^{-(m-k)}\gamma^{j'k}\la^\frac{i-i'}{\om}\la^\frac{(j-j')d}{\om}\la^\frac{-(m-k)d}{\om}  \\
& & \;\;\;\;\;\;\;\;\;\;\;\;\;\;\;\;\;\;\;\;\;\;\;\;\;\;\;\;\;\;\;\;\;\;\;\;\;\;\;
       (1-\gamma^{m-k-1}\la^\frac{d}{\om})(1-\gamma^{m-k-2}\la^\frac{d}{\om})\cdots(1-\la^\frac{d}{\om})  \\
&=& \delta_{l,0}\delta_{l',0}\la^\frac{i-i'}{\om}\la^\frac{(j-j')d}{\om}  \\
& & -\sum_{k=1}^{m-k}\la^\frac{-(m-k)d}{\om}
       (1-\la^\frac{d}{\om})(1-\gamma\la^\frac{d}{\om})\cdots(1-\gamma^{m-k-1}\la^\frac{d}{\om})
       \frac{m}{(1-\gamma)\cdots(1-\gamma^{m-k-1})}  \\
& &    \;\;\;\;\;\;\;\;\;\;
       \frac{\xi^k}{(1-\gamma^{-1})\cdots(1-\gamma^{-k})}
       \delta_{l,k}\delta_{l',m-k}\gamma^{j'k}\la^\frac{i-i'}{\om}\la^\frac{(j-j')d}{\om}  \\
&=& \langle\chi_{\la^\frac{1}{\om},\la^\frac{d}{\om}}\otimes\zeta_{\la^\frac{-1}{\om},\la^\frac{-d}{\om}}
      -\ta_0\sum_{k=1}^{m-1}\la^\frac{-(m-k)d}{\om}\ta_1\cdots\ta_{m-k-1}
             \chi_{\la^\frac{1}{\om},\la^\frac{d}{\om}}E_1^{[k]}
             \otimes\zeta_{\la^\frac{-1}{\om},\la^\frac{-d}{\om}}\zeta_{1,\gamma}^kE_1^{[m-k]},  \\
& & \;\;  x^ig^ju_l\otimes x^{i'}g^{j'}y^{l'}\rangle,  \\
\end{eqnarray*}
It can be concluded that
\begin{align*}
& \Delta(\chi_{\la^\frac{1}{\om},\la^\frac{d}{\om}})
= \zeta_{\la^\frac{1}{\om},\la^\frac{d}{\om}}\otimes\chi_{\la^\frac{1}{\om},\la^\frac{d}{\om}}
      -\ta_0\sum_{k=1}^{m-1}\ta_1\cdots\ta_{k-1}
         \zeta_{\la^\frac{1}{\om},\la^\frac{d}{\om}}E_1^{[k]}
           \otimes\chi_{\la^\frac{1}{\om},\la^\frac{d}{\om}}\xi^k\chi_{1,\gamma}^kE_1^{[m-k]}  \\
&  \;\;\;\;\;\;\;\;\;\;\;\;\;\;\;\;\;\;\;\;\;\;
   +\chi_{\la^\frac{1}{\om},\la^\frac{d}{\om}}\otimes\zeta_{\la^\frac{-1}{\om},\la^\frac{-d}{\om}}  \\
&  \;\;\;\;\;\;\;\;\;\;\;\;\;\;\;\;\;\;\;\;\;\;
   -\ta_0\sum_{k=1}^{m-1}\la^\frac{-(m-k)d}{\om}\ta_1\cdots\ta_{m-k-1}
           \chi_{\la^\frac{1}{\om},\la^\frac{d}{\om}}E_1^{[k]}
             \otimes\zeta_{\la^\frac{-1}{\om},\la^\frac{-d}{\om}}\zeta_{1,\gamma}^kE_1^{[m-k]}.
\end{align*}

The counit is clear.
\end{proof}

\begin{lemma}
Following equations hold in $D(m,d,\xi)^\circ$:
\begin{eqnarray*}
S(\zeta_{1,\gamma}+\xi\chi_{1,\gamma})
&=& \zeta_{1,\gamma^{-1}}+\xi^{-1}\chi_{1,\gamma^{-1}},  \\
S(E_1)
&=& -\gamma^{-1}(\zeta_{1,\gamma^{-1}}+\xi^{-1}\chi_{1,\gamma^{-1}})E_1,  \\
S(E_2)
&=& -\zeta_{1,1}E_2+\chi_{1,1}E_2+\frac{1-m}{2m}\chi_{1,1},  \\
S(\zeta_{\la^\frac{1}{\om},\la^\frac{1}{m}})
&=& \zeta_{\la^\frac{-1}{\om},\la^\frac{-1}{m}},  \\
S(\chi_{\la^\frac{1}{\om},\la^\frac{1}{m}})
&=& \la^\frac{(1-m)d/2}{\om}\gamma^{-k}\chi_{\la^\frac{1}{\om},\la^\frac{d}{\om}\gamma^{-k}}\;\;
  \text{when}\;\;\la^\frac{1}{m}=\la^\frac{d}{\om}\gamma^k.
\end{eqnarray*}
\end{lemma}

\begin{proof}
We also prove the lemma through checking their values on the basis.
\begin{eqnarray*}
\langle S(E_1),x^ig^jy^l\rangle
&=& \delta_{l,1}\langle E_1,-\gamma^{-j-1}x^{-i}g^{-j-1}y\rangle  \\
&=& -\delta_{l,1}\gamma^{-j-1}
~=~ \langle-\gamma^{-1}\zeta_{1,\gamma^{-1}}E_1,x^ig^jy^l\rangle,  \\
\langle S(E_1),x^ig^ju_l\rangle
&=& \delta_{l,1}\langle E_1,-\xi^{-1}\gamma^{-j-1}x^{i+d+\frac{3}{2}(1-m)d+2dj}g^{m-2-j}u_1\rangle  \\
&=& -\xi^{-1}\gamma^{-j-1}\frac{\xi}{1-\gamma^{-1}}\delta_{l,1}
~=~ \langle-\gamma^{-1}\xi^{-1}\chi_{1,\gamma^{-1}}E_1,x^ig^ju_l\rangle.
\end{eqnarray*}
It can be concluded that
$$S(E_1)=-\gamma^{-1}(\zeta_{1,\gamma^{-1}}+\xi^{-1}\chi_{1,\gamma^{-1}})E_1.$$
Also,
\begin{eqnarray*}
\langle S(E_2),x^ig^jy^l\rangle
&=& \delta_{l,0}\langle E_2,x^{-i}g^{-j}\rangle  \\
&=& -\delta_{l,0}(\frac{i}{\om}+\frac{j}{m})
~=~ \langle -\zeta_{1,1}E_2,x^ig^jy^l\rangle,  \\
\langle S(E_2),x^ig^ju_l\rangle
&=& \delta_{l,0}\langle E_2,x^{i+2dj+\frac{3}{2}(1-m)d}g^{m-1-j}u_0\rangle  \\
&=& \delta_{l,0}(\frac{i+2dj}{\om}+\frac{3(1-m)d/2}{\om}+\frac{m-1-j}{m})  \\
&=& \delta_{l,0}(\frac{i}{\om}+\frac{j}{m}+\frac{1-m}{2m})
~=~ \langle\chi_{1,1}E_2+\frac{1-m}{2m}\chi_{1,1},x^ig^ju_l\rangle.
\end{eqnarray*}
It can be concluded that
$$S(E_2)=-\zeta_{1,1}E_2+\chi_{1,1}E_2+\frac{1-m}{2m}\chi_{1,1}.$$
Moreover,
\begin{eqnarray*}
\langle S(\zeta_{\la^\frac{1}{\om},\la^\frac{1}{m}}),x^ig^jy^l\rangle
&=& \delta_{l,0}\langle\zeta_{\la^\frac{1}{\om},\la^\frac{1}{m}},x^{-i}g^{-j}\rangle
~=~ \delta_{l,0}\la^\frac{-i}{\om}\la^\frac{-j}{m}  \\
&=& \langle\zeta_{\la^\frac{-1}{\om},\la^\frac{-1}{m}},x^ig^jy^l\rangle,  \\
\langle S(\zeta_{\la^\frac{1}{\om},\la^\frac{1}{m}}),x^ig^ju_l\rangle
&=& 0.
\end{eqnarray*}
When $\la^\frac{1}{m}=\la^\frac{d}{\om}\gamma^k$ holds where $k$ is a non-negative integer,
\begin{eqnarray*}
\langle S(\chi_{\la^\frac{1}{\om},\la^\frac{1}{m}}),x^ig^jy^l\rangle
&=& 0,  \\
\langle S(\chi_{\la^\frac{1}{\om},\la^\frac{1}{m}}),x^ig^ju_l\rangle
&=& \delta_{l,0}
    \langle\chi_{\la^\frac{1}{\om},\la^\frac{1}{m}},x^{i+2dj+\frac{3}{2}(1-m)d}g^{m-1-j}u_0\rangle  \\
&=& \delta_{l,0}\la^\frac{i+2dj}{\om}\la^\frac{3(1-m)d/2}{\om}\la^\frac{m-1-j}{m}  \\
&=& \delta_{l,0}\la^\frac{i+2dj}{\om}\la^\frac{3(1-m)d/2}{\om}(\la^\frac{d}{\om}\gamma^k)^{m-1-j}  \\
&=& \la^\frac{(1-m)d/2}{\om}\gamma^{-k}\delta_{l,0}\la^\frac{i}{\om}(\la^\frac{d}{\om}\gamma^{-k})^j  \\
&=& \langle\la^\frac{(1-m)d/2}{\om}\gamma^{-k}\chi_{\la^\frac{1}{\om},\la^\frac{d}{\om}\gamma^{-k}},
           x^ig^ju_l\rangle.
\end{eqnarray*}
It can be concluded that
\begin{eqnarray*}
S(\zeta_{\la^\frac{1}{\om},\la^\frac{1}{m}})
&=& \zeta_{\la^\frac{-1}{\om},\la^\frac{-1}{m}},  \\
S(\chi_{\la^\frac{1}{\om},\la^\frac{1}{m}})
&=& \la^\frac{(1-m)d/2}{\om}\gamma^{-k}\chi_{\la^\frac{1}{\om},\la^\frac{d}{\om}\gamma^{-k}}\;\;
  \text{when}\;\;\la^\frac{1}{m}=\la^\frac{d}{\om}\gamma^k.
\end{eqnarray*}
\end{proof}

\subsection{The Generation Problem}\label{subsection:D3}

\begin{proposition}\label{p.5.5}
As an algebra, $D(m,d,\xi)^\circ$ is generated by $\zeta_{\la^\frac{1}{\om},\la^\frac{1}{m}}$, $\chi_{\la^\frac{1}{\om},\la^\frac{1}{m}}$, $E_1$ and $E_2$ for $\la^\frac{1}{\om},\la^\frac{1}{m}\in\k^\ast$.
\end{proposition}

\begin{proof}
Clearly $D(m,d,\xi)$ has a Laurent polynomial subalgebra $P=\k[g^{\pm1}]$. Choose
$$\mathcal{I}_0=\{((g^m-1)^r),\;((g^m+1)^r),\;((g^m-\la)^r(g^m-\la^{-1})^r)\lhd_l D(m,d,\xi)
  \mid \la\in\k^\ast\setminus\{\pm1\},\;r\in\N\}.$$
At first for any cofinite left ideal $I$ of $D(m,d,\xi)$, according to Lemma \ref{lem:p(x)}(3), there must be some non-zero polynomial
$$p(g)=(g^m-1)^{r'}(g^m+1)^{r''}
      \prod_{\alpha=1}^N(g^m-\la_\alpha)^{r_\alpha}(g^m-\la_\alpha^{-1})^{r_\alpha}$$
for some $N\geq1$, $r',r'',r_\alpha\in\N$ and distinct $\la_\alpha\in\k^\ast\setminus\{\pm1\}$ which are not inverses of each other, such that $p(g)$ generates a cofinite left ideal contained in $I$.
It can be known by Lemma \ref{lem:intersection} that
$$((g^m-1)^{r'})\cap((g^m+1)^{r''})\cap\left[\bigcap_{\alpha=1}^N
  ((g^m-\la_\alpha)^{r_\alpha}(g^m-\la_\alpha^{-1})^{r_\alpha})\right]= (p(g)) \subseteq I$$
as left ideals of $D(m,d,\xi)$,
and thus Lemma \ref{lem:I0} or Corollary \ref{cor:circ}(3) can be applied to obtain that $D(m,d,\xi)^\circ=\sum_{I\in\mathcal{I}_0}(D(m,d,\xi)/I)^\ast$.

Now we try to prove that for each $I\in\mathcal{I}_0$, the subspace $(D(m,d,\xi)/I)^\ast$ can be spanned by some products of
$\zeta_{\la^\frac{1}{\om},\la^\frac{1}{m}},\chi_{\la^\frac{1}{\om},\la^\frac{1}{m}},
 E_2,E_1\;(\la^\frac{1}{\om},\la^\frac{1}{m}\in\k^\ast)$.
For simplicity, we always denote
$$\psi_{\la^\frac{1}{\om},\la^\frac{1}{m}}
  :=\zeta_{\la^\frac{1}{\om},\la^\frac{1}{m}}+\chi_{\la^\frac{1}{\om},\la^\frac{1}{m}}:
  \left\{\begin{array}{ll}
    x^ig^jy^l\mapsto \delta_{l,0}\lambda^\frac{i}{\om}\lambda^\frac{j}{m}  \\
    x^ig^ju_l\mapsto \delta_{l,0}\lambda^\frac{i}{\om}\lambda^\frac{j}{m}
  \end{array}\right.,\;\;$$
where $\la^\frac{1}{\om},\la^\frac{1}{m}\in\k^\ast$. Note that
$\psi_{\la_1^\frac{1}{\om},\la_1^\frac{1}{m}}\psi_{\la_2^\frac{1}{\om},\la_2^\frac{1}{m}}
=\psi_{\la_1^\frac{1}{\om}\la_2^\frac{1}{\om},\la_1^\frac{1}{m}\la_2^\frac{1}{m}}$ always holds. Moreover, let $\eta$ be an primitive $\om$th root of $1$. 

We continue to show the spanning by a discussion on the form of $I\in\mathcal{I}_0$:

\textbf{Case 1.} Suppose $I=((g^m-\la)^r(g^m-\la^{-1})^r)$ for some $r\in\N$ and $\la\in\k^\ast\setminus\{\pm1\}$. Let $\la^\frac{1}{\om},\la^\frac{1}{m}$ be some fixed roots of $\la$ respectively. 
In this case we aim to show that
\begin{equation}\label{eqn:D1}
\{\psi_{\la^\frac{(-1)^e}{\om},\la^\frac{(-1)^e}{m}}(\zeta_{\eta,1}+\chi_{\eta,1})^i
  (\zeta_{1,\gamma}+\xi\chi_{1,\gamma})^kE_2^sE_1^l
  \mid e\in 2,\;i\in\om,\;k\in 2m,\;s\in r,\;l\in m\}
\end{equation}
is a linear basis of $4\om m^2r$-dimensional space $(D(m,d,\xi)/I)^\ast$. We remark that (\ref{eqn:D1}) is linearly equivalent to
\begin{equation}\label{eqn:D1'}
\{\zeta_{\la^\frac{(-1)^e}{\om},\la^\frac{(-1)^e}{m}}\zeta_{\eta,1}^i\zeta_{1,\gamma}^jE_2^sE_1^l,\;
  \chi_{\la^\frac{(-1)^e}{\om},\la^\frac{(-1)^e}{m}}\chi_{\eta,1}^i\chi_{1,\gamma}^jE_2^sE_1^l
  \mid e\in 2,\;i\in\om,\;j\in m,\;s\in r,\;l\in m\}.
\end{equation}

Evidently
$$I=\k\{x^ig^j(g^m-\la)^r(g^m-\la^{-1})^ry^l,\;x^ig^j(g^m-\la)^r(g^m-\la^{-1})^ru_l
       \mid i\in\om,\;j\in\Z,\;l\in m\},$$
since $yg^m=g^my$ and $u_ig^m=g^{-m}u_i$ for each $i\in m$.
Also, $D(m,d,\xi)/I$ has a linear basis
$$\{x^ig^j(g^m-\la)^s(g^m-\la^{-1})^sy^l,\;x^ig^j(g^m-\la)^s(g^m-\la^{-1})^su_l
       \mid i\in\om,\;j\in 2m,\;s\in r,\;l\in m\},$$
and hence $\dim((D(m,d,\xi)/I)^\ast)=\dim(D(m,d,\xi)/I)=4\om m^2r$. Next we show that all elements in (\ref{eqn:D1'}) vanish on $I$.
We make computations for any $s\in r$, $l,l'\in m$ and $j'\in\Z$, especially for arbitrary roots $\la^\frac{1}{\om}$ and $\la^\frac{1}{m}$ of $\la$: Clearly
\begin{eqnarray}\label{eqn:Dvanish1}
\langle\zeta_{\la^\frac{1}{\om},\la^\frac{1}{m}}E_2^sE_1^l,x^{i'}g^{j'}(g^m-\la)^ru_{l'}\rangle&=&0,  \\
\langle\chi_{\la^\frac{1}{\om},\la^\frac{1}{m}}E_2^sE_1^l,x^{i'}g^{j'}(g^m-\la)^ry^{l'}\rangle&=&0,
\end{eqnarray}
and
\begin{eqnarray}
\langle\zeta_{\la^\frac{1}{\om},\la^\frac{1}{m}}E_2^sE_1^l,x^{i'}g^{j'}(g^m-\la)^ry^{l'}\rangle&=&0
\end{eqnarray}
due to the same reason as Equation (\ref{eqn:Bvanish}). Similarly,
\begin{eqnarray}\label{eqn:Dvanish4}
& & \langle\chi_{\la^\frac{1}{\om},\la^\frac{1}{m}}E_2^sE_1^l,x^{i'}g^{j'}(g^m-\la)^ru_{l'}\rangle  \\
&=& \langle\chi_{\la^\frac{1}{\om},\la^\frac{1}{m}}E_2^sE_1^l,
           \sum_{t=0}^r\binom{r}{t}x^{i'}g^{j'+mt}(-\la)^{r-t}u_{l'}\rangle  \nonumber  \\
&=& \sum_{t=0}^r\binom{r}{t}(-\la)^{r-t}
    \langle\chi_{\la^\frac{1}{\om},\la^\frac{1}{m}}E_2^s\otimes E_1^l,
           \Delta(x^{i'}g^{j'+nt}u_{l'})\rangle  \nonumber  \\
&=& \delta_{l',l}\sum_{t=0}^r\binom{r}{t}(-\la)^{r-t}
    \langle\chi_{\la^\frac{1}{\om},\la^\frac{1}{m}}E_2^s,x^{i'}g^{j'+mt}u_0\rangle
    \langle E_1^l,x^{i'}g^{j'+nt}u_l\rangle  \nonumber  \\
&=& \delta_{l',l}\sum_{t=0}^r\binom{r}{t}(-\la)^{r-t}
    \la^\frac{i'}{\om}\la^\frac{j'+mt}{m}(\frac{i'}{\om}+\frac{j'+mt}{m})^s
    \frac{\xi^{l^2}}{(1-\gamma^{-1})^l}  \nonumber  \\
&=& \delta_{l',l}\frac{\xi^{l^2}}{(1-\gamma^{-1})^l}\la^r\la^\frac{i'}{\om}\la^\frac{j'}{m}
    \sum_{t=0}^r\binom{r}{t}(-1)^{r-t}
    \sum_{u=0}^s\binom{s}{u}(\frac{i'}{\om}+\frac{j'}{m})^ut^{s-u}  \nonumber  \\
&=& \delta_{l',l}\frac{\xi^{l^2}}{(1-\gamma^{-1})^l}\la^r\la^\frac{i'}{\om}\la^\frac{j'}{m}
    \sum_{u=0}^s\binom{s}{u}(\frac{i'}{\om}+\frac{j'}{m})^u
    \sum_{t=0}^r\binom{r}{t}(-1)^{r-t}t^{s-u}  \nonumber  \\
&=& 0,  \nonumber
\end{eqnarray}
since $\sum_{t=0}^r\binom{r}{t}(-1)^{r-t}t^{s-u}=0$ when $s-u<r$ is a part of the second Stirling number (\ref{eqn:2Stirling}).

Of course, these equations still hold if we substitute $\la^\frac{1}{\om}$ and $\la^\frac{1}{m}$ by $\la^\frac{-1}{\om}$ and $\la^\frac{-1}{m}$ respectively. As a conclusion,
all elements in (\ref{eqn:D1'}) as well as (\ref{eqn:D1}) belong to $I^\perp=(D(m,d,\xi)/I)^\ast$.

Finally we prove that (\ref{eqn:D1}) are linearly independent. Choose a lexicographically ordered set of elements
$\{h_{e',i',k',s',l'}\in D(m,d,\xi)\mid (e',i',k',s',l')\in 2\times\om\times 2m\times r\times m\}$,
where
$$h_{e',i',k',s',l'}:=\left\{\begin{array}{ll}
  \frac{1}{l'!_\gamma}x^{i'}g^{\frac{k'}{2}+2s'm+e'm}y^{l'},
  & \text{when}\;\; 2\mid k'; \\
  \frac{(1-\gamma^{-1})^{l'}}{\xi^{{l'}^2}}x^{i'}g^{\frac{k'-1}{2}+2s'm+e'm}u_{l'},
  & \text{when}\;\; 2\nmid k'.
\end{array}\right.$$

Our goal is to show that the $4\om m^2r\times 4\om m^2r$ square matrix
\begin{eqnarray*}
A&:=& (\langle\psi_{\la^\frac{(-1)^e}{\om},\la^\frac{(-1)^e}{m}}(\zeta_{\eta,1}+\chi_{\eta,1})^i
          (\zeta_{1,\gamma}+\xi\chi_{1,\gamma})^kE_2^sE_1^l,
           h_{e',i',k',s',l'}\rangle  \\
& &   \;\;\mid (e,i,k,s,l),(e',i',k',s',l')\in 2\times\om\times 2m\times r\times m )
\end{eqnarray*}
is invertible for fixed $\la^\frac{1}{\om},\la^\frac{1}{n}$, which would imply the linear independence of (\ref{eqn:D1}) as well as (\ref{eqn:D1'}). We make following computations: When $2\mid k'$,
\begin{eqnarray*}
& & \langle\psi_{\la^\frac{(-1)^e}{\om},\la^\frac{(-1)^e}{m}}(\zeta_{\eta,1}+\chi_{\eta,1})^i
           (\zeta_{1,\gamma}+\xi\chi_{1,\gamma})^kE_2^sE_1^l,h_{e',i',k',s',l'}\rangle  \\
&=& \langle\zeta_{\la^\frac{(-1)^e}{\om},\la^\frac{(-1)^e}{m}}\zeta_{\eta,1}^i\zeta_{1,\gamma}^kE_2^sE_1^l,
           \frac{1}{l'!_\gamma}x^{i'}g^{\frac{k'}{2}+2s'm+e'm}y^{l'}\rangle  \\
&=& \delta_{l',l}\langle\zeta_{\la^\frac{(-1)^e}{\om}\eta^i,\la^\frac{(-1)^e}{m}\gamma^k}E_2^s,
           x^{i'}g^{\frac{k'}{2}+2s'm+e'm}\rangle
    \frac{1}{l!_\gamma}\langle E_1^l,x^{i'}g^{\frac{k'}{2}+2s'm+e'm}y^l\rangle  \\
&=& \delta_{l',l}\la^{(-1)^e(\frac{i'}{\om}+\frac{k'}{2m}+2s'+e')}\eta^{ii'}\xi^{kk'}
    (\frac{i'}{\om}+\frac{k'}{2m}+2s'+e')^s;
\end{eqnarray*}
When $2\nmid k'$,
\begin{eqnarray*}
& & \langle\psi_{\la^\frac{(-1)^e}{\om},\la^\frac{(-1)^e}{m}}(\zeta_{\eta,1}+\chi_{\eta,1})^i
           (\zeta_{1,\gamma}+\xi\chi_{1,\gamma})^kE_2^sE_1^l,h_{e',i',k',s',l'}\rangle  \\
&=& \langle\chi_{\la^\frac{(-1)^e}{\om},\la^\frac{(-1)^e}{m}}\chi_{\eta,1}^i(\xi\chi_{1,\gamma})^k
           E_2^sE_1^l,
           \frac{(1-\gamma^{-1})^{l'}}{\xi^{{l'}^2}}x^{i'}g^{\frac{k'-1}{2}+2s'm+e'm}u_{l'}\rangle  \\
&=& \delta_{l',l}\langle\chi_{\la^\frac{(-1)^e}{\om}\eta^i,\la^\frac{(-1)^e}{m}\gamma^k}\xi^kE_2^s,
                        x^{i'}g^{\frac{k'-1}{2}+2s'm+e'm}u_0\rangle
    \frac{(1-\gamma^{-1})^l}{\xi^{l^2}}\langle E_1^l,x^{i'}g^{\frac{k'-1}{2}+2s'm+e'm}u_l\rangle  \\
&=& \delta_{l',l}\la^{(-1)^e(\frac{i'}{\om}+\frac{k'-1}{2m}+2s'+e')}\eta^{ii'}\xi^{kk'}
    (\frac{i'}{\om}+\frac{k'-1}{2m}+2s'+e')^s.
\end{eqnarray*}
One concludes that
\begin{eqnarray*}
A&:=& (\delta_{l',l}\la^{(-1)^e(\frac{i'}{\om}+\lfloor\frac{k'}{2}\rfloor\frac{1}{m}+2s'+e')}
         \eta^{ii'}\xi^{kk'}(\frac{i'}{\om}+\lfloor\frac{k'}{2}\rfloor\frac{1}{m}+2s'+e')^s  \\
& &    \;\;\mid (e,i,k,s,l),(e',i',k',s',l')\in 2\times\om\times 2m\times r\times m).
\end{eqnarray*}
However, there are following two facts:
\begin{itemize}
  \item The matrix $(\delta_{l',l}\eta^{ii'}\xi^{kk'}\mid (i,k,l),(i',k',l')\in \om\times 2m\times m)$ is invertible by Lemma \ref{lem:Kron}, since it is the Kronecker product of three invertible matrices;
  \item For each $(i',k')\in \om\times 2m$, the matrix
      $$\left(\la^{(-1)^e(\frac{i'}{\om}+\lfloor\frac{k'}{2}\rfloor\frac{1}{m}+2s'+e')}
              (\frac{i'}{\om}+\lfloor\frac{k'}{2}\rfloor\frac{1}{m}+2s'+e')^s
         \mid (e,s),(e',s')\in 2\times r\right)$$
      is invertible by Corollary \ref{cor:mat}(2).
\end{itemize}
Thus $A$ is invertible according to Lemma \ref{lem:Kron2}, and (\ref{eqn:D1}) or (\ref{eqn:D1'}) is a basis of $(D(m,d,\xi)/I)^\ast$.

\textbf{Case 2.} Suppose $I=((g^m-1)^r)$ for some $r\in\N$. In this case we aim to show that
\begin{equation}\label{eqn:D2}
\{(\zeta_{\eta,1}+\chi_{\eta,1})^i(\zeta_{1,\gamma}+\xi\chi_{1,\gamma})^kE_2^sE_1^l
  \mid i\in\om,\;k\in 2m,\;s\in r,\;l\in m\}
\end{equation}
is a linear basis of $2\om m^2r$-dimensional space $(D(m,d,\xi)/I)^\ast$. We remark that (\ref{eqn:D2}) is linearly equivalent to
\begin{equation}\label{eqn:D2'}
\{\zeta_{\eta,1}^i\zeta_{1,\gamma}^jE_2^sE_1^l,\;\chi_{\eta,1}^i\chi_{1,\gamma}^jE_2^sE_1^l
  \mid i\in\om,\;j\in m,\;s\in r,\;l\in m\}.
\end{equation}

Evidently
$$I=\k\{x^ig^j(g^m-1)^ry^l,\;x^ig^j(g^m-1)^ru_l \mid i\in\om,\;j\in\Z,\;l\in m\},$$
since $yg^m=g^my$ and $u_ig^m=g^{-m}u_i$ for each $i\in m$.
Also, $D(m,d,\xi)/I$ has a linear basis
$$\{x^ig^j(g^m-1)^sy^l,\;x^ig^j(g^m-1)^su_l \mid i\in\om,\;j\in m,\;s\in r,\;l\in m\},$$
and hence $\dim((D(m,d,\xi)/I)^\ast)=\dim(D(m,d,\xi)/I)=2\om m^2r$. Similar computations to Equations (\ref{eqn:Dvanish1}) to (\ref{eqn:Dvanish4}) follows that all elements in (\ref{eqn:D2'}) vanish on $I$, and hence all elements in (\ref{eqn:D2'}) as well as (\ref{eqn:D2}) belong to $I^\perp=(D(m,d,\xi)/I)^\ast$.

Finally we prove that (\ref{eqn:D2}) are linearly independent. Our goal is to show that the $2\om m^2r\times 2\om m^2r$ square matrix
\begin{eqnarray*}
A&:=& (\langle(\zeta_{\eta,1}+\chi_{\eta,1})^i(\zeta_{1,\gamma}+\xi\chi_{1,\gamma})^kE_2^sE_1^l,
            h_{0,i',k',s',l'}\rangle  \\
& & \;\;\mid (i,k,s,l),(i',k',s',l')\in \om\times 2m\times r\times m)  \\
&=& \left(\delta_{l',l}\eta^{ii'}\xi^{kk'}(\frac{i'}{\om}+\lfloor\frac{k'}{2}\rfloor\frac{1}{m}+2s')^s
          \mid (i,k,s,l),(i',k',s',l')\in \om\times 2m\times r\times m\right)
\end{eqnarray*}
is invertible. This is true due to the same argument in the proof of Proposition \ref{prop:B} that $A$ is invertible there.

\textbf{Case 3.} Suppose $I=((g^m+1)^r)$ for some $r\in\N$. Let $(-1)^\frac{1}{\om},(-1)^\frac{1}{m}$ be some fixed roots of $-1$ respectively. In this case we aim to show that
\begin{equation}\label{eqn:D3}
\{\psi_{(-1)^\frac{1}{\om},(-1)^\frac{1}{m}}(\zeta_{\eta,1}+\chi_{\eta,1})^i
  (\zeta_{1,\gamma}+\xi\chi_{1,\gamma})^kE_2^sE_1^l \mid i\in\om,\;k\in 2m,\;s\in r,\;l\in m\}
\end{equation}
is a linear basis of $2\om m^2r$-dimensional space $(D(m,d,\xi)/I)^\ast$. We remark that (\ref{eqn:D3}) is linearly equivalent to
\begin{equation}\label{eqn:D3'}
\{\zeta_{(-1)^\frac{1}{\om},(-1)^\frac{1}{m}}\zeta_{\eta,1}^i\zeta_{1,\gamma}^jE_2^sE_1^l,\;
  \chi_{(-1)^\frac{1}{\om},(-1)^\frac{1}{m}}\chi_{\eta,1}^i\chi_{1,\gamma}^jE_2^sE_1^l
  \mid i\in\om,\;j\in m,\;s\in r,\;l\in m\}.
\end{equation}

Evidently
$$I=\k\{x^ig^j(g^m+1)^ry^l,\;x^ig^j(g^m+1)^ru_l \mid i\in\om,\;j\in\Z,\;l\in m\},$$
since $yg^m=g^my$ and $u_ig^m=g^{-m}u_i$ for each $i\in m$.
Also, $D(m,d,\xi)/I$ has a linear basis
$$\{x^ig^j(g^m+1)^sy^l,\;x^ig^j(g^m+1)^su_l \mid i\in\om,\;j\in m,\;s\in r,\;l\in m\},$$
and hence $\dim((D(m,d,\xi)/I)^\ast)=\dim(D(m,d,\xi)/I)=2\om m^2r$. Similar computations to Equations (\ref{eqn:Dvanish1}) to (\ref{eqn:Dvanish4}) follows that all elements in (\ref{eqn:D3'}) vanish on $I$, and hence all elements in (\ref{eqn:D3'}) as well as (\ref{eqn:D3}) belong to $I^\perp=(D(m,d,\xi)/I)^\ast$.

Finally we prove that (\ref{eqn:D3}) are linearly independent. Our goal is to show that the $2\om m^2r\times 2\om m^2r$ square matrix
\begin{eqnarray*}
A&:=& (\langle\psi_{(-1)^\frac{1}{\om},(-1)^\frac{1}{m}}
              (\zeta_{\eta,1}+\chi_{\eta,1})^i(\zeta_{1,\gamma}+\xi\chi_{1,\gamma})^kE_2^sE_1^l,
              h_{0,i',k',s',l'}\rangle  \\
& &   \;\;\mid (i,k,s,l),(i',k',s',l')\in \om\times 2m\times r\times m) \\
&=& ((-1)^{\frac{i'}{\om}+\lfloor\frac{k'}{2}\rfloor\frac{1}{m}+2s'}
          \delta_{l',l}\eta^{ii'}\xi^{kk'}(\frac{i'}{\om}+\lfloor\frac{k'}{2}\rfloor\frac{1}{m}+2s')^s  \\
& & \;\;  \mid (i,k,s,l),(i',k',s',l')\in \om\times 2m\times r\times m)
\end{eqnarray*}
is invertible. This is true due to the same argument in the proof of Proposition \ref{prop:B} that $A$ is invertible there.

We summarize three cases above in a result that
\begin{eqnarray*}
D(m,d,\xi)^\circ
&=& \sum_{I\in\mathcal{I}_0}(D(m,d,\xi)/I)^\ast  \\
&=& \k\{\zeta_{\la^\frac{1}{\om},\la^\frac{1}{m}}E_2^sE_1^l,\;
        \chi_{\la^\frac{1}{\om},\la^\frac{1}{m}}E_2^sE_1^l
        \mid \la^\frac{1}{\om},\la^\frac{1}{m}\in\k^\ast,\;s\in\N,\;l\in m\}.
\end{eqnarray*}
\end{proof}

\begin{theorem}\label{thm:D}
\begin{itemize}
  \item[(1)] $D_\circ(m,d,\xi)$ constructed in Subsection \ref{subsection:D1} is a Hopf algebra;
  \item[(2)] As a Hopf algebra, $D(m,d,\xi)^\circ$ is isomorphic to $D_\circ(m,d,\xi)$.
\end{itemize}
\end{theorem}

\begin{proof}
Consider the following map
\begin{eqnarray*}
\Theta &:& D_\circ(m,d,\xi)\rightarrow D(m,d,\xi)^\circ,  \\
& & \Zeta_{\la^\frac{1}{\om},\la^\frac{1}{m}}\mapsto\zeta_{\la^\frac{1}{\om},\la^\frac{1}{m}},\;
    \Chi_{\la^\frac{1}{\om},\la^\frac{1}{m}}\mapsto\chi_{\la^\frac{1}{\om},\la^\frac{1}{m}},\;
    F_1\mapsto E_1,\;F_2\mapsto E_2,
\end{eqnarray*}
where $\la^\frac{1}{\om},\la^\frac{1}{m}\in\k^\ast$. This is a epimorphism of algebras by Lemma \ref{lem:Dalg} and Proposition \ref{p.5.5}. Furthermore, $\Theta$ would become an isomorphism of Hopf algebras with desired coalgebra structure and antipode, as long as it is injective (since $D(m,d,\xi)^\circ$ is in fact a Hopf algebra).

In order to show that $\Theta$ is injective, we aim to show the linear independence of
$$\{\zeta_{\la^\frac{1}{\om},\la^\frac{1}{m}}\zeta_{\eta,1}^i\zeta_{1,\gamma}^jE_2^sE_1^l,\;
    \chi_{\la^\frac{1}{\om},\la^\frac{1}{m}}\chi_{\eta,1}^i\chi_{1,\gamma}^jE_2^sE_1^l
    \mid \la\in\k^\ast,\;i\in\om,\;j\in m,\;s\in\N,\;l\in m\}$$
in $D(m,d,\xi)^\circ$, where $\la^\frac{1}{\om}$ and $\la^\frac{1}{m}$ are fixed roots of each $\la\in\k^\ast$. By linear independence of elements in \eqref{eqn:D1'}, \eqref{eqn:D2'} and \eqref{eqn:D3'}, we only need to show that

This is due to the fact that any finite sum of form
\begin{eqnarray*}
& & (D(m,d,\xi)/((g^m-1)^r))^\ast+(D(m,d,\xi)/((g^m+1)^r))^\ast  \\
& & +\sum_{\alpha=1}^N(D(m,d,\xi)/((g^m-\la_\alpha)^r(g^m-\la_\alpha^{-1})^r)^\ast  \\
&=& ((g^m-1)^r)^\perp+((g^m+1)^r)^\perp
    +\sum_{\alpha=1}^N((g^m-\la_\alpha)^r(g^m-\la_\alpha^{-1})^r)^\perp
\end{eqnarray*}
is direct, as long as $\la_\alpha$'s are distinct and not inverses of each other. The reason is the same as Equations (\ref{eqn:Binj2}).
\end{proof}

\begin{remark}
\emph{The infinite dihedral group $\mathbb{D}_\infty$ is generated by $g$ and $x$ with relations
$$x^2=1,\;\;xgx=g^{-1},$$
and the finite dual of $\k\mathbb{D}_\infty$ is established by \cite{GL21}. Note that $\k\mathbb{D}_\infty=D(1,1,-1)$ and thus we cover the presentation of $(\k\mathbb{D}_\infty)^\circ$. For our convenience, we use our notion to write $(\k\mathbb{D}_\infty)^\circ$ out: It is generated by
$$\zeta_\la:g^jx^k\mapsto \delta_{k,0}\la^j,\;\;\chi_\la:g^jx^k\mapsto \delta_{k,1}\la^j,\;\;
E_2:g^jx^k\mapsto j\;\;\;(j\in\Z,\;k\in2)$$
for $\la\in\k^\ast$, with relations
\begin{align*}
& \zeta_{\la_1}\zeta_{\la_2}=\zeta_{\la_1\la_2},\;\;\chi_{\la_1}\chi_{\la_2}=\chi_{\la_1\la_2},\;\;
  \zeta_{\la_1}\chi_{\la_2}=\chi_{\la_1}\zeta_{\la_2}=0,\;\;\chi_1+\zeta_1=1,  \\
& E_2\zeta_\la=\zeta_\la E_2,\;\;E_2\chi_\la=\chi_\la E_2,
\end{align*}
and
\begin{align*}
& \Delta(\zeta_\la)=\zeta_\la\otimes\zeta_\la+\chi_\la\otimes\chi_{\la^{-1}},\;\;
  \Delta(\chi_\la)=\zeta_\la\otimes\chi_\la+\chi_\la\otimes\zeta_{\la^{-1}},  \\
& \Delta(E_2)=(\zeta_1-\chi_1)\otimes E_2+E_2\otimes 1,  \\
& \varepsilon(\zeta_\la)=1,\;\;\varepsilon(\chi_\la)=\varepsilon(E_2)=0,\;\;  \\
& S(\zeta_\la)=\zeta_{\la^{-1}},\;\;S(\chi_\la)=\chi_{\la^{-1}},\;\;S(E_2)=-(\zeta_1-\chi_1)E_2.
\end{align*}
This presentation is equivalent to that in \cite[Section 3.1]{GL21} and the equivalence  is given through:
$$\zeta_\la\mapsto \frac{1}{2}(\phi_\la+\psi_\la),\;\;
  \chi_\la\mapsto \frac{1}{2}(\phi_\la-\psi_\la)\;\;(\la\in\k^\ast),\;\;
  E_2\mapsto F,$$
  where we used notions of \cite{GL21} freely.}
\end{remark}

\section{Hopf Pairing and Consequences}\label{section6}

With the help of Sections \ref{section3}, \ref{section4} and \ref{section5}, we show that there always is a non-degenerate Hopf pairing on each affine prime regular Hopf algebra of GK-dimension one. Using such Hopf pairing, we attempt to consider a infinite-dimensional version of some conclusions which are valid in the finite-dimensional situation.

\subsection{Non-Degenerate Hopf Pairing}\label{subsection:bullet}

As mentioned in the introduction, the notion of pairing of a bialgebra or a Hopf algebra was given by Majid  \cite{Maj90}:

\begin{definition}\label{def:Hopfpairing}
Let $H$ and $H^\bullet$ be Hopf algebras.
A linear map $\langle-,-\rangle:H^\bullet\otimes H\rightarrow\k$ is called a Hopf pairing (on $H$), if
$$\begin{array}{ll}
\mathrm{(i)}\;\;\;\;\langle ff',h\rangle=\sum\langle f,h_{(1)}\rangle\langle f',h_{(2)}\rangle,
& \mathrm{(ii)}\;\;\;\;\langle f,hh'\rangle=\sum\langle f_{(1)},h\rangle\langle f_{(2)},h'\rangle,  \\
\mathrm{(iii)}\;\;\langle 1,h\rangle=\varepsilon(h),
& \mathrm{(iv)}\;\;\;\langle f,1\rangle=\varepsilon(f),  \\
\mathrm{(v)}\;\;\;\langle f,S(h)\rangle=\langle S(f),h\rangle
\end{array}$$
hold for all $f,f'\in H^\bullet$ and $h,h'\in H$.
Moreover, it is said to be non-degenerate, if for any $f\in H^\bullet$ and any $h\in H$,
\center{$\langle f,H\rangle=0$ implies $f=0$, and $\langle H^\bullet,h\rangle=0$ implies $h=0$.}
\end{definition}

Clearly, the definition follows that there are linear maps
$$\alpha:H^\bullet\rightarrow H^\ast,\;f\mapsto \langle f,-\rangle\;\;\;\text{and}\;\;\;\beta:H\rightarrow H^\bullet{}^\ast,\;h\mapsto \langle-,h\rangle.$$
Furthermore, we know by (ii) in Definition \ref{def:Hopfpairing} that for any $f\in H^\bullet$,
$$M^\ast(\alpha(f))=\sum\alpha(f_{(1)})\otimes\alpha(f_{(2)})\in H^\ast\otimes H^\ast,$$
where $M$ denotes the multiplication on $H$, and hence the image of $\alpha$ is in fact contained in $H^\circ$ by \cite[Lemma 9.1.1]{Mon93}.

As a conclusion, Definition \ref{def:Hopfpairing} follows that there are two maps of Hopf algebras
$$\alpha:H^\bullet\rightarrow H^\circ,\;f\mapsto \langle f,-\rangle\;\;\;\text{and}\;\;\;\beta:H\rightarrow H^\bullet{}^\circ,\;h\mapsto \langle-,h\rangle,$$
which are both injective if and only if the Hopf pairing $\langle-,-\rangle$ is non-degenerate. Thus we might be able to construct non-degenerate Hopf pairing on $H$, as long as the structure of finite dual $H^\circ$ is determined.

As examples, non-degenerate Hopf pairing on $\k\mathbb{D}_\infty$, $T_\infty(n,v,\xi)$, $B(n,\om,\gamma)$ and $D(m,d,\xi)$ are constructed as follows, respectively.

\begin{proposition}\label{prop:bullet}
For each affine prime regular Hopf algebra $H$ of GK-dimension one, we can construct a Hopf algebra $H^\bullet$ and a non-degenerate Hopf pairing  $\langle-,-\rangle:H^\bullet\otimes H\rightarrow\k$ as follows: Specifically, keeping the notions used in Sections $3,4$ and $5$, we have
\begin{itemize}
\item[(1)]
The evaluation $\langle-,-\rangle:(\k\mathbb{D}_\infty)^\bullet\otimes\k\mathbb{D}_\infty\rightarrow\k$ is a non-degenerate Hopf pairing, where
\begin{eqnarray*}
(\k\mathbb{D}_\infty)^\bullet
&=& \k\{\zeta_1E_2^s,\;\chi_1E_2^s\mid s\in\N\}  \\
&=& \k\{(\zeta_1-\chi_1)^kE_2^s\mid k\in2,\;s\in\N\}
\;\subseteq\;(\k\mathbb{D}_\infty)^\circ.
\end{eqnarray*}

\item[(2)]
The evaluation $\langle-,-\rangle:T_\infty(n,v,\xi)^\bullet\otimes T_\infty(n,v,\xi)\rightarrow\k$ is a non-degenerate Hopf pairing, where
\begin{eqnarray*}
T_\infty(n,v,\xi)^\bullet
&=& \k\{\omega^jE_2^sE_1^l \mid j\in n,\;s\in\N,\;l\in m\}  \\
&=& \k\{\omega^jE_2^{[s]}E_1^{[l]} \mid j\in n,\;s\in\N,\;l\in m\}
\;\subseteq\; T_\infty(n,v,\xi)^\circ,
\end{eqnarray*}
and $m=\frac{n}{\gcd(n,v)}$, $E_2^{[s]}E_1^{[l]}=\frac{1}{s!\cdot l!_{\xi^v}}E_2^sE_1^l$.

\item[(3)]
The evaluation $\langle-,-\rangle:B(n,\omega,\gamma)^\bullet\otimes B(n,\omega,\gamma)\rightarrow\k$ is a non-degenerate Hopf pairing, where
\begin{eqnarray*}
B(n,\omega,\gamma)^\bullet
&=& \k\{\psi_{1,\gamma}^jE_2^sE_1^l\mid j\in n,\;s\in\N,\;l\in n\}  \\
&=& \k\{\psi_{1,\gamma}^jE_2^sE_1^{[l]}\mid j\in n,\;s\in\N,\;l\in n\}
\;\subseteq\; B(n,\omega,\gamma)^\circ,
\end{eqnarray*}
and $E_1^{[l]}=\frac{1}{l!_\gamma}E_1^l$.

\item[(4)]
The evaluation $\langle-,-\rangle:D(m,d,\xi)^\bullet\otimes D(m,d,\xi)\rightarrow\k$ is a non-degenerate Hopf pairing, where
\begin{eqnarray*}
D(m,d,\xi)^\bullet
&=& \k\{\zeta_{1,\gamma}^jE_2^sE_1^l,\;\chi_{1,\gamma}^jE_2^sE_1^l
        \mid i\in\om,\;j\in m,\;s\in\N,\;l\in m\}  \\
&=& \k\{(\zeta_{1,\gamma}+\xi\chi_{1,\gamma})^kE_2^sE_1^l
        \mid k\in 2m,\;s\in\N,\;l\in m\}
\subseteq D(m,d,\xi)^\circ.
\end{eqnarray*}
\end{itemize}
\end{proposition}

\begin{proof}
It is not hard to see that all $H^{\bullet}$ given above are Hopf subalgebras of the corresponding $H^{\circ}$ and the evaluation gives a Hopf pairing. Thus the remaining task is to show that the evaluation is non-degenerate. We only prove the non-degeneracy in (4) since the others can be proved similarly.

Clearly, the Hopf pairing $\langle-,-\rangle:D(m,d,\xi)^\bullet\otimes D(m,d,\xi)\rightarrow\k$ gives two Hopf algebra maps
$$\alpha:D(m,d,\xi)^\bullet\rightarrow D(m,d,\xi)^\circ\;\;\text{and}\;\;
  \beta:D(m,d,\xi)\rightarrow D(m,d,\xi)^\bullet{}^\circ,$$
and the map $\alpha$ is just an inclusion. Therefore, to show that $\langle-,-\rangle$  is non-degenerate, we only need to prove the injectivity of $\beta$. Recall that in Section \ref{section5} that
$D(m,d,\xi)=\k\{x^ig^jy^l,\;x^ig^ju_l\mid i\in\om,\;j\in\Z,\;l\in m\}$. Also,
\begin{eqnarray*}
\langle(\zeta_{1,\gamma}+\xi\chi_{1,\gamma})^{k'}E_2^{s'}E_1^{l'},
       x^ig^jy^l\rangle
&=& l!_\gamma \cdot \delta_{l,l'} \xi^{2jk'}(\frac{i}{\om}+\frac{j}{m})^{s'},  \\
\langle(\zeta_{1,\gamma}+\xi\chi_{1,\gamma})^{k'}E_2^{s'}E_1^{l'},
       x^ig^ju_l\rangle
&=& \frac{\xi^{l^2}}{(1-\gamma^{-1})^l} \cdot
    \delta_{l,l'} \xi^{(2j+1)k'}(\frac{i}{\om}+\frac{j}{m})^{s'}
\end{eqnarray*}
hold for any $k'\in 2m,\;s'\in\N,\;l'\in m$ and $i\in\om,\;j\in\Z,\;l\in m$. This means that we only need to show the linear independence of
$$\{\beta(x^ig^jy^l),\;\beta(x^ig^ju_l)\mid i\in\om,\;j\in\Z,\;l\in m\}
  \subseteq D(m,d,\xi)^\bullet{}^\circ.$$

Choose a lexicographically ordered set of elements for any positive integer $N$:
$\{h_{i,k,s,l}\in D(m,d,\xi)\mid (i,k,s,l)\in \om\times 2m\times \{-N,\cdots,N\}\times m\}$,
where
$$h_{i,k,s,l}:=\left\{\begin{array}{ll}
  \frac{1}{l!_\gamma}x^ig^{\frac{k}{2}+sm}y^l,
  & \text{when}\;\; 2\mid k; \\
  \frac{(1-\gamma^{-1})^l}{\xi^{l^2}}x^ig^{\frac{k-1}{2}+sm}u_l,
  & \text{when}\;\; 2\nmid k.
\end{array}\right.$$
The desired linear independence would be implied by the invertibility of the following $2\om m^2(2N+1)\times2\om m^2(2N+1)$ matrix
\begin{eqnarray*}
A: &=& (\langle\beta(h_{i,k,s,l}),
         (\zeta_{1,\gamma}+\xi\chi_{1,\gamma})^{k'}E_2^{i'+s'\om+N\om}E_1^{l'}\rangle  \\
   & &  \;\;\mid (i,k,s,l),(i',k',s',l')\in \om\times 2m\times \{-N,\cdots,N\}\times m)
\end{eqnarray*}
for any positive integer $N$. However,
\begin{eqnarray*}
A
&=& (\langle(\zeta_{1,\gamma}+\xi\chi_{1,\gamma})^{k'}E_2^{i'+s'\om+N\om}E_1^{l'},
            h_{i,k,s,l}\rangle  \\
&&  \;\;\mid (i,k,s,l),(i',k',s',l')\in \om\times 2m\times \{-N,\cdots,N\}\times m)  \\
&=& (\delta_{l,l'}\xi^{kk'}(\frac{i}{\om}+\lfloor\frac{k}{2}\rfloor\frac{1}{m}+s)^{i'+s'\om+N\om}
      \mid (i,k,s,l),(i',k',s',l')\in \om\times 2m\times \{-N,\cdots,N\}\times m)
\end{eqnarray*}
is similar to the matrix
\begin{eqnarray*}
&&  (\delta_{l,l'}\xi^{kk'}(\lfloor\frac{k}{2}\rfloor\frac{1}{m}+\frac{s\om+i}{\om})^{s'\om+i'+N\om}
      \mid (k,l,s,i),(k',l',s',i',)\in 2m\times m\times\{-N,\cdots,N\}\times\om)  \\
&=& (\delta_{l,l'}\xi^{kk'}(\lfloor\frac{k}{2}\rfloor\frac{1}{m}+\frac{t}{\om})^{t'+N\om}
      \mid (k,l,t),(k',l',t')\in 2m\times m\times\{-N\om,\cdots,(N+1)\om-1\}),
\end{eqnarray*}
which is invertible according to Lemma \ref{lem:Kron2}, because of the invertibility of
$$(\delta_{l,l'}\xi^{kk'}\mid (k,l),(k',l')\in 2m\times m)$$
and
$$((\lfloor\frac{k}{2}\rfloor\frac{1}{m}+\frac{t}{\om})^{t'+N\om}
    \mid t,t'\in \{-N\om,\cdots,(N+1)\om-1\})$$
for each $k\in 2m$.
\end{proof}

For a general infinite-dimensional Hopf algebra $H$, there may be \emph{no} $H^{\bullet}$ such that there is a non-degenerate Hopf pairing between $H$ and $H^{\bullet}$ (see \cite[Example 9.1.5(1)]{Mon93} for example).
In fact, the existence of non-degenerate Hopf pairings over a Hopf algebra $H$ is equivalent to require that $H$ is residually finite-dimensional. Recall that an algebra $H$ is said to be \textit{residually finite-dimensional}, if there is a family $\{\pi_i\}$ of finite-dimensional $\k$-representations of $H$ such that $\bigcap_i \Ker(\pi_i)=0$. See \cite[Definition 9.2.8]{Mon93} for example.

\begin{lemma}
Let $H$ be a Hopf algebra over an arbitrary field $\k$. Then there is a non-degenerate Hopf pairing $\langle-,-\rangle:H^\bullet\otimes H\rightarrow\k$, if and only if $H$ is residually finite-dimensional.
\end{lemma}

\begin{proof}
This is followed by \cite[Proposition 9.2.10]{Mon93}.
\end{proof}

However, our constructions for $H^\bullet$ in Proposition \ref{prop:bullet} satisfy further conditions, which might be crucial in a sense. Before that, we recall in \cite[Definition 15.4.4]{Rad12} that a non-cosemisimple pointed Hopf algebra $H$ is called \textit{minimal-pointed}, if each of its proper Hopf subalgebras must be cosemisimple.

\begin{proposition}\label{prop.6.4}
Suppose that $H$ is any of $\k\mathbb{D}_\infty$, $T_\infty(n,v,\xi)$, $B(\om,n,\gamma)$ and $D(m,d,\xi)$, and $H^\bullet$ denotes the Hopf algebra constructed in Proposition \ref{prop:bullet}. Then $H^\bullet$ is a minimal (under inclusion) Hopf subalgebra of $H^\circ$ such that the evaluation $\langle-,-\rangle:H^\bullet\otimes H\rightarrow\k$ is a non-degenerate Hopf pairing.
\end{proposition}

\begin{proof}
Since $H$ is infinite-dimensional, it is sufficient to show that $H^\bullet$ is the only infinite-dimensional Hopf subalgebra of itself.

(1) For the case when $H=\k\mathbb{D}_\infty$, we find that $(\k\mathbb{D}_\infty)^\bullet$ is a minimal-pointed Hopf algebra (see \cite[Proposition 15.4.6]{Rad12} and \cite[Proposition 2]{Rad99}) with the finite-dimensional coradical. Thus $(\k\mathbb{D}_\infty)^\bullet$ is the only infinite-dimensional Hopf subalgebra as desired.

(2) For the case when $H=T_\infty(n,v,\xi)$, note that $K:=\k\{\om^jE_1^l\mid j,l\in n\}$ is a finite-dimensional normal Hopf subalgebra of $H^\bullet=T_\infty(n,v,\xi)^\bullet$. The quotient Hopf algebra $H^\bullet/K^+H^\bullet=\k[\overline{E_2}]$ is the polynomial Hopf algebra, which is minimal-pointed too. Therefore, any infinite-dimensional Hopf subalgebra of $T_\infty(n,v,\xi)^\bullet$ must contain the element $E_2$. By the coproduct of $E_2$, it follows that $T_\infty(n,v,\xi)^\bullet$ is a minimal infinite-dimensional Hopf algebra.

(3) The case when $H=B(n,\om,\gamma)$ is completely similar to the case $T_\infty(n,v,\xi)$, due to the facts that $K:=\k\{\psi_{1,\gamma}^jE_1^l\mid j,l\in n\}$ is a finite-dimensional normal Hopf subalgebra, and the quotient $H^\bullet/K^+H^\bullet$ is the polynomial Hopf algebra.

(4) The case when $H=D(m,d,\xi)$ is similar, too. Note that the finite-dimensional normal Hopf subalgebra of $D(m,d,\xi)^\bullet$ is chosen as a generalized Taft algebra $\k\{(\zeta_{1,\gamma}+\xi\chi_{1,\gamma})^kE_1^l\mid k\in 2m,\;l\in m\}$, and the responding quotient is also the polynomial Hopf algebra.
\end{proof}

\subsection{Properties of $H^\bullet$ and Consequences}\label{subsection:HPs}

For possible future applications, we might consider Hopf pairings $H^\bullet\otimes H\rightarrow\k$ with additional requirements (including the non-degeneracy), which are listed as follows:
\begin{itemize}
\item[(HP1)] The pairing is non-degenerate;
\item[(HP2)] \;$H^\bullet$ and $H$ are both affine and noetherian;
\item[(HP3)] \;$\GKdim(H^\bullet)=\GKdim(H)$.
\end{itemize}
Furthermore, we might also wish $H^\bullet$ would have certain properties dual to some of $H$, for examples in this paper:
\begin{itemize}
\item[(HP4)] \;$H^\bullet$ is indecomposable as a coalgebra, when $H$ is prime as an algebra.
\end{itemize}

These considerations motivate us to discuss some properties of $H^\bullet$ in the following. Recall that a coalgebra is indecomposable if and only if it is \emph{link-indecomposable} (see \cite[Corollary 2.2]{Mon95}). Besides, a ring is called \emph{regular} if it has finite global dimension.

\begin{proposition}\label{prop:bulletproperty}
For affine prime regular Hopf algebras $H$ of GK-dimension one, consider Hopf algebras $H^\bullet$ constructed in Proposition \ref{prop:bullet}. We have
\begin{itemize}
\item[(1)]
All the Hopf algebras $H^\bullet$ have GK-dimension one.
\item[(2)]
All the Hopf algebras $H^\bullet$ are pointed and (link-)indecomposable as coalgebras.
\item[(3)]
All the Hopf algebras $H^\bullet$ are noetherian.
\item[(4)]
The Hopf algebra $(\k\mathbb{D}_\infty)^\bullet$ is regular while $T_\infty(n,v,\xi)^\bullet$, $B(n,\om,\gamma)^\bullet$ and $D(m,d,\xi)^\bullet$ are not when $n,m\geq 2$.
\end{itemize}
\end{proposition}

\begin{proof}
(1) and (2) are clear by our constructions for $H^\bullet$.

(3)
This is because every $H^\bullet$ constructed is a finitely generated module (on both sides) over its noetherian subalgebra $\k[E_2]$.

(4) The regularity of $(\k\mathbb{D}_\infty)^\bullet$ is a direct consequence of the Hilbert Theorem on Syzygies (see \cite[Theorem 8.36]{Rot09} e.g.). Recall that as an algebra, $(\k\mathbb{D}_\infty)^\bullet$ is generated by $\zeta_1-\chi_1$ and $E_2$ with relations
\begin{align*}
& (\zeta_1-\chi_1)^2=1,\;\;E_2(\zeta_1-\chi_1)=(\zeta_1-\chi_1)E_2.
\end{align*}
Clearly, $(\k\mathbb{D}_\infty)^\bullet$ is the polynomial algebra over the subalgebra
$$\k\{1,\zeta_1-\chi_1\}\cong\k\Z_2$$
with the indeterminate $E_2$ commuting with all the coefficients. However, $\k\Z_2$ is semisimple and hence regular. Thus $\gldim((\k\mathbb{D}_\infty)^\bullet)=\gldim(\k\Z_2)+1<\infty$, where $\gldim$ denotes the global dimension.

Next we show that $B(n,\om,\gamma)^\bullet$ is not regular when $n\geq2$ and the non-regularity for $T_\infty(n,v,\xi)^\bullet$ can be proved similarly.
Recall that as an algebra, $B(n,\om,\gamma)^\bullet$ is generated by $\psi_{1,\gamma}$, $E_2$ and $E_1$ with relations
$$\psi_{1,\gamma}^n=1,\;\;E_1^n=0,\;\;
  E_2\psi_{1,\gamma}=\psi_{1,\gamma}E_2,\;\;
  E_1\psi_{1,\gamma}=\gamma\psi_{1,\gamma}E_1,\;\;
  E_1E_2=E_2E_1+\frac{1}{n}E_1.$$
Note in Proposition \ref{prop:bullet}(3) that
\begin{eqnarray*}
B(n,\om,\gamma)^\bullet
&=& \k\{\psi_{1,\gamma}^jE_2^sE_1^l \mid j\in n,\;s\in\N,\;l\in n\}  \\
&=& \k\{E_2^sE_1^l\psi_{1,\gamma}^j \mid s\in\N,\;l\in n,\;j\in n\}.
\end{eqnarray*}
Since $B(n,\om,\gamma)^\bullet$ is a free $\k\langle E_1\rangle$-module by  Poincar\'{e}-Birkhoff-Witt theorem, if $B(n,\om,\gamma)^\bullet$ has finite global dimension, then the projective dimension $\mathrm{prdim}_{\k\langle E_1\rangle}(\k)<\infty$. However, this is not possible since $\k\langle E_1\rangle$ is a finite-dimensional commutative local, but not a field.

Finally, let's show the non-regularity of $D(m,d,\xi)^\bullet$. By Proposition \ref{prop:bullet}(4), we find that
\begin{eqnarray*}
D(m,d,\xi)^\bullet
&=& \k\{\zeta_{1,\gamma}^jE_2^sE_1^l,\;\chi_{1,\gamma}^jE_2^sE_1^l
        \mid j\in m,\;s\in\N,\;l\in m\}  \\
&=& \k\{\zeta_{1,\gamma}^jE_2^sE_1^l\mid j\in m,\;s\in\N,\;l\in m\}  \\
& &   \oplus\;\k\{\chi_{1,\gamma}^jE_2^sE_1^l\mid j\in m,\;s\in\N,\;l\in m\}
\end{eqnarray*}
is a direct sum of ideals. The former ideal is isomorphic to $B(m,\om,\gamma)^\bullet$. and thus it is not regular as a ring when $m\geq 2$. Therefore,
$$\gldim(D(m,d,\xi)^\bullet)\geq \gldim(B(m,\om,\gamma)^\bullet)=\infty.$$
\end{proof}

Thus concluding Propositions \ref{prop:bullet} and \ref{prop:bulletproperty}, we find that
\begin{corollary}
For affine prime regular Hopf algebras $H$ of GK-dimension one, the Hopf pairings $\langle-,-\rangle:H^\bullet\otimes H\rightarrow\k$ constructed in Proposition \ref{prop:bullet} satisfy the requirements (HP1) to (HP4).
\end{corollary}

\begin{remark}
Due to Proposition \ref{prop.6.4}, if one verifies that $H^\bullet$ is exactly the (link-)indecomposable component of $H^\circ$ containing the unit element, then $H^\bullet$ would be the unique Hopf subalgebra of $H^\circ$ such that (HP1) to (HP4) hold.
\end{remark}

\begin{remark}\label{rmk:quantumgrp}
\emph{Takeuchi \cite{Tak92} defined a quantum group $G$ to be a triple $(A,U,\langle-,-\rangle)$ where $\langle-,-\rangle:A\otimes U\to \k$ is a Hopf pairing. By Proposition \ref{prop:bullet}, for each affine prime regular Hopf algebra $H$ of GK-dimension one, one can get a quantum group in Takeuchi's sense naturally. }

\emph{A classical result by Larson and Radford \cite[Theorem 3.3]{LR88} states that a finite-dimensional Hopf algebra $H$ in characteristic $0$ is semisimple, if and only if $H^\ast$ is semisimple. Note that the regularity can be regarded as an infinite-dimensional analogue for the semisimplicity. However according to Proposition \ref{prop:bulletproperty}(3), it is not true that $H$ is regular if and only if $H^\bullet$ is regular for a non-degenerate Hopf pairing $H^\bullet\otimes H\rightarrow\k$ (or a quantum group). This might be a version negating the semisimplicity result by Larson and Radford in infinite-dimensional cases.}
\end{remark}

\begin{question}
\emph{For a general infinite-dimensional Hopf algebra $H$ which is residually finite-dimensional, natural questions are: When does a minimal Hopf algebra $H^{\bullet}$ forming a non-degenerate Hopf pairing over $H$ exist? When are such minimal Hopf algebras $H^\bullet$ unique and of the same GK-dimensions with $H$? }
\end{question}

%

\section*{Acknowledgements}

The authors would like to thank Doctor Ruipeng Zhu for discussions on additional techniques to estimate the regularity for Hopf algebras appeared, as well as the condition for the existence of non-degenerate Hopf pairings.

\end{document}